\documentclass[a4paper, 10pt]{article}

\usepackage{bbm,xspace}

\usepackage{mathtools} 
\mathtoolsset{showonlyrefs,showmanualtags}

\usepackage{amsmath}
\usepackage{amssymb}
\usepackage{amsthm}
\usepackage{mathabx}
\usepackage{color}
\usepackage{graphicx}
\usepackage{pgf,tikz}
\usepackage{tcolorbox}
\usepackage{enumerate}
\usepackage{enumitem}
\setlist[enumerate,1]{label=(\roman*)}
\setlist[enumerate,2]{label=(\alph*)}
\usepackage[margin=3.2cm]{geometry}
\usepackage{float}
\usepackage[format=plain,labelsep=period,justification=justified,font=small,labelfont=bf, margin=1em]{caption}
\usepackage{overpic}
\usepackage{contour}

\graphicspath{{pics/}}

\theoremstyle{plain}
\newtheorem{theorem}{Theorem}[section]

\newtheorem{lemma}[theorem]{Lemma}
\newtheorem{proposition}[theorem]{Proposition}

\theoremstyle{definition}
\newtheorem{definition}[theorem]{Definition}

\newtheorem{remark}[theorem]{Remark}

\makeatletter 
\newcommand\nobreakpar{\par\nobreak\@afterheading} 
\makeatother

\newcommand{\as}{\\[.6em]}

\newcommand{\bela}[1]{\begin{equation}\label{#1}}
\newcommand{\ela}{\end{equation}}
\newcommand{\bear}[1]{\begin{array}{#1}}
\newcommand{\ear}{\end{array}}

\renewcommand{\Xi}{\mathsf{H}}

\newcommand{\br}{\mbox{\boldmath $r$}}
\newcommand{\bR}{\mbox{\boldmath $R$}}
\newcommand{\plane}{\mbox{$\Box$}}

\newcommand{\Z}{\mathbbm{Z}}
\newcommand{\R}{\mathbbm{R}}
\newcommand{\C}{\mathbbm{C}}
\newcommand{\two}{\mathbbm{2}}

\renewcommand{\P}{\mathrm{P}}
\newcommand{\RP}{\R\P}
\newcommand{\CP}{\C\P}
\newcommand{\setbar}{\ | \ }
\newcommand{\set}[2]{\left\{#1 \setbar #2 \right\}}

\newcommand{\hide}[1]{}

\usepackage[hidelinks]{hyperref}  

\usepackage[backend=biber,
            style=alphabetic,
            giveninits=true,
            maxnames=99,   
            minnames=99,   
            maxbibnames=99,     
            maxalphanames=99,   
            minalphanames=99,   
            doi=true,
            url=true,
            isbn=false]{biblatex}
\numberwithin{equation}{section}

\title{A canonical discrete analogue of the classical circular cross sections of ellipsoids and their isometric deformation}

\author{B.\ Huang$^{1}$, W.K.\ Schief$^{1}$ and J.\ Techter$^{2}$
\bigskip\\  
$^{1}$School of Mathematics and Statistics, The University of New South Wales,\\ Sydney, NSW 2052, Australia
\bigskip\\
$^{2}$Institut f\"ur Mathematik, TU Berlin,\\
Str.\ des 17.\ Juni 136, 10623 Berlin, Germany}

\begin{document}

\maketitle

\begin{abstract}
  Based on a novel discretization procedure which has recently been proposed and applied in the construction of a canonical discrete analogue of confocal coordinate systems, an explicit method of constructing discrete analogues of ellipsoids is recorded.
  These discrete ellipsoids are entirely composed of planar quadrilaterals
  and come in pairs of combinatorially dual polyhedra,
  which together form princpal binets, a discrete counterpart of curvature line parametrizations.
  They exhibit a variety of algebraic and geometric properties which are classical in the case of their continuous counterparts.
  In a special case, it is demonstrated that the diagonals of the quadrilaterals of these discrete ellipsoids form two families of closed planar polygons.
  These polygons may be regarded as the discrete analogues of the classical circular cross sections of an ellipsoid. As in the continuous case, the discrete circles of any family lie in parallel planes and degenerate at the ``boundary'' of the discrete ellipsoids to ``umbilic vertices'' or ``umbilic edges''.
  Moreover, in analogy with the continuous theory, discrete ellipsoids admit a one-parameter family of deformations that preserve these discrete circles.
\end{abstract}

\newpage

\tableofcontents

\newpage

\begin{figure}[H]
  \centering
  \includegraphics[width=0.49\textwidth, trim={50pt 100pt 50pt 100pt}, clip]{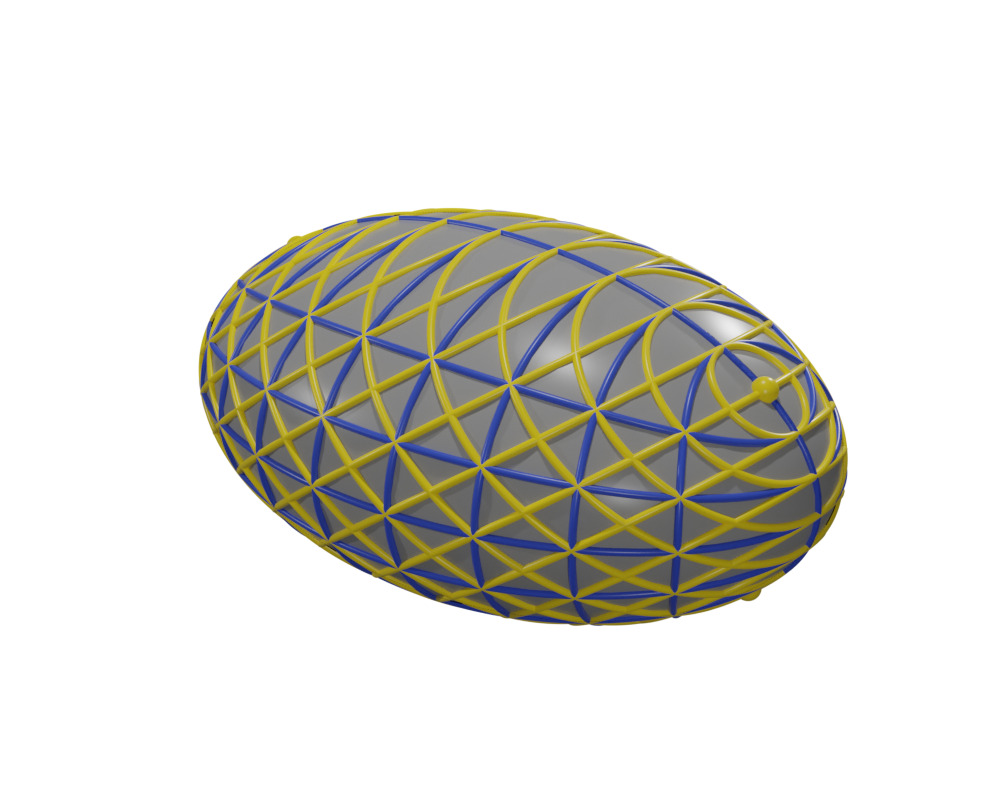}
  \includegraphics[width=0.49\textwidth, trim={50pt 100pt 50pt 100pt}, clip]{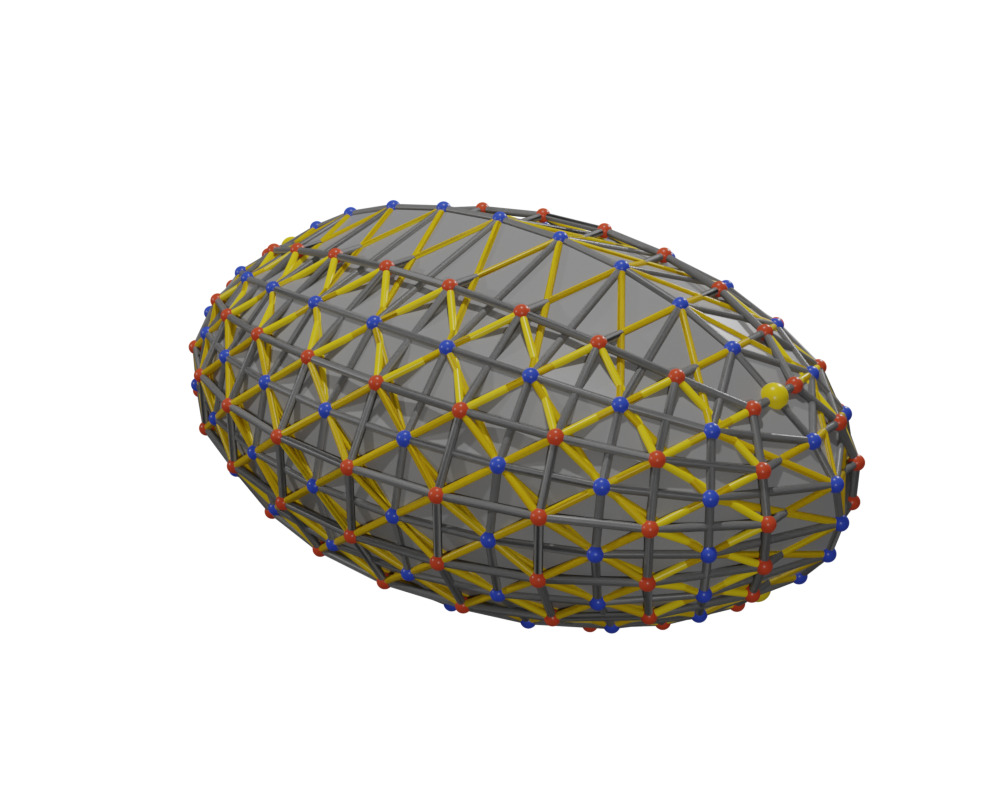}
  \caption{
    \emph{Left:}
    A closed sampling of a web of diagonally related nets of lines of curvature and circular cross sections on an ellipsoid.
    \emph{Right:}
    A discrete principal binet representing a discrete ellipsoid with discrete circular cross sections.
  }
  \label{fig:smooth-vs-discrete}
\end{figure}

\section{Introduction}
Since the publication of the seminal papers on the integrability-preserving discretization of the classical classes of surfaces of constant negative Gaussian curvature \cite{BobenkoPinkall96a} and isothermic surfaces by Bobenko and Pinkall \cite{BobenkoPinkall96b}, discrete differential geometry \cite{BobenkoSuris08} has rapidly developed into a sophisticated area of both pure and applied mathematics. In particular, canonical discretizations of classical coordinate systems on surfaces associated with asymptotic and conjugate nets and nets formed by lines of curvature have been used extensively to discretize (classical) classes of surfaces in Euclidean, affine, conformal and projective geometry (see \cite{BobenkoSuris08} and references therein). This development has also revealed the intimate connection between discrete geometries and the (B\"acklund) transformation theory \cite{RogersSchief02} which may be used to generate the analogous continuous counterparts.

In general, the local geometric and algebraic properties of these discrete coordinate systems and surfaces are well understood. However, the ``global'' properties are much harder to predict. For instance, classical quadrics constitute the simplest isothermic surfaces, but it has not been evident how their discrete counterparts may be generated based on the standard discretization of isothermic surfaces. In fact, this problem has recently led to an important alternative discretization of lines of curvature which has been used to define a canonical analogue of classical discrete confocal quadrics, thereby generating discrete analogues of quadrics of all signatures, including ellipsoids \cite{BobenkoSchiefSurisTechter16,BobenkoSchiefSurisTechter18}. Initially, in \cite{BobenkoSchiefSurisTechter16}, this discretization was based on (a discrete version \cite{KonopelchenkoSchief14} of) the Euler-Poisson-Darboux equation which encodes a special parametrization of the lines of curvature on quadrics. However, subsequently, it has been demonstrated \cite{BobenkoSchiefSurisTechter18} that any parametrization of the lines of curvature on confocal quadrics may be discretized geometrically with the algebraic properties of the continuous and discrete coordinate systems being remarkably alike.
Building on these ideas, this discretization of curvature line parametrizations has evolved into the broader theory of principal binets \cite{principal-binets}.

In \cite{AkopyanBobenkoSchiefTechter21}, a parametrization of the lines of curvature on (classical) ellipsoids has been introduced which is such that the coordinate lines of the associated diagonal coordinate system \cite{Blaschke28,Koch76,Koch77} constitute circles (see Figure~\ref{fig:smooth-vs-discrete}, left). These circles are the classical cross sections of generic ellipsoids \cite{HilbertCohnVossen52} which are obtained by slicing any ellipsoid by planes which are parallel to the tangent planes of the ellipsoid at the umbilic points. Since generic ellipsoids possess four umbilic points with pairs of associated tangent planes being parallel, one obtains two families of circles which form the coordinate lines of the above-mentioned diagonal coordinate system. It has also been shown in \cite{AkopyanBobenkoSchiefTechter21} that one may scale any one-parameter family of confocal ellipsoids in such a manner that corresponding circular cross sections of the members of this family are congruent. Hence, one may regard this one-parameter family of ellipsoids as being generated by the deformation of a single ellipsoid which preserves the circular cross sections. The existence of such a deformation was asserted in the classical book {\em Geometry and the Imagination} by Hilbert and Cohn-Vossen \cite{HilbertCohnVossen52} but, therein, no proof was given. 

The aim of this paper is to demonstrate that the novel discretization procedure which has led to the construction of discrete confocal quadrics preserves the remarkable geometric properties of continuous ellipsoids related to their circular cross sections (see Figure~\ref{fig:smooth-vs-discrete}, right). Thus, based on a discrete analogue of the above-mentioned privileged parametrization of the lines of curvature on ellipsoids, we show that the diagonals of the planar faces of a discrete ellipsoid form closed planar polygons which lie in two families of parallel planes. At the ``boundary'' of a discrete ellipsoid, these ``discrete circles'' degenerate to vertices or edges which may be regarded as discrete analogues of umbilic points. These ``umbilic vertices'' or the vertices of these ``umbilic edges'' are distinct in that their valence turns out to be different from four. We also introduce a scaling of any one-parameter family of discrete confocal ellipsoids which, once again, leads to congruent discrete circles amongst the members of this family. Accordingly, the deformation property of the circular cross sections of ellipsoids is, remarkably, also shown to be maintained by the discretization procedure.

The main result of this paper is Theorem~\ref{discretemaintheorem}, which summarizes the properties of the discrete ellipsoids and their discrete circular cross sections considered here. This is complemented by Theorem~\ref{thm:discrete-umbilics}, which establishes corresponding properties of the discrete umbilics of these discrete ellipsoids. Together, these results form a discrete analogue of the classical statements on circular cross sections of classical ellipsoids as summarized in Theorem~\ref{thm:circular-cross-sections}.
We then investigate boundary conditions for the discrete ellipsoids in Theorem~\ref{thm:boundary-conditions} and Theorem~\ref{thm:boundary-condition-quality}, and compare them to the closing condition for sampled webs on classical ellipsoids stated in Theorem~\ref{thm:closing-condition-web}.
Along the way, in Theorem~\ref{thm:discrete-ellipsoid-curvature-line-parametrization}, we also obtain a broader result worth mentioning:
a general principal binet representation which provides a discretization of arbitrary curvature line parametrizations of ellipsoids.

\subsection*{Acknowledgements}
This paper is partially based on the honours thesis by one of the authors (B.H.) \cite{huang-honours}.
W.K.S.\ was supported by the Australian Research Council (ARC Discovery Project DP200102118).
J.T.\ was supported by the Deutsche Forschungsgemeinschaft (DFG) Collaborative Research Center TRR 109 ``Discretization in Geometry and Dynamics''.
We would also like to thank Günther Rothe for an insightful discussion on affine transformations.

Most rendered images in this paper were generated using \texttt{pyddg} \cite{pyddg} and \texttt{Blender}.

\section{Confocal quadrics, curvature lines on quadrics, and diagonally related nets}

A classical system of confocal quadrics in Euclidean space $\R^3$ is a set of all quadrics which share their focal conics \cite{HilbertCohnVossen52, universe-of-quadrics}. In terms of Cartesian coordinates $\br=(x,y,z)$, these are algebraically defined by
\bela{E0}
  \frac{x^2}{\lambda+a} + \frac{y^2}{\lambda+b} + \frac{z^2}{\lambda+c} = 1,
\ela
where the distinct constants $a,b$ and $c$ determine the overall shape of the collection of confocal quadrics and the parameter $\lambda$ labels the quadrics. Any system of confocal quadrics is composed of three families of quadrics given by 
\bela{E1}
 \begin{split}
  \frac{x^2}{u_1 + a} + \frac{y^2}{u_1 + b} + \frac{z^2}{u_1 + c} & = 1\\
  \frac{x^2}{u_2 + a} + \frac{y^2}{u_2 + b} + \frac{z^2}{u_2 + c} & = 1\\
  \frac{x^2}{u_3 + a} + \frac{y^2}{u_3 + b} + \frac{z^2}{u_3 + c} & = 1\as
  -a < u_1 < -b < u_2 <- c &< u_3 ,
\end{split}
\ela
where it has been assumed without loss of generality that $a>b>c$. The parameters $u_i$ label the three families of quadrics of different type, namely two-sheeted hyperboloids, one-sheeted hyperboloids and ellipsoids respectively. The lines of intersection of the confocal quadrics are the coordinate lines of an orthogonal coordinate system $\br = \br(u_1,u_2,u_3)$ and, therefore \cite{Eisenhart60}, constitute lines of curvature on the quadrics $u_i=\mbox{const}$ parametrized by $u_j$ and $u_k$, where $i,j,k$ are distinct. Any triple of confocal quadrics of different signature meet orthogonally in a unique point $\br = (x,y,z)$ of any octant of $\R^3$ as indicated in Figure~\ref{confocal}.
\begin{figure}
  \centering
  \includegraphics[width=0.4\textwidth]{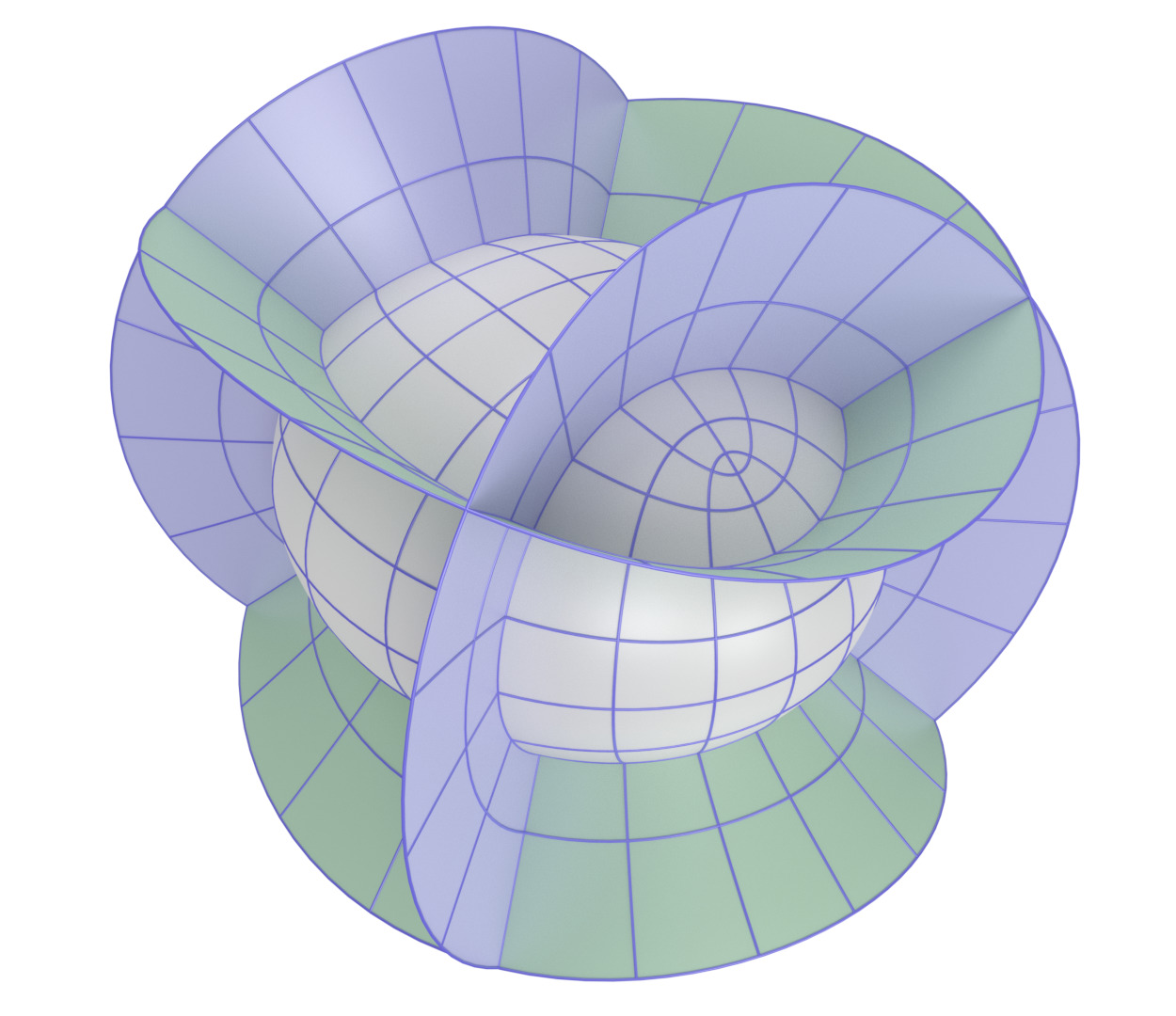}
  \caption{The intersection of three confocal quadrics of different signature.}
  \label{confocal}
\end{figure}
This fact is encoded in the explicit representation
\bela{E2}
 \begin{split}
  x(u_1, u_2, u_3)^2 &= \frac{(u_1+a)(u_2+a)(u_3+a)}{(a-b)(a-c)}\\ 
  y(u_1, u_2, u_3)^2 & = \frac{(u_1+b)(u_2+b)(u_3+b)}{(b-a)(b-c)}\\
  z(u_1, u_2, u_3)^2 & = \frac{(u_1+c)(u_2+c)(u_3+c)}{(c-a)(c-b)}  
 \end{split}
\ela
of the coordinate system. The representation above is also valid in the limiting cases $u_1=-a,-b$, $u_2=-b,-c$ and $u_3=-c$, leading to planar ``quadrics''. 

We introduce an arbitrary parametrization $\br(s_1,s_2,s_3)$ with $u_i = u_i(s_i)$ of the confocal coordinate lines,
leading to the novel representation of confocal coordinate systems set down in~\cite{BobenkoSchiefSurisTechter18}:
 \bela{E3}
 \begin{split}
   x(s_1, s_2, s_3) &= \frac{f_1(s_1)f_2(s_2)f_3(s_3)}{\sqrt{(a-b)(a-c)}}\\
   y(s_1, s_2, s_3) &= \frac{g_1(s_1)g_2(s_2)g_3(s_3)}{\sqrt{(a-b)(b-c)}}\\
   z(s_1, s_2, s_3) &= \frac{h_1(s_1)h_2(s_2)h_3(s_3)}{\sqrt{(a-c)(b-c)}},
 \end{split}
\ela
where the functions $f_i$, $g_i$ and $h_i$ are related by the functional equations
\bela{E4}
 \begin{aligned}
  f_1^2(s_1) + g_1^2(s_1) & = a-b,\quad & f_1^2(s_1) +h_1^2(s_1) & = a-c\\
  f_2^2(s_2) - g_2^2(s_2) & = a-b,\quad & f_2^2(s_2) +h_2^2(s_2) & = a-c\\
  f_3^2(s_3) - g_3^2(s_3) & = a-b,\quad & f_3^2(s_3) -h_3^2(s_3) & = a-c
 \end{aligned}
\ela
and the parameters $u_i$ are given by
\bela{E5}
\begin{aligned}
  u_1(s_1) &= f_1^2(s_1) - a = - g_1^2(s_1) - b = - h_1^2(s_1) - c\\
  u_2(s_2) &= f_2^2(s_2) - a = \phantom{-} g_2^2(s_2) - b   = - h_2^2(s_2) - c\\
  u_3(s_3) &= f_3^2(s_3) - a = \phantom{-} g_3^2(s_3) - b   = \phantom{-} h_3^2(s_3) - c.
\end{aligned}
\ela
Thus, any system of confocal quadrics is encoded in the system \eqref{E4} via the representation \eqref{E3}. It may easily be seen that \eqref{E2} is indeed satisfied modulo \eqref{E4}.

We now introduce the key notion of \emph{diagonally related nets}.
\begin{definition}\
  \label{mutual}
  \nobreakpar
  \begin{enumerate}
  \item
    A \emph{net} on a surface $\Sigma$ is a pair of one-parameter families of curves foliating $\Sigma$
    such that there exists exactly one curve of each family passing through any point on $\Sigma$.
  \item
    Let $\mathcal{N}_1$ and $\mathcal{N}_2$ be two nets on a surface $\Sigma$.
    Then, $\mathcal{N}_2$ is termed {\em diagonal to} $\mathcal{N}_1$
    if the existence of a curve of $\mathcal{N}_2$ through a pair of opposite vertices of any ``quadrilateral'' formed by four curves of $\mathcal{N}_1$
    implies that the remaining pair of opposite vertices is also connected by a curve of $\mathcal{N}_2$ (cf.\ Figure \ref{diagonal}).
  \end{enumerate}
\end{definition}
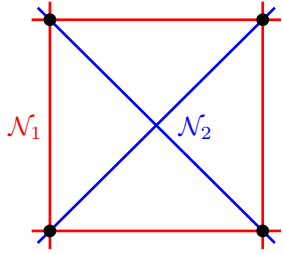
\begin{figure}
  \centering
  \begin{tikzpicture}[line cap=line join=round,>=stealth,x=1.0cm,y=1.0cm, scale=0.8]
    \draw [line width=1pt, color=red,  domain=-0.3:3.8] plot(0.,\x);
    \draw [line width=1pt, color=red,  domain=-0.3:3.8] plot(3.5,\x);
    \draw [line width=1pt, color=red,  domain=-0.3:3.8] plot(\x,0.);
    \draw [line width=1pt, color=red,  domain=-0.3:3.8] plot(\x,3.5);
    \draw [line width=1pt, color=blue,  domain=-0.2:3.7] plot(\x,\x);
    \draw [line width=1pt, color=blue,  domain=-0.2:3.7] plot(\x,3.5 - \x);
    \coordinate (p1) at (0.0,3.5);
    \fill (p1) circle (3pt);
    \coordinate (p2) at (0.0,0.0);
    \fill (p2) circle (3pt);
    \coordinate (p3) at (3.5,3.5);
    \fill (p3) circle (3pt);
    \coordinate (p4) at (3.5,0.0);
    \fill (p4) circle (3pt);
    \draw [color=red] (-0.4, 1.7) node {$\mathcal{N}_1$}; 
    \draw [color=blue] (2.4, 1.7) node {$\mathcal{N}_2$};      
  \end{tikzpicture}
  \caption{Quadrilateral of two diagonally related nets $\mathcal{N}_1$ and $\mathcal{N}_2$.}
  \label{diagonal}
\end{figure}

The above relation is symmetric \cite{Blaschke28},
that is, $\mathcal{N}_2$ being diagonal to $\mathcal{N}_1$ is equivalent to $\mathcal{N}_1$ being diagonal to $\mathcal{N}_2$.
Here, we make the genericity assumption that any curve of one family of a net intersects any curve of the other family of the net.

For a parametrization $(u_1,u_2) \mapsto \br(u_1,u_2)$ of a surface $\Sigma$,
the two families of coordinate lines $u_1=\mbox{const}$ and $u_2=\mbox{const}$ constitute a net $\mathcal{N}_1$ on $\Sigma$.
In this case, all nets $\mathcal{N}_2$ diagonal to $\mathcal{N}_1$ are given by reparametrizations $(s_+,s_-) \mapsto \br(s_+,s_-)$,
which are compositions of a reparametrization along the coordinate lines $u_1 = u_1(s_1)$, $u_2 = u_2(s_2)$ and $s_\pm = \frac{s_1 \pm s_2}{2}$.
Applied to the system of confocal coordinates, we obtain the following.

\begin{theorem}
  \label{thm:diagonal-curvature-lines}
  Let $i,j,k \in \{1,2,3\}$ be distinct and $s_i = \rm{const}$.
  Consider the parametrization
  \[
    (s_j,s_k) \mapsto \br(s_1,s_2,s_3)
  \]
  given by \eqref{E3} with \eqref{E4}
  of the quadric $\mathcal{Q}$ defined by $s_i = \rm{const}$.
  \begin{enumerate}
  \item
    The net of curvature lines on $\mathcal{Q}$ is given by the curves $s_j = \mbox{const}$ and $s_k = \mbox{const}$.
  \item
    A net given by the curves
    \[
      s_j \pm s_k = \mbox{const}
    \]
    is diagonal to the net of curvature lines on $\mathcal{Q}$.
    Conversely, every net diagonal to the net of curvature lines on $\mathcal{Q}$ is described in this way.
  \end{enumerate}
\end{theorem}

\section{Circular cross sections of ellipsoids and their isometric deformation}

The main result of the current paper is to demonstrate that the following key properties of ellipsoids admit canonical discrete analogues.
Any generic ellipsoid possesses four umbilic points, that is, four points at which the principal curvatures coincide so that the normal curvature at these points is independent of the direction of the normal plane. Pairs of ``opposite'' tangent planes at the four umbilic points are parallel. The cross section with the ellipsoid of any plane which is parallel to any of these tangent planes constitutes a circle (see Figure~\ref{fig:circular-cross-sections-and-curvature-lines}, left).
\begin{figure}
  \centering
  \includegraphics[width=0.49\textwidth, trim={50pt 100pt 50pt 100pt}, clip]{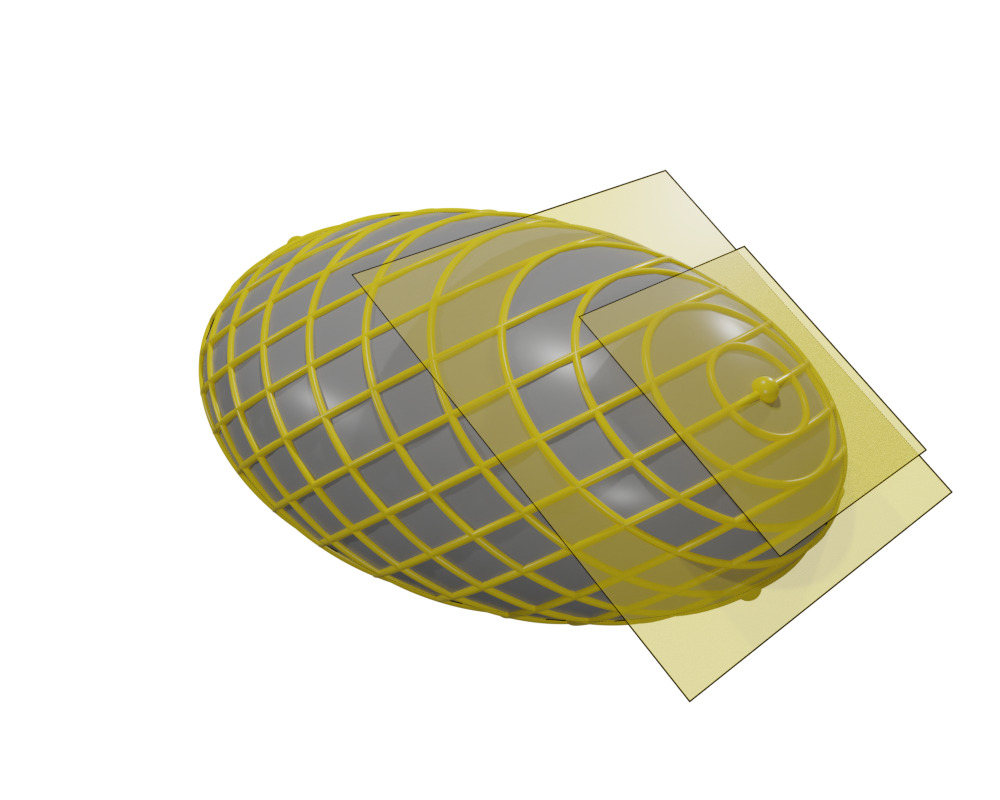}
  \includegraphics[width=0.49\textwidth, trim={50pt 100pt 50pt 100pt}, clip]{ellipsoid_3}
  \caption{
    \emph{Left:}
    The circular cross sections of an ellipsoid lie in the two families of planes
    which are parallel to the tangent planes at the two pairs of opposite umbilic points.
    \emph{Right:}
    The net of curvature lines and the net of circular cross sections of an ellipsoid are diagonally related.
  }
  \label{fig:circular-cross-sections-and-curvature-lines}
\end{figure}

In \cite{AkopyanBobenkoSchiefTechter21} it has been observed, that the net
composed of the two families of circular cross sections of an ellipsoid
is diagonally related to the net consisting of its lines of curvature (see Figure~\ref{fig:circular-cross-sections-and-curvature-lines}, right).
In \cite{HilbertCohnVossen52}, it has been stated without proof that any elliposid may be deformed
into a one-parameter family of ellipsoids in such a manner that its two one-parameter families of circular cross sections are preserved
and the deformed circles are congruent to the original circles (see Figure~\ref{fig:ellipsoid-family}).
In particular, the deformation is isometric along the circles, that is, arc length along any circle between any two points on the circle is preserved by the deformation.
Thus, one may construct a deformable model of an ellipsoid by realising a sample of circular cross sections as a collection of ``rings''
which are fastened at the points of intersection (see Figure \ref{circular_wien}).
In \cite{AkopyanBobenkoSchiefTechter21} this deformation is constructed by scaling the one-parameter family of confocal ellipsoids
in such a manner that all the ellipsoids share the second semi-axis.
\begin{figure}
  \centering
  \includegraphics[width=0.24\textwidth, trim={50pt 0 50pt 0}, clip]{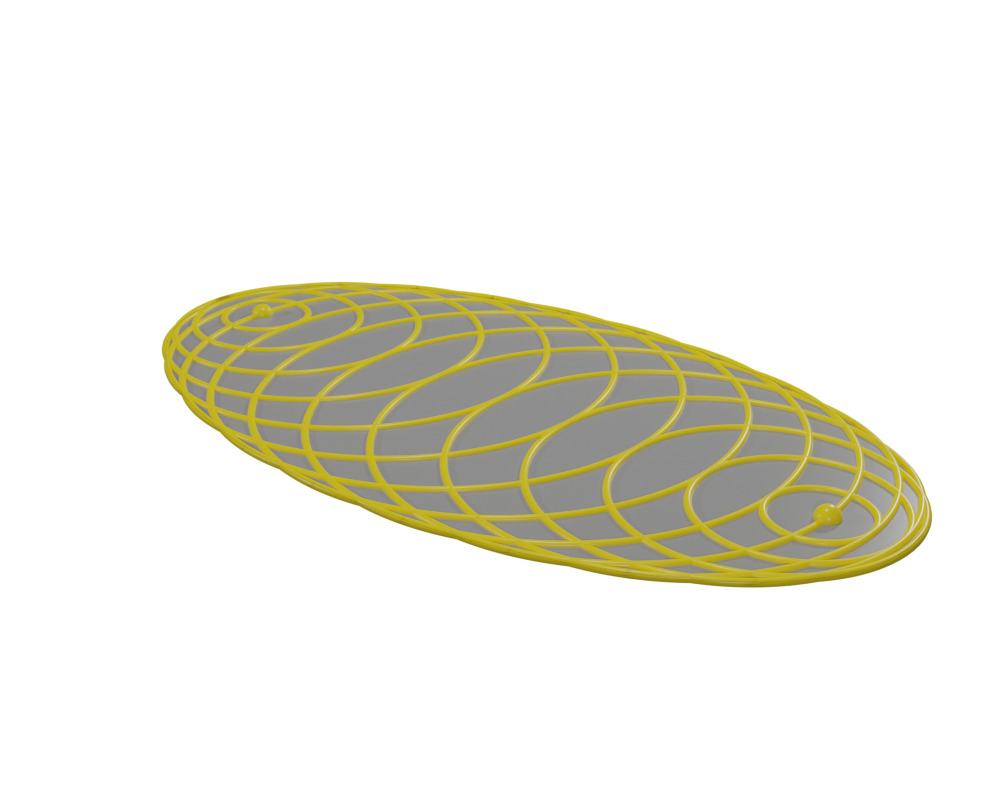}
  \includegraphics[width=0.24\textwidth, trim={50pt 0 50pt 0}, clip]{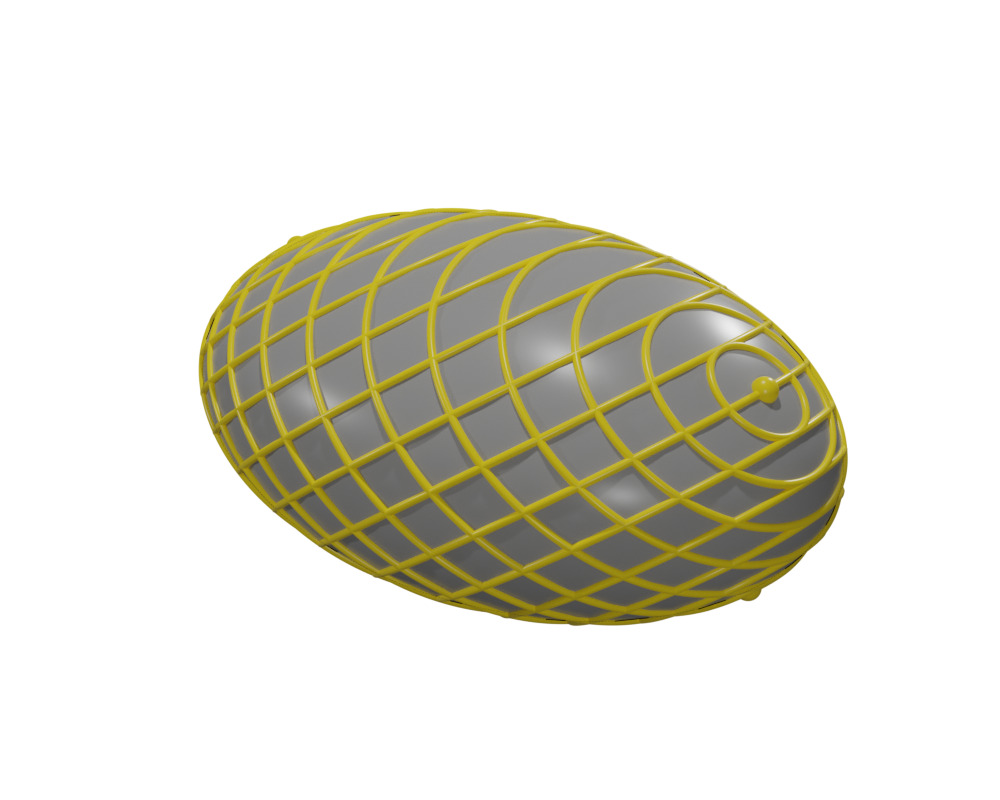}
  \includegraphics[width=0.24\textwidth, trim={50pt 0 50pt 0}, clip]{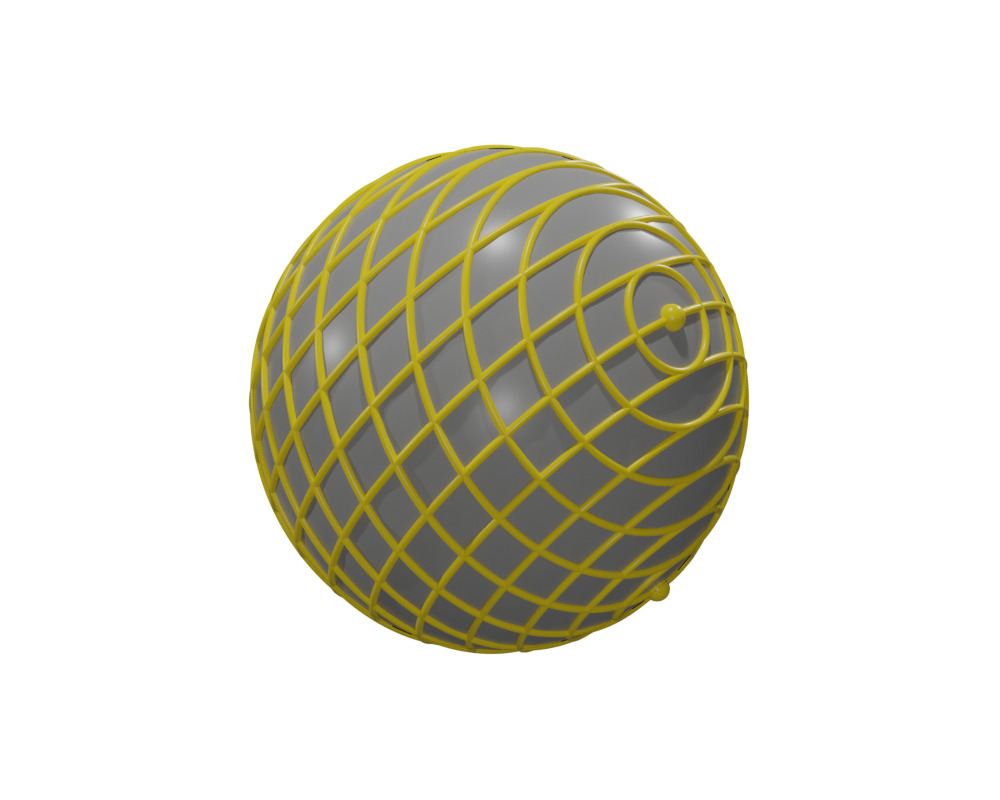}
  \includegraphics[width=0.24\textwidth, trim={50pt 0 50pt 0}, clip]{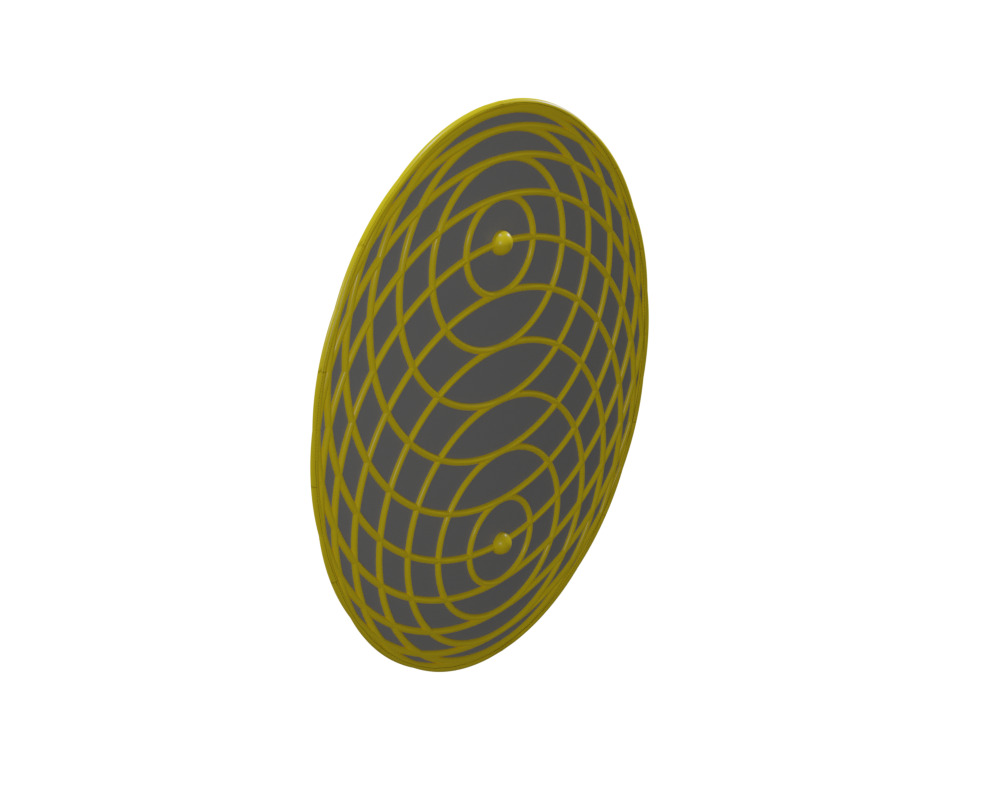}
  \caption{A one-parameter family of deformations of an ellipsoid which is isometric along the circular cross sections.}
  \label{fig:ellipsoid-family}
\end{figure}
\begin{figure}
  \centering
  \includegraphics[width=0.32\textwidth]{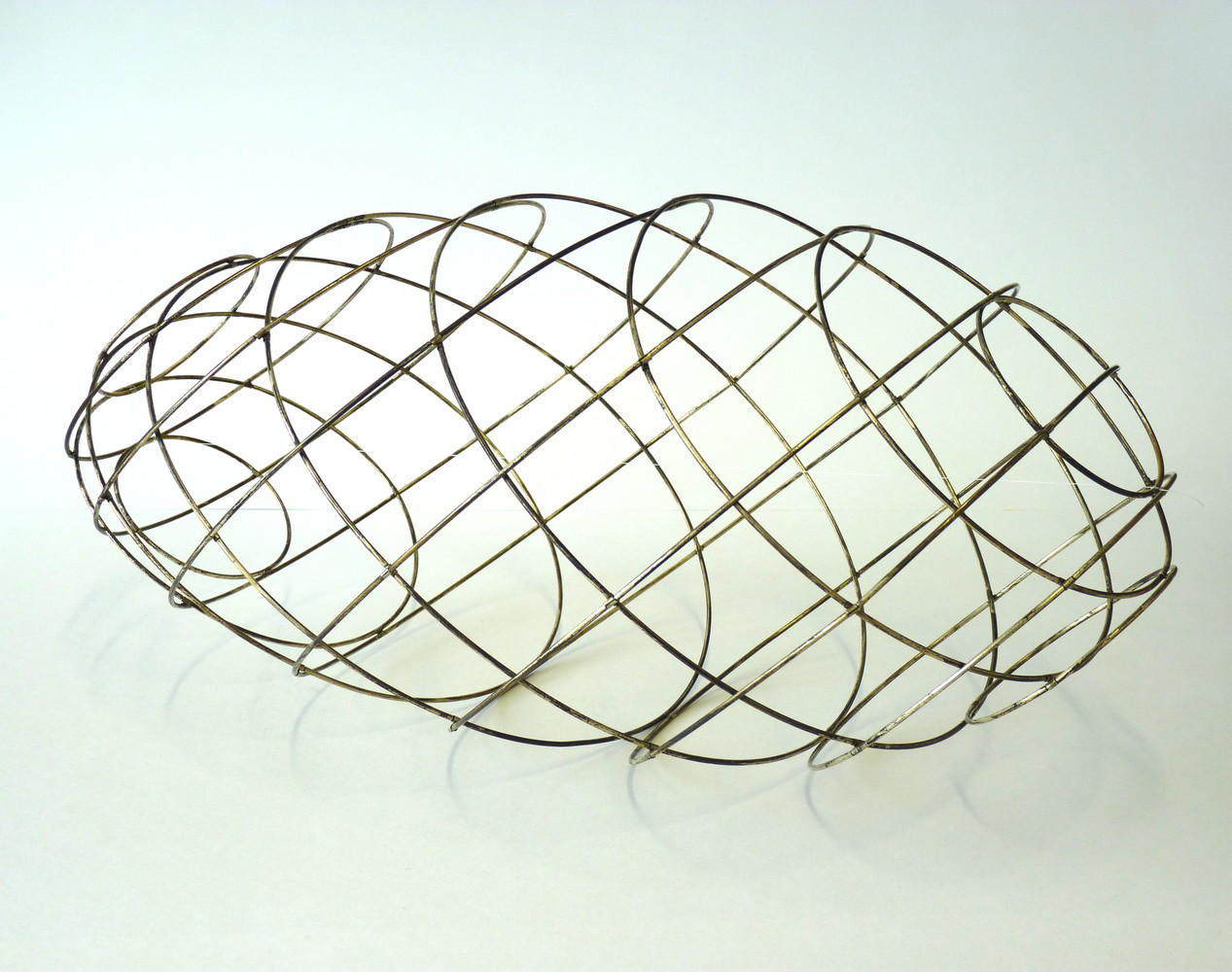}
  \includegraphics[width=0.32\textwidth]{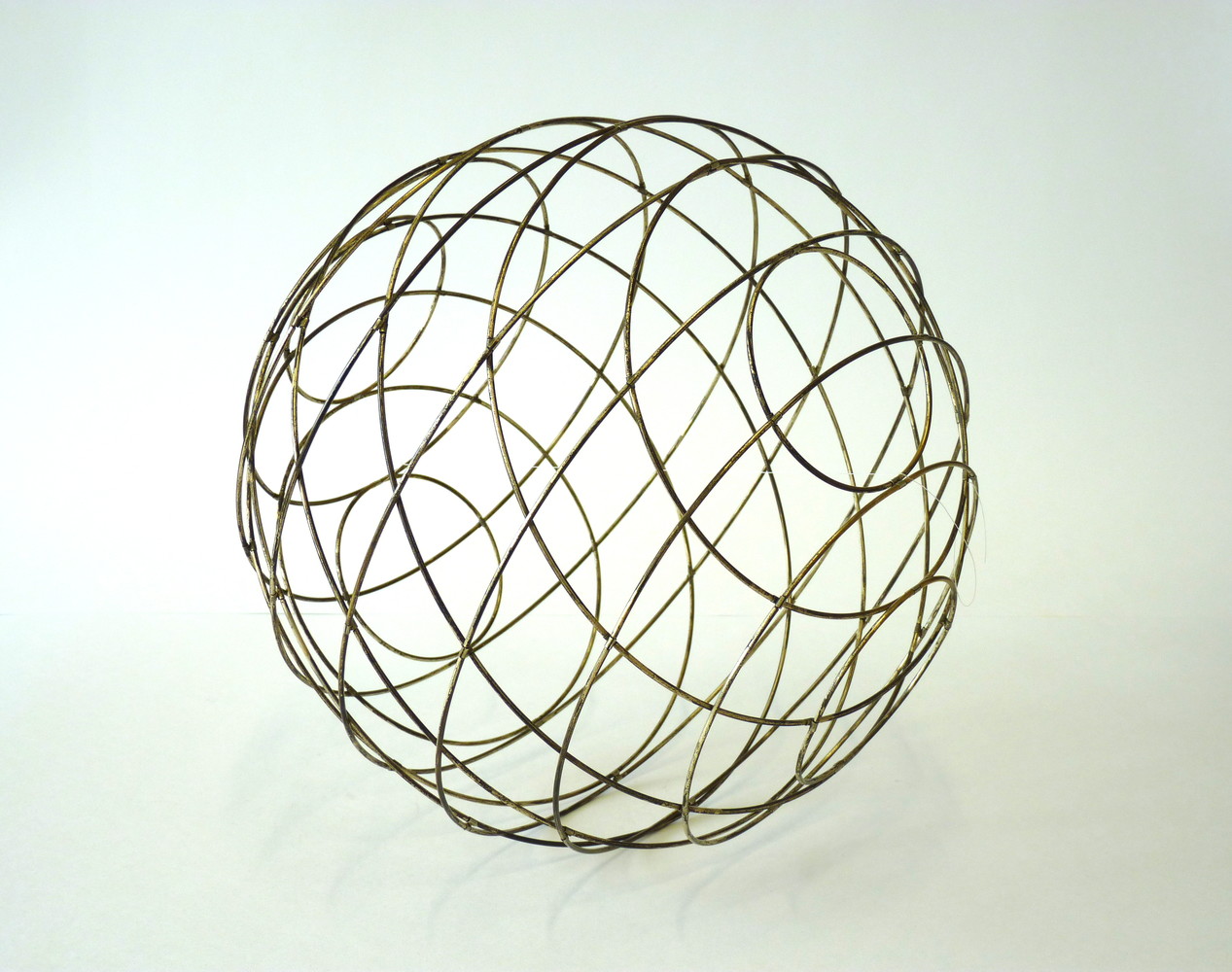}
  \includegraphics[width=0.32\textwidth]{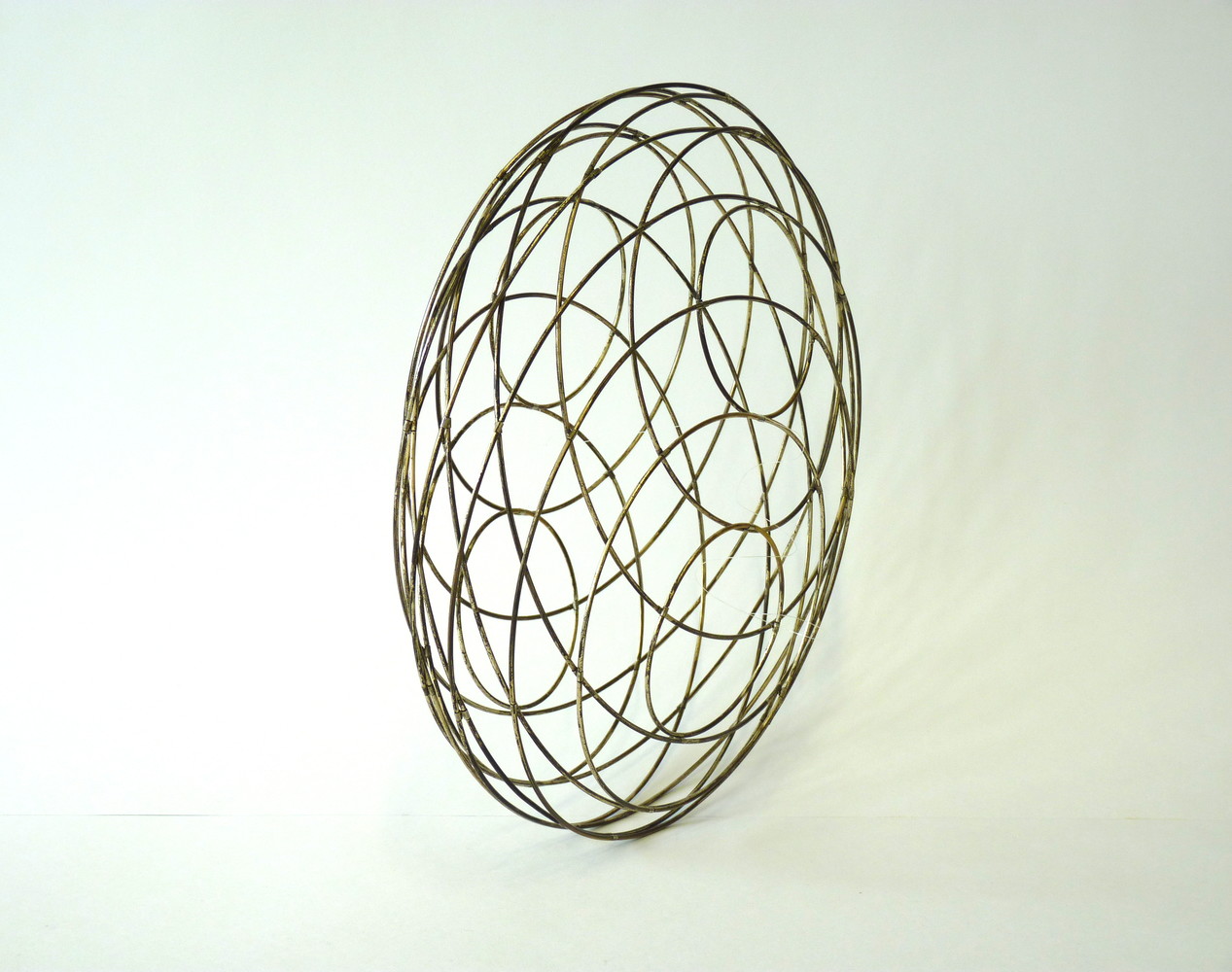}
  \caption{Photos of an ``isometrically'' deformable model of an ellipsoid from the collection of the Institute of Discrete Mathematics and Geometry, TU Wien.}
  \label{circular_wien}
\end{figure}
\begin{remark}
  A similar deformation also exists for one-sheeted hyperboloids,
  which in this case is isometric along the two one-parameter families of
  generators \cite{HilbertCohnVossen52, AkopyanBobenkoSchiefTechter21}.
  Here, the one-parameter family of deformations is given by the
  (unscaled) family of confocal hyperboloids.
\end{remark}

We collect the properties of ellipsoids related to its circular cross sections in the following theorem,
and later establish analogous properties for the discretization presented in Theorem~\ref{discretemaintheorem}.
\begin{figure}
  \centering
  \includegraphics[width=0.5\textwidth]{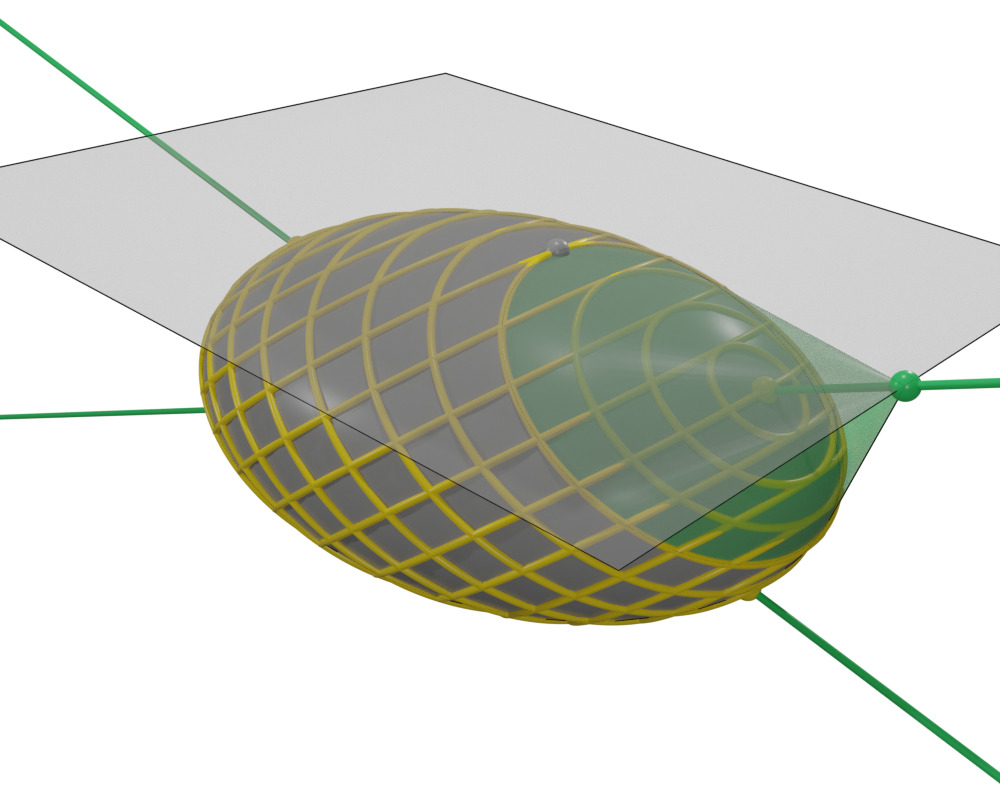}
  \caption{
    The tangent planes to an ellipsoid along a circular cross section meet at a point,
    which is the pole of the plane of the circular cross section.
    Equivalently, it is the apex of the touching cone enveloped by the tangent planes along the circular cross section.
    Furthermore, this point lies on one of the two lines passing through opposite umbilic points.
  }
  \label{fig:ellipsoid-tangent-cone}
\end{figure}
\begin{theorem}
  \label{thm:circular-cross-sections}
  Ellipsoids exhibit the following properties.
  \begin{enumerate}
  \item
    \label{thm:circular-cross-sections1}
    Any ellipsoid may be sliced into two one-parameter families of circles
    by taking the intersection with planes which are parallel to the tangent planes at the ellipsoid's umbilic points (see Figure~\ref{fig:circular-cross-sections-and-curvature-lines}, left).
  \item
    \label{thm:circular-cross-sections2}
    The tangent planes to an ellipsoid along any circular cross section meet at a point which lies on one of the two lines passing through opposite umbilic points (see Figure~\ref{fig:ellipsoid-tangent-cone}).
  \item
    \label{thm:circular-cross-sections3}
    The net composed of the two families of circular cross sections of an ellipsoid
    and the net consisting of its lines of curvature are mutually diagonal (see Figure~\ref{fig:circular-cross-sections-and-curvature-lines}, right).
  \item
    \label{thm:circular-cross-sections4}
    The two families of circular cross sections of an ellipsoid may be isometrically deformed in such a manner
    that the deformed circles are congruent to the original circles and remain the circular cross sections of ellipsoids (see Figure~\ref{fig:ellipsoid-family}).
    
    Furthermore, any two ellipsoids from this family of deformations are confocal up to a scaling.
  \end{enumerate}
\end{theorem}

Each part of the theorem will be shown to be an implication of a more technical proposition.
Thus, the proof of this theorem is distributed throughout the remaining chapter.
While most of the claims are well known, we still provide proofs for the convenience of the reader
and in some cases because they constitute a suitable preparation for the discretization to be established.
Indeed, subsequently we give a natural discrete analogue for each of these propositions.

Let $\alpha > \beta > \gamma > 0$ and $\mathcal{E} \subset \R^3$ be the ellipsoid
\begin{equation}
  \label{eq:ellipsoid}
  \mathcal{E} : \quad \frac{x^2}{\alpha} + \frac{y^2}{\beta} + \frac{z^2}{\gamma} = 1.
\end{equation}
To prove part \ref{thm:circular-cross-sections1} of Theorem~\ref{thm:circular-cross-sections},
we give explicit equations for the circular cross sections of the ellipsoid $\mathcal{E}$.
For the convenience of the reader we provide a projective proof of the following classical statement in Appendix~\ref{sec:appendix-proof}.
\begin{proposition}
  \label{prop:ellipsoid-circular-sections}
  Let $\alpha > \beta > \gamma > 0$ and $\mathcal{E} \subset \R^3$ be the ellipsoid \eqref{eq:ellipsoid}.
  Then the two one-parameter families of parallel planes
  \begin{equation}
    \label{eq:ellipsoid-circular-sections}
    \Pi_\pm(\mu_\pm) : ~ \sqrt{\tfrac{1}{\beta} - \tfrac{1}{\alpha}} \, x \pm \sqrt{\tfrac{1}{\gamma} - \tfrac{1}{\beta}} \, z = \mu_\pm,\qquad
    \mu_\pm \in \left[ - \sqrt{\tfrac{\alpha-\gamma}{\beta}}, \sqrt{\tfrac{\alpha-\gamma}{\beta}} \right]
  \end{equation}
  are exactly the planes that intersect the ellipsoid $\mathcal{E}$ in circles (see Figure~\ref{fig:circular-cross-sections-and-curvature-lines}, left).
  For the values $\mu_\pm = \pm \sqrt{\frac{\alpha-\gamma}{\beta}}$, the circles degenerate to (umbilic) points,
  and the planes become the tangent planes at these points.
\end{proposition}
\begin{remark}
  For a non-spherical ellipsoid, the families of circular cross sections given by \eqref{eq:ellipsoid-circular-sections}
  are its only circular cross sections.
  In the limiting case in which the ellipsoid becomes an ellipsoid of revolution,
  the two families of circular cross sections coincide.  
\end{remark}

To prove part \ref{thm:circular-cross-sections2} of Theorem~\ref{thm:circular-cross-sections},
recall that the polar plane with respect to an ellipsoid of a point on that ellipsoid is the tangent plane at that point.
Furthermore, the tangent planes along any planar section intersect in a point, which is the pole of the plane of that planar section.
Thus, part \ref{thm:circular-cross-sections2} is a corollary of the following proposition.
\begin{proposition}
  \label{prop:ellipsoid-polar-points}
  Let $\alpha > \beta > \gamma > 0$ and $\mathcal{E} \subset \R^3$ be the ellipsoid \eqref{eq:ellipsoid}.
  Then, for each of the two families of planes $\Pi_\pm$, given by \eqref{eq:ellipsoid-circular-sections},
  the poles of the planes with respect to the ellipsoid $\mathcal{E}$ lie on a line
  passing through opposite umbilic points (see Figure~\ref{fig:ellipsoid-tangent-cone}).
\end{proposition}

\begin{proof}
  Projectively, a family of parallel planes intersect in a line at infinity.
  Thus, the family of poles of these planes lie on a line as well, namely the polar line.
  Since each of the two families of planes $\Pi_\pm$ contains two tangent planes, the poles of which are opposite umbilic points,
  the polar line is the line through this pair of opposite umbilic points.
\end{proof}
\begin{remark}
  The planes in $\Pi_\pm$ are exactly the planes from the corresponding parallel family that have non-empty intersection with $\mathcal{E}$.
  Thus, the poles are exactly the points on the line that are outside the ellipsoid.
\end{remark}

To prove part~\ref{thm:circular-cross-sections3} of Theorem~\ref{thm:circular-cross-sections},
we need a suitable curvature line parametrization of the ellipsoid $\mathcal{E}$.
To this end, we may temporarily identify $\mathcal{E}$
with the ellipsoid of a confocal family \eqref{E0} with $\lambda = 0$.
Hence, on replacing $a, b, c$ by $\alpha, \beta, \gamma$ in \eqref{E3} and setting $u_3(s_3) = 0$ 
we obtain a curvature line parametrization of $\mathcal{E}$, namely
\begin{equation}
  \label{eq:ellipsoid-curvature-line-parametrization}
  \begin{split}
    x(s_1,s_2) &= \sqrt{\frac{\alpha}{(\alpha - \beta)(\alpha-\gamma)}} \, f_1(s_1)f_2(s_2) \\
    y(s_1,s_2) &= \sqrt{\frac{\beta}{(\alpha - \beta)(\beta-\gamma)}} \, g_1(s_1)g_2(s_2) \\
    z(s_1,s_2) &= \sqrt{\frac{\gamma}{(\alpha - \gamma)(\beta-\gamma)}} \, h_1(s_1)h_2(s_2)
  \end{split}
\end{equation}
with
\begin{equation}
  \label{eq:ellipsoid-curvature-line-parametrization-f}
  \begin{aligned}
    f_1(s_1)^2 + g_1(s_1)^2 &= \alpha - \beta, & \quad f_1(s_1)^2 + h_1(s_1)^2 &= \alpha - \gamma\\
    f_2(s_2)^2 - g_2(s_2)^2 &= \alpha - \beta, & \quad f_2(s_2)^2 + h_2(s_2)^2 &= \alpha - \gamma.
  \end{aligned}
\end{equation}
Now part~\ref{thm:circular-cross-sections3} is a corollary of the following proposition,
which is based on a suitable choice of the functions $f_i$, $g_i$, $h_i$ in the above functional equations.

The derivation of a discrete analogue of the circular cross sections of ellipsoids is based on the discretization of this proposition.
We present a version of the proof recorded in \cite{AkopyanBobenkoSchiefTechter21}
which is custom-made for the translation into the discrete language.
The proof is constructive in that we will derive an explicit representation of the circular cross sections of individual ellipsoids.
In Proposition~\ref{prop:isometric-circles} and Theorem~\ref{thm:isometric-circles},
we embed this representation into a one-parameter family of ``isometric deformations''.

\begin{proposition}
  \label{prop:ellipsoid-diagonal-relation}
  Let $\alpha > \beta > \gamma > 0$ and $\mathcal{E} \subset \R^3$ be the ellipsoid \eqref{eq:ellipsoid}.
  Then, the two parametrizations
  \begin{equation}
    \label{eq:ellipsoid-diagonal-relation}
    \begin{split}
    \br &= (x,y,z) : [-s_1^0,s_1^0] \times [-s_2^0,s_2^0] \rightarrow \R^3\\[0.5em]
      x(s_1,s_2) &= \sqrt{\frac{\alpha(\alpha-\gamma)}{\alpha - \beta}} \, \sin s_1 \cos s_2 \\
      y(s_1,s_2) &= \pm \sqrt{\frac{\beta}{\beta-\gamma}} \, \sqrt{1 - \frac{\alpha-\gamma}{\alpha-\beta}\sin^2s_1} \, \sqrt{\frac{\alpha-\gamma}{\alpha-\beta}\cos^2s_2 - 1} \\
      z(s_1,s_2) &= \sqrt{\frac{\gamma(\alpha-\gamma)}{\beta-\gamma}} \, \cos s_1 \sin s_2
    \end{split}
  \end{equation}
  with
  \[
    s_1^0 = \sin^{-1}\sqrt{\frac{\alpha-\beta}{\alpha-\gamma}}, \qquad s_2^0 = \cos^{-1}\sqrt{\frac{\alpha-\beta}{\alpha-\gamma}}
  \]
  are curvature line parametrizations of the two halves $\mathcal{E} \cap \{ y \geq 0 \}$ and $\mathcal{E} \cap \{ y \leq 0 \}$
  of the ellipsoid,
  such that the curves
  \[
    s_1 \pm s_2 = \mbox{const}
  \]
  are the circular cross sections (see Figure~\ref{fig:circular-cross-sections-and-curvature-lines}, right).
\end{proposition}

\begin{proof}
  We show that there exist solutions $f_1, f_2, g_1, g_2, h_1, h_2$
  of the functional equations~\eqref{eq:ellipsoid-curvature-line-parametrization-f}
  such that the diagonal net given by the curves
  \[
    s_\pm = \frac{s_1 \pm s_2}{2} = \text{const}
  \]
  are the circular cross sections of the ellipsoid.
  By substituting the parametrization \eqref{eq:ellipsoid-curvature-line-parametrization} into the planes $\Pi_\pm(\mu_\pm)$ given by \eqref{eq:ellipsoid-circular-sections}, we obtain
  \begin{equation}
    \label{eq:ellipsoid-parametrization-into-planes}
    f_1(s_1)f_2(s_2) \pm h_1(s_1)h_2(s_2) = \sqrt{\beta(\alpha-\gamma)}\,\mu_\pm.
  \end{equation}
  For the diagonal lines $s_\pm = \text{const}$ to lie in these planes,
  the right-hand side of equation \eqref{eq:ellipsoid-parametrization-into-planes} must be a function only depending on $s_+$ or $s_-$, respectively.
  Thus, the functions $f_1, f_2, h_1, h_2$ are subject to the four equations
  \begin{equation}
    \label{eq:diagonal-curvature-circular}
    \begin{aligned}
      f_1^2(s_1) + h_1^2(s_1) &= \alpha - \gamma\\
      f_2^2(s_2) + h_2^2(s_2) &= \alpha - \gamma\\
      f_1(s_1)f_2(s_2) \pm h_1(s_1)h_2(s_2) &= \tilde\mu_\pm\left(\tfrac{s_1 \pm s_2}{2}\right),
    \end{aligned}
  \end{equation}
  where $\tilde\mu_\pm$ are some functions of their indicated argument only.
  Due to the trigonometric addition theorems
  \[
    \sin s_1 \cos s_2 \pm \cos s_1 \sin s_2 = \sin(s_1 \pm s_2),
  \]
  these equations are readily shown to admit the solutions
  \begin{equation}
    \label{eq:f-solution}
    \begin{aligned}
      f_1(s_1) &= \sqrt{\alpha - \gamma}\sin s_1 , &  f_2(s_2) &= \sqrt{\alpha - \gamma}\cos s_2\\
      h_1(s_1) &= \sqrt{\alpha - \gamma}\cos s_1 , &  h_2(s_2) &= \sqrt{\alpha - \gamma}\sin s_2
    \end{aligned}
  \end{equation}
  with the functions $\tilde\mu_\pm$ given by
  \[
    \tilde\mu_\pm(s_\pm) = (\alpha - \gamma)\sin (2 s_\pm).
  \]
  The functions $g_1$ and $g_2$ are then obtained from
  \begin{equation}
    \label{eq:ellipsoid-g-equations}
    \begin{aligned}
      g_1(s_1)^2 &= \alpha - \beta - (\alpha - \gamma)\sin^2s_1,\\
      g_2(s_2)^2 &= (\alpha-\gamma)\cos^2s_2 - \alpha + \beta.
    \end{aligned}
  \end{equation}
  The right-hand sides are non-negative as long as
  \[
    \sin^2s_1 \leq \frac{\alpha - \beta}{\alpha - \gamma}, \qquad
    \cos^2s_2 \geq \frac{\alpha - \beta}{\alpha - \gamma},
  \]
  which is satisfied for $(s_1, s_2) \in [-s_1^0,s_1^0] \times [-s_2^0,s_2^0]$.
  We note that $s_1^0, s_2^0 > 0$ since
  \[
    0 < \frac{\alpha - \beta}{\alpha - \gamma} < 1.
  \]
  Also note that, for the values $s_1^0$ and $s_2^0$, the boundary of the halves of the ellipsoid at $y=0$ is reached,
  which ensures that each of the two parametrizations covers the entire half, respectively.
\end{proof}
By Proposition~\ref{prop:ellipsoid-circular-sections}, the circular cross sections of an ellipsoid lie in two families of parallel planes.
It is known that any pair of families of parallel planes may be continuously deformed into any other such pair by a unique one-parameter family of affine transformations whose restrictions to the planes are isometries.
Applied to the two families of parallel planes containing the circular cross sections of an ellipsoid, this deformation preserves the circles and, being affine, maps the ellipsoid to another ellipsoid (see Figure~\ref{fig:affine-isometric}).
In particular, this means that one can construct a deformable model of an ellipsoid by interlocking slotted disks (see Figure~\ref{fig:slotted-disks}).
\begin{figure}
  \centering
  \includegraphics[width=0.38\textwidth]{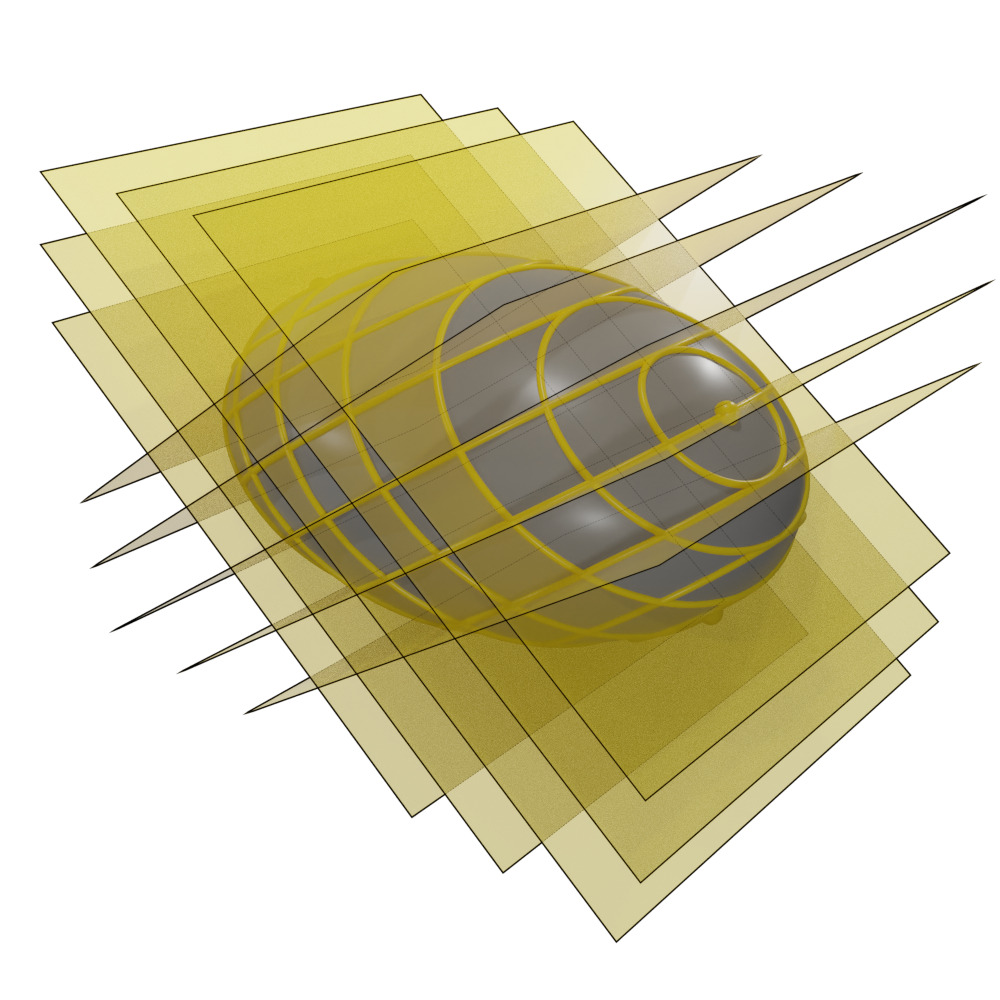}
  \hspace{0.5cm}
  \includegraphics[width=0.38\textwidth]{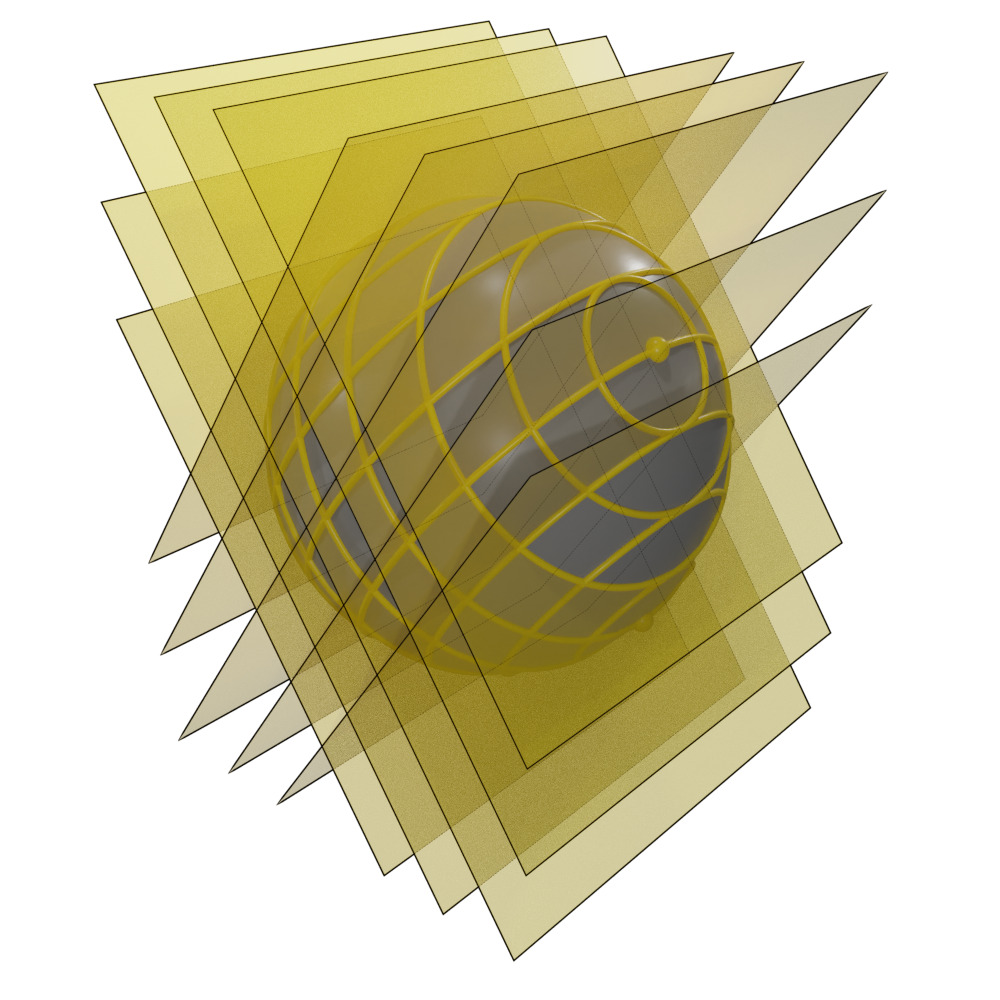}
  \caption{
    An affine transformation which is isometric on two families of parallel planes
    containing the circular cross sections of an ellipsoid.
  }
  \label{fig:affine-isometric}
\end{figure}
\begin{figure}
  \centering
  \includegraphics[width=0.38\textwidth]{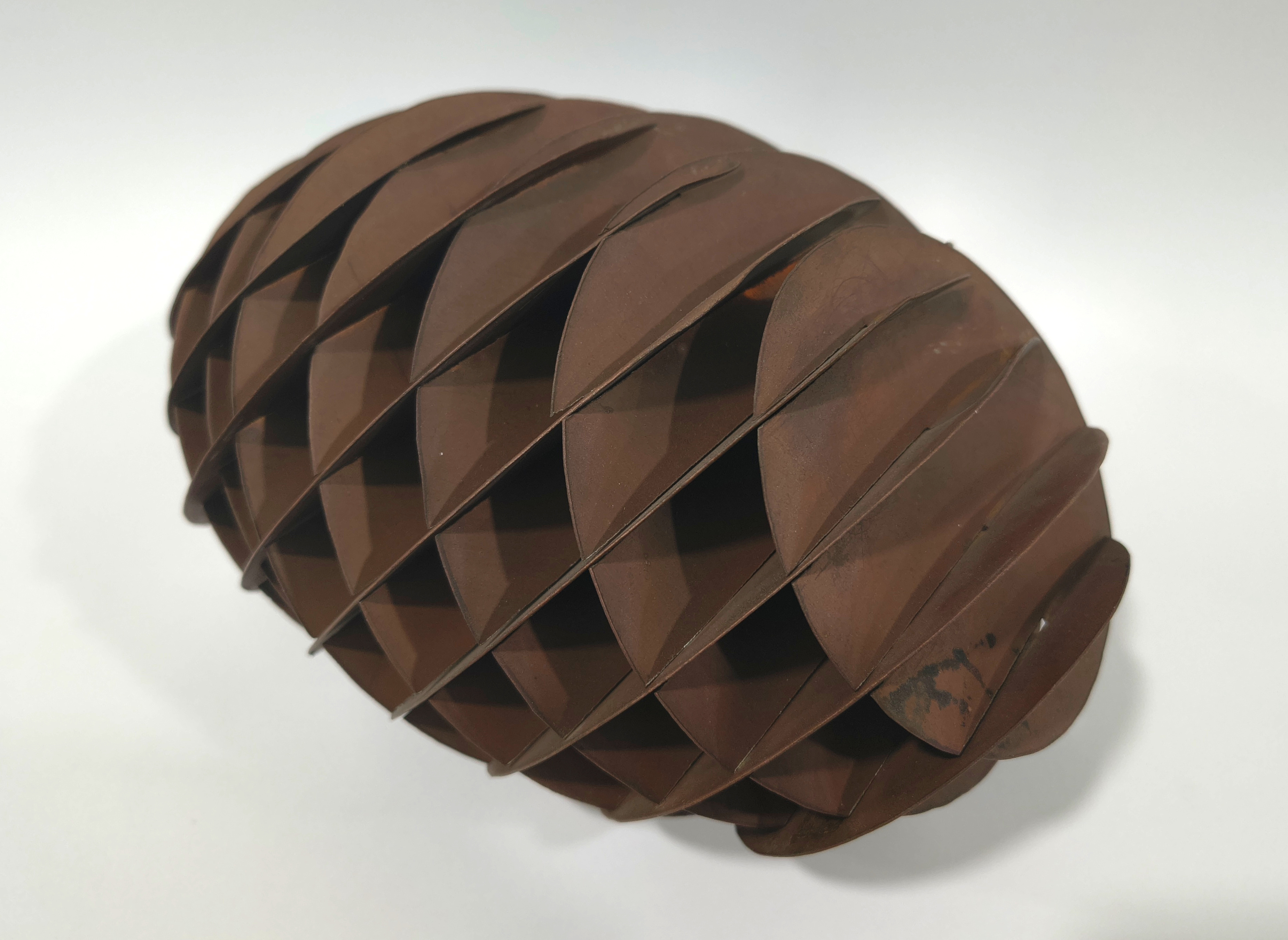}
  \hspace{0.5cm}
  \includegraphics[width=0.38\textwidth]{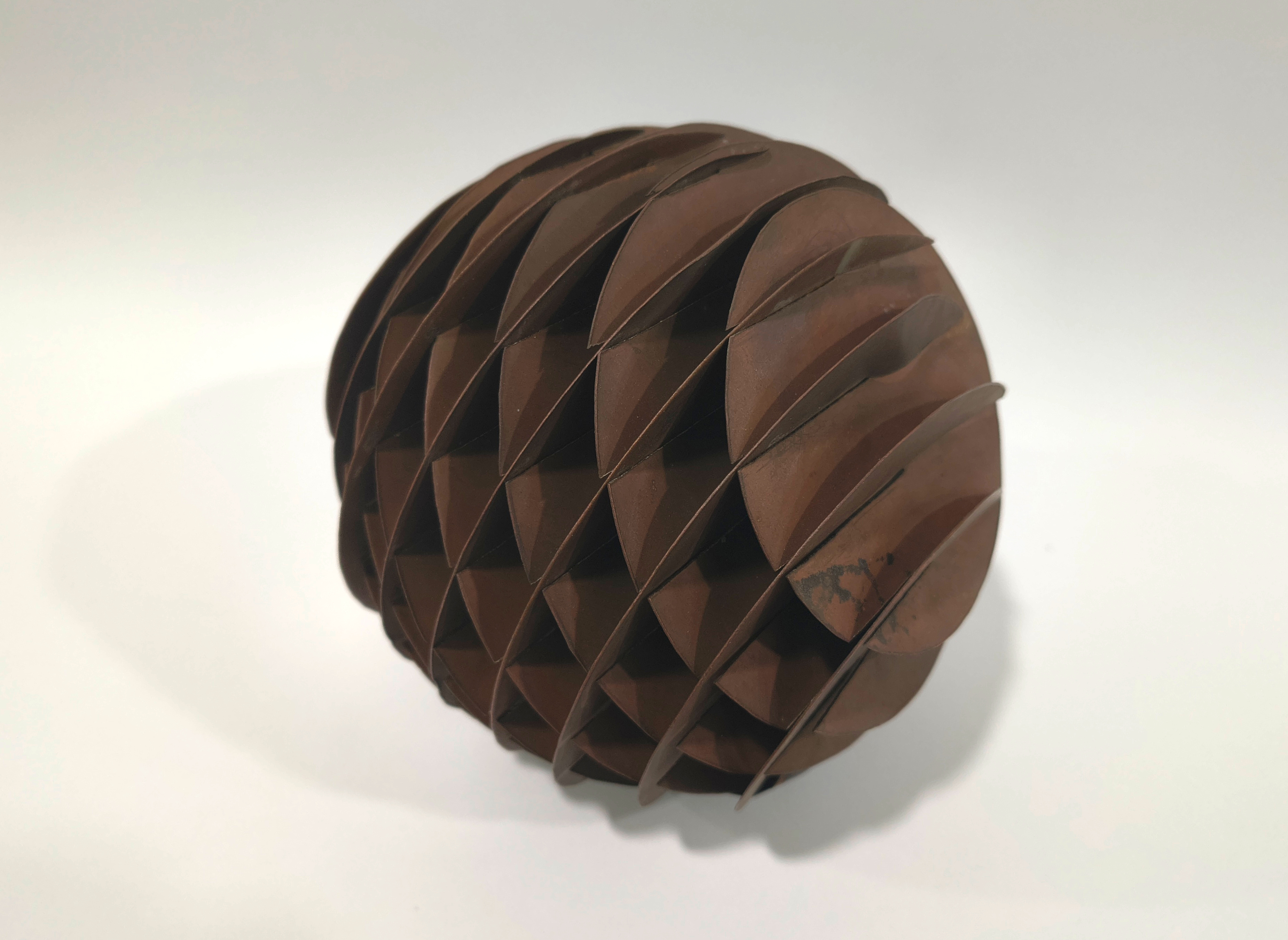}
  \caption{
    Photos of an ``isometrically'' deformable model of an ellipsoid made from slotted metal disks.
    UNSW Sydney.
  }
  \label{fig:slotted-disks}
\end{figure}

To prove part~\ref{thm:circular-cross-sections4} of Theorem~\ref{thm:circular-cross-sections},
we provide the explicit family of affine transformations which is isometric on the circular cross sections of a given ellipsoid.
\begin{proposition}
  \label{prop:isometric-circles}
  Let $\alpha > \beta > \gamma > 0$ and $\mathcal{E} \subset \R^3$ be the ellipsoid \eqref{eq:ellipsoid}.
  Then, the one-parameter family of affine transformations
  \[
    (x, y, z) \mapsto (\sigma_1 x, y, \sigma_3 z)
  \]
  with $\sigma_1, \sigma_3 > 0$ such that
  \begin{equation}
    \label{eq:affine-circle-family}
    \alpha(\beta - \gamma)\sigma_1^2 + \gamma(\alpha-\beta)\sigma_3^2 = \beta(\alpha-\gamma)
  \end{equation}
  is isometric on the circular cross sections of $\mathcal{E}$ (see Figure~\ref{fig:affine-isometric}).
  Any two affine transforms of $\mathcal{E}$ constructed in this manner are confocal up to a scaling.
\end{proposition}
\begin{proof}
  We now derive this affine transformation.
  Let $\sigma_1, \sigma_2, \sigma_3 > 0$ and consider the affine transformation
  \[
    \tilde x = \sigma_1 x, \quad
    \tilde y = \sigma_2 y, \quad
    \tilde z = \sigma_3 z
  \]
  that transforms the ellipsoid \eqref{eq:ellipsoid} into the ellipsoid
  \begin{equation}
    \label{eq:ellipsoid-affine-trafo}
    \tilde{\mathcal{E}} : \frac{\tilde x^2}{\sigma_1^2 \alpha} + \frac{\tilde y^2}{\sigma_2^2 \beta} + \frac{\tilde z^2}{\sigma_3^2 \gamma} = 1.
  \end{equation}
  The planes of the circular cross sections of $\mathcal{E}$ and $\tilde{\mathcal{E}}$ are orthogonal to the plane $y=0$.
  Consider the circular cross section $C$ of one of the two planes through the origin.
  Then, the radius of $C$ is given by $\sqrt{\beta}$.
  To preserve the radius of the circle we must have $\sigma_2 = 1$.
 
  Now, a point that lies on $C$ and in the symmetry plane $y=0$ satisfies
  \[
    x^2 + z^2 = \beta,\qquad
    \frac{x^2}{\alpha} + \frac{z^2}{\gamma} = 1,
  \]
  and, thus,
  \[
    x^2 = \alpha \frac{\beta - \gamma}{\alpha - \gamma}, \qquad
    z^2 = \gamma \frac{\alpha - \beta}{\alpha - \gamma}.
  \]
  The restriction of the affine transformation to the plane of $C$ constitutes an isometry only if
  \[
    x^2 + z^2 = \tilde x^2 + \tilde z^2,
  \]
  i.e.,
  \[
    \beta = \sigma_1^2 \alpha \frac{\beta - \gamma}{\alpha - \gamma} + \sigma_3^2 \gamma \frac{\alpha - \beta}{\alpha - \gamma},
  \]
  which is equivalent to \eqref{eq:affine-circle-family}.
  Finally, the two ellipsoids $\mathcal{E}$ and $\tilde{\mathcal{E}}$ are confocal up to a scaling
  if there exists a $\nu > 0$ such that
  \[
    \alpha \sigma_1^2 - \beta = \nu(\alpha - \beta),\qquad
    \beta - \gamma \sigma_3^2 = \nu(\beta - \gamma),
  \]
  which is again equivalent to \eqref{eq:affine-circle-family}.
\end{proof}

\begin{remark}
  Given an ellipsoid, one may show as above
  that the two-parameter family of affine transformations which map the ellipsoid to every confocal ellipsoid up to scaling
  is equal to the two-parameter family of affine transformations which deform the circular cross sections isometrically up to scaling.
  This family is given by
  \[
    A = \rm{diag}\left( \sigma_1, \sigma_2, \sigma_3 \right)
  \]
  with
  \[
    \alpha(\beta - \gamma)\sigma_1^2 + \gamma(\alpha-\beta)\sigma_3^2 = \beta(\alpha-\beta)\sigma_2^2.
  \]
  Accordingly, this family preserves both the circular cross sections and the curvature lines of an ellipsoid
  and, thus, their diagonal relation.
\end{remark}

As an extension to Proposition~\ref{prop:isometric-circles},
we now construct a curvature line parametrization of a one-parameter family of ellipsoids
which describes the isometric deformation of the circular cross sections.
For each single ellipsoid, the parametrization is similar to the one given in Proposition~\ref{prop:ellipsoid-diagonal-relation}
in the sense that the circular cross sections are given by the curves
\[
  s_1 \pm s_2 = \mbox{const}.
\]
The starting point is the family of confocal ellipsoids from \eqref{E1}
\begin{equation}
  \label{eq:confocal-ellipsoids}
  \frac{x^2}{u_3 + a} + \frac{y^2}{u_3 + b} + \frac{z^2}{u_3 + c} = 1, \quad
  u_3 > -c
\end{equation}
or, by \eqref{E4}, equivalently,
\[
  \frac{x^2}{f_3^2(s_3)} + \frac{y^2}{g_3^2(s_3)} + \frac{z^2}{h_3^2(s_3)} = 1.
\]
We scale each ellipsoid uniformly to obtain a family of ellipsoids with
common second semi-axis equal to 1:
\begin{equation}
  \frac{g_3^2(s_3)}{f_3^2(s_3)}x^2 + y^2 + \frac{g_3^2(s_3)}{h_3^2(s_3)}z^2 = 1.
\end{equation}
By virtue of \eqref{E5}, on introduction of the functions
\bela{E10}
\hat{f}_3(s_3) = \frac{f_3(s_3)}{g_3(s_3)}, \quad \hat{h}_3(s_3) = \frac{h_3(s_3)}{g_3(s_3)},
\ela
this family may be written as
\bela{E12}
\frac{x^2}{\hat{f}_3^2(s_3)} + y^2 + \frac{z^2}{\hat{h}_3^2(s_3)} = 1,
\ela
wherein the functions $\hat f_3$ and $\hat h_3$ constitute the semi-axes of the ellipsoids and are related by
\bela{E11}
(b-c)\hat{f}_3^2(s_3) + (a-b)\hat{h}_3^2(s_3) = a-c.
\ela
It is observed that, by virtue of \eqref{E4}$_5$, the original family is only defined for $\hat f_3^2(s_3) > 1$.
We drop this condition and consider the extended family of ellipsoids from now on.
\begin{lemma}
  \label{lem:ellispoid-isometric-circles}
  Let $a > b > c > 0$.
  Then, the one-parameter family of ellipsoids \eqref{E12} with \eqref{E11}
  consists of affine transforms of any one ellipsoid in the family which are isometric on the circular cross sections.
\end{lemma}
\begin{proof}
  Equations \eqref{E12} and \eqref{E11} are equivalent to \eqref{eq:ellipsoid-affine-trafo} and \eqref{eq:affine-circle-family}
  upon setting
  \[
    \alpha = \frac{a}{b}, \quad \beta = 1, \quad \gamma = \frac{c}{b},
    \quad
    \sigma_1^2 = \frac{b}{a}\hat f_3^2(s_3), \quad \sigma_3^2 = \frac{b}{c}\hat h_3^2(s_3).
  \]
\end{proof}
Again, by \eqref{E5}, the corresponding curvature line parametrizations of the ellipsoids
are obtained by scaling the confocal coordinate system \eqref{E3} with \eqref{E4} by the factor $\frac{1}{g(s_3)}$.
This still holds for the extended family of ellipsoids so that
\bela{E9}
\begin{aligned}
  x(s_1, s_2, s_3) &= \frac{f_1(s_1)f_2(s_2)\hat{f}_3(s_3)}{\sqrt{(a-b)(a-c)}}\\
  y(s_1, s_2, s_3) &= \frac{g_1(s_1)g_2(s_2)}{\sqrt{(a-b)(b-c)}}\\
  z(s_1, s_2, s_3) &= \frac{h_1(s_1)h_2(s_2)\hat{h}_3(s_3)}{\sqrt{(a-c)(b-c)}},
\end{aligned}
\ela
where the functions $f_i$, $g_i$ and $h_i$ are related by
\begin{equation}
  \label{eq:scaled-ellipsoid-f}
  \begin{aligned}
    f_1^2(s_1) + g_1^2(s_1) & = a-b,\quad & f_1^2(s_1) +h_1^2(s_1) & = a-c\\
    f_2^2(s_2) - g_2^2(s_2) & = a-b,\quad & f_2^2(s_2) +h_2^2(s_2) & = a-c
  \end{aligned}
\end{equation}
The following theorem is based on a suitable choice of the functions $f_i$, $g_i$, $h_i$, $\hat f_i$, $\hat h_i$
in the above functional equations.
\begin{theorem}
  \label{thm:isometric-circles}
  Let $a > b > c > 0$ and $\mathcal{E}(s_3) \subset \R^3$ be the one-parameter family of ellipsoids
  \begin{equation}
    \label{eq:deformation-family}
    \frac{b-c}{a-c}\frac{x^2}{\cos^2 s_3} + y^2 + \frac{a-b}{a-c}\frac{z^2}{\sin^2 s_3} = 1.
  \end{equation}
  Then, for each $s_3 \in [0, \pi)$, the two parametrizations
  \begin{equation}
    \label{eq:discrete-ellipsoids-parametrization-family}
    \begin{split}
      \br &= (x, y, z) : [-s_1^0,s_1^0] \times [-s_2^0,s_2^0] \rightarrow \R^3\\[0.5em]
      x(s_1,s_2,s_3) &= \frac{a-c}{\sqrt{(a-b)(b-c)}} \, \sin s_1 \cos s_2 \cos s_3 \\
      y(s_1,s_2,s_3) &= \pm \sqrt{\frac{1}{b-c}} \, \sqrt{1 - \frac{a-c}{a-b}\sin^2s_1} \, \sqrt{\frac{a-c}{a-b}\cos^2s_2 - 1} \\
      z(s_1,s_2,s_3) &= \frac{a-c}{\sqrt{(a-b)(b-c)}} \, \cos s_1 \sin s_1 \sin s_3
    \end{split}
  \end{equation}
  with
  \[
    s_1^0 = \sin^{-1}\sqrt{\frac{a-b}{a-c}}, \quad s_2^0 = \cos^{-1}\sqrt{\frac{a-b}{a-c}}
  \]
  are curvature line parametrizations of the two halves $\mathcal{E}(s_3) \cap \{ y \geq 0 \}$ and $\mathcal{E}(s_3) \cap \{ y \leq 0 \}$
  of the ellipsoid.
  For any two values of $s_3$, the two curves 
  \[
    s_1 + s_2 = \rm{const}
  \]
  are congruent circles,
  and so are the two curves
  \[
    s_1 - s_2 = \rm{const}.
  \]
  The one-parameter family of ellipsoids includes two planar degenerations
  corresponding to $s_3=0$ ($z=0$) and $s_3=\pi/2$ ($x=0$) and the unit sphere
  \bela{E32}
  x^2 + y^2 + z^2 = 1,\quad s_3 = s_3^\circ := \arctan\sqrt{\frac{a-b}{b-c}}
  \ela
  (see Figure~\ref{fig:ellipsoid-family}).
\end{theorem}

\begin{proof}
  We show that there exist solutions $f_1, f_2, g_1, g_2, h_1, h_2$
  of the functional equations~\eqref{eq:scaled-ellipsoid-f}
  such that the diagonal net given by the curves
  \[
    s_\pm = \frac{s_1 \pm s_2}{2} = \text{const}
  \]
  consists of the circular cross sections of the ellipsoids.
  The congruence of the corresponding circles follows from Lemma~\ref{lem:ellispoid-isometric-circles}.
  By substituting the parametrization \eqref{E9} into the planes $\Pi_\pm(\mu_\pm)$ given by \eqref{eq:ellipsoid-circular-sections}
  with
  \[
    \alpha = \hat f_3^2(s_3), \quad
    \beta = 1, \quad
    \gamma = \hat h_3^2(s_3),
  \]
  we obtain
  \[
    \sqrt{\frac{\hat f_3^2(s_3) - 1}{(a - b)\hat f_3^2(s_3)}} f_1(s_1)f_2(s_2)\hat f_3(s_3) \pm \sqrt{\frac{1 - \hat h_3^2(s_3)}{(b - c)\hat h_3^2(s_3)}} h_1(s_1)h_2(s_2)\hat h_3(s_3) = \sqrt{a - c}\,\mu_\pm,
  \]
  where we have assumed that $\hat f_3^2(s_3) > 1$ so that $\alpha > \beta > \gamma$.
  The latter coincides with the assumption on which \eqref{eq:ellipsoid-circular-sections} has been derived.
  Applying \eqref{E11}, which is equivalent to
  \begin{equation}
    \label{E11-variation}
    \frac{\hat f_3^2(s_3) - 1}{a -b} = \frac{1 - \hat h_3^2(s_3)}{b-c},
  \end{equation}
  this may be formulated as
  \begin{equation}
    \label{eq:ellipsoid-parametrization-into-planes2}
    f_1(s_1)f_2(s_2) \pm h_1(s_1)h_2(s_2) = \tilde\mu_\pm.
  \end{equation}
  The case $\hat f_3^2(s_3) < 1$ may be dealt with in a similar manner and leads to the same result.
  For the diagonal lines $s_\pm = \text{const}$ to lie in these planes,
  the right-hand side of equation \eqref{eq:ellipsoid-parametrization-into-planes2}
  must be a function only depending on $s_+$ or $s_-$ respectively.
  Thus, the functions $f_1, f_2, h_1, h_2$ are subject to the four equations
  \begin{equation}
    \label{eq:diagonal-curvature-circular2}
    \begin{aligned}
      f_1(s_1)^2 + h_1(s_1)^2 &= a - c\\
      f_2(s_2)^2 + h_2(s_2)^2 &= a - c\\
      f_1(s_1)f_2(s_2) \pm h_1(s_1)h_2(s_2) &= \tilde\mu_\pm\left(\tfrac{s_1 \pm s_2}{2}\right),
    \end{aligned}
  \end{equation}
  where $\tilde\mu_\pm$ are some functions of their indicated argument only.
  These conditions coincide algebraically with those derived in the proof of Proposition~\ref{prop:ellipsoid-diagonal-relation} so that,
  in complete analogy, we may choose the solutions
  \[
    \begin{aligned}
      f_1(s_1) &= \sqrt{a-c}\sin s_1 , &  f_2(s_2) &= \sqrt{a-c}\cos s_2\\
      h_1(s_1) &= \sqrt{a-c}\cos s_1 , &  h_2(s_2) &= \sqrt{a-c}\sin s_2
    \end{aligned}
  \]
  with the functions $\tilde\mu_\pm$ given by
  \[
    \tilde\mu_\pm(s_\pm) = (a-c)\sin(2s_\pm)
  \]
  and the functions $g_1$ and $g_2$ being obtained from
  \[
    \begin{aligned}
      g_1(s_1)^2 &= a - b - (a - c)\sin^2 s_1\\
      g_2(s_2)^2 &= (a - c)\cos^2 s_2 - a + b.
    \end{aligned}
  \]
  The right-hand sides are non-negative as long as
  \[
    \sin^2s_1 \leq \frac{a-b}{a-c}, \qquad
    \cos^2s_2 \geq \frac{a-b}{a-c},
  \]
  which is satisfied for $(s_1, s_2) \in [-s_1^0,s_1^0] \times [-s_2^0,s_2^0]$.

  Finally, a solution of \eqref{E11} is given by
  \[
    \hat f_3(s_3) = \sqrt{\frac{a-c}{b-c}} \cos s_3,\qquad
    \hat h_3(s_3) = \sqrt{\frac{a-c}{a-b}} \sin s_3.
  \]
  Substitution into \eqref{E9} yields the stated parametrization.

  We observe that $z = 0$ if and only if $s_3 = 0$ and $x = 0$ if and only if $s_3 = \frac{\pi}{2}$.
  A member of the family \eqref{eq:deformation-family} becomes a sphere if and only if
  \[
    \frac{(a-c)\cos^2 s_3}{b-c} = \frac{(a-c)\sin^2 s_3}{a-b} = 1,
  \]
  which is solved by
  \[
    \tan^2 s_3 = \frac{a-b}{b-c}.
  \]
\end{proof}

\newpage
\section{Discrete confocal coordinate systems, ellipsoids as principal binets, and diagonal nets}

In \cite{BobenkoSchiefSurisTechter16,BobenkoSchiefSurisTechter18}, a canonical discretization of confocal coordinate systems and the associated quadrics has been introduced and examined extensively in connection with both their geometric and algebraic properties.
In this section, we summarize and adapt some of the conclusions which we then use for the discretization of circular cross sections of discrete ellipsoids presented in the next section.

\begin{figure}
  \centering
  \includegraphics[width=0.45\textwidth]{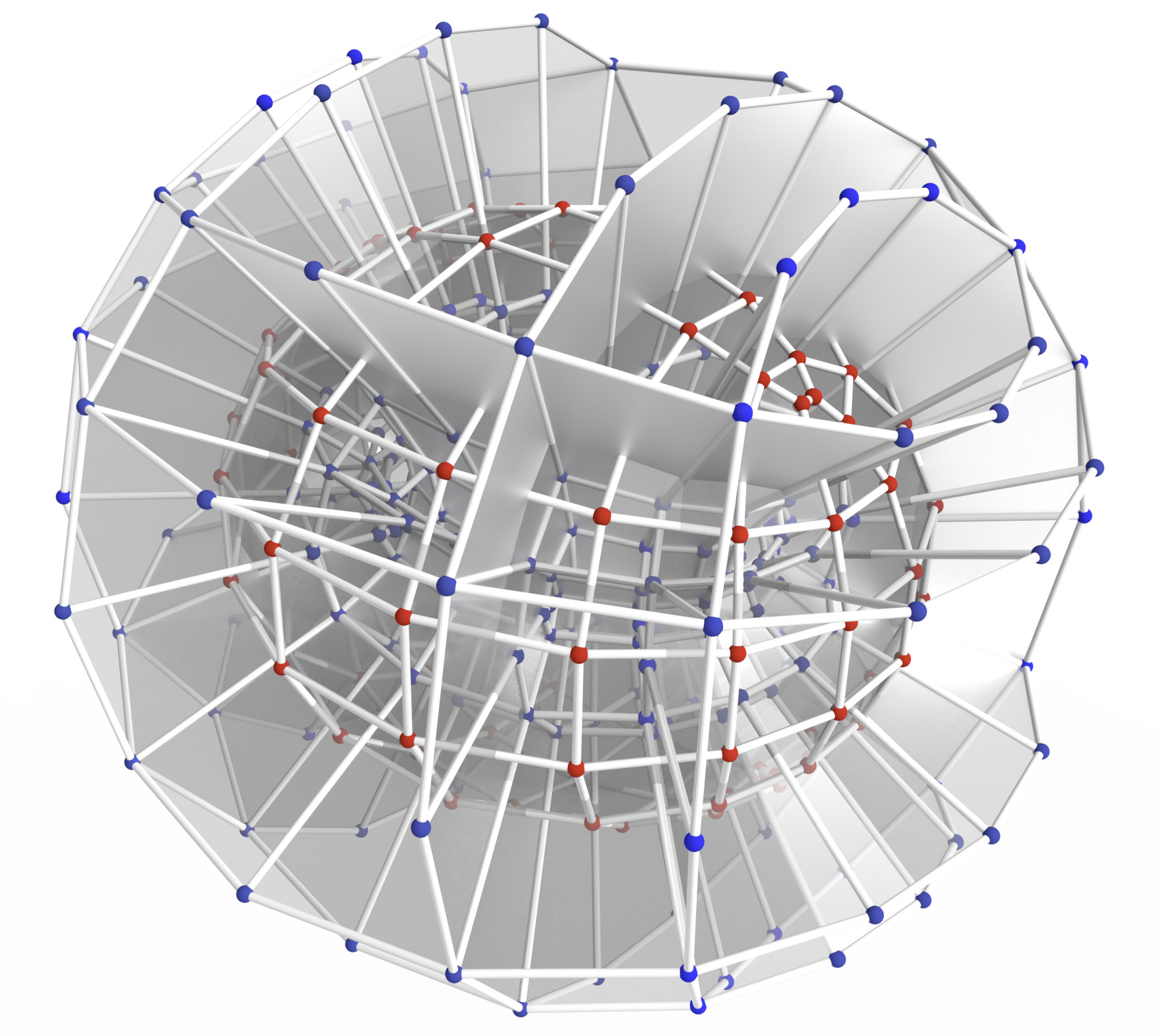}
  \hspace{1cm}
  \includegraphics[width=0.4\textwidth]{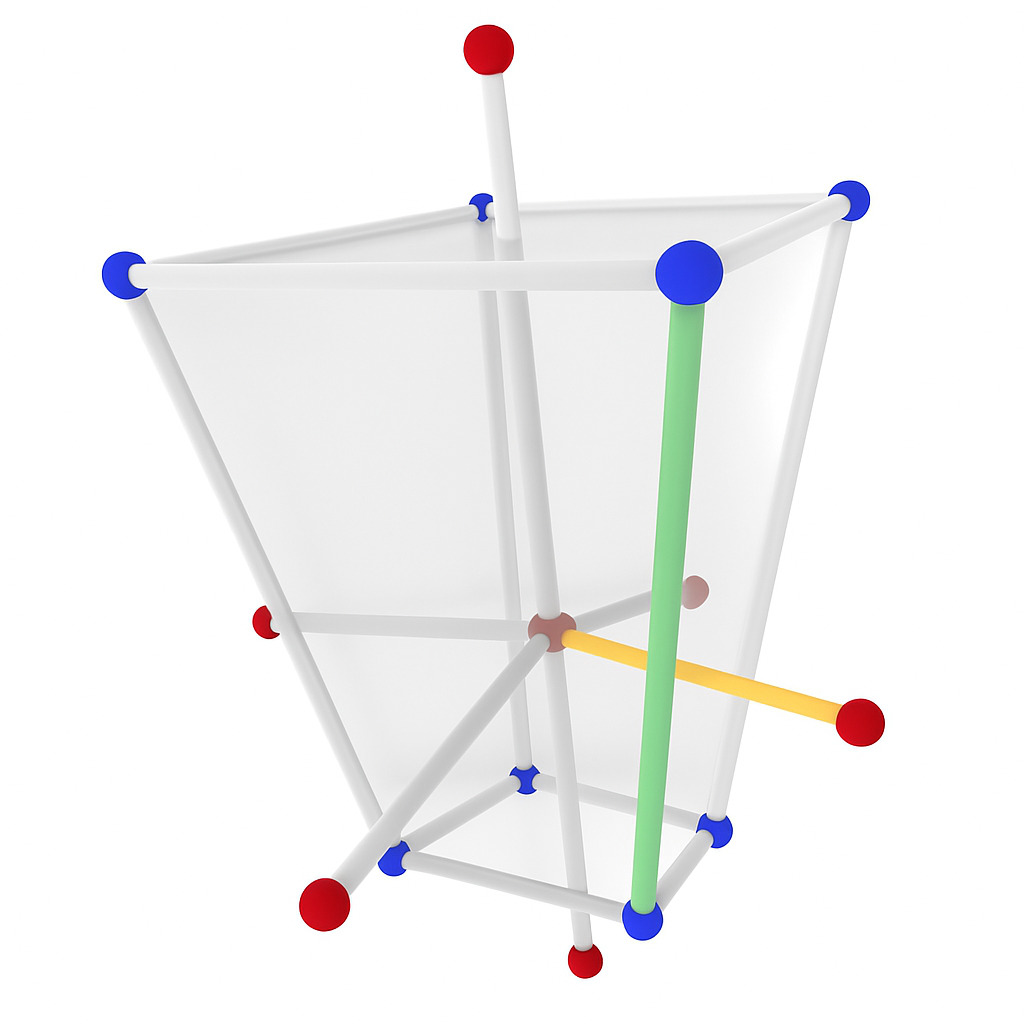}
  \caption{
    \emph{Left:} Part of the image of two dual sublattices $\Z^3$ and $\left(\Z^3\right)^* = \left(\Z+\tfrac{1}{2}\right)^3$ of a discrete confocal coordinate system.
    \emph{Right:} Elementary cube of $\Z^3$ and its dual edges. The green and yellow edges are an example of an orthogonal pair of edges.
  }
  \label{fig:dual-lattices}
\end{figure}

In the following, we require the algebraic notion of a ``discrete square'' which, however, has its origin in natural geometric considerations \cite{BobenkoSchiefSurisTechter18}. Thus, the discrete square of a function
\bela{F1}
  f : \tfrac{1}{2}\Z \rightarrow \R
\ela
is defined by
\bela{F2}
  f^\two(n) = f(n)f(n^+), \quad n^{\pm} = n \pm \tfrac{1}{2}. 
\ela

The discrete analogue $\br$ of a confocal coordinate system \eqref{E3}, \eqref{E4} is then given by a map
\bela{FF4}
\begin{aligned}
  \br &= (x,y,z) : \left(\tfrac{1}{2}\Z\right)^3 \supset\Omega \rightarrow\R^3\\[0.5em]
  x(n_1,n_2,n_3) &= \frac{f_1(n_1)f_2(n_2)f_3(n_3)}{\sqrt{(a-b)(a-c)}}\\
  y(n_1,n_2,n_3) &= \frac{g_1(n_1)g_2(n_2)g_3(n_3)}{\sqrt{(a-b)(b-c)}}\\
  z(n_1,n_2,n_3) &= \frac{h_1(n_1)h_2(n_2)h_3(n_3)}{\sqrt{(a-c)(b-c)}},
\end{aligned}
\ela
where the functions $f_i$, $g_i$ and $h_i$ are related by the functional equations
\bela{FF5}
 \begin{aligned}
  f_1^\two(n_1) + g_1^\two(n_1) & = a-b,\quad & f_1^\two(n_1) +h_1^\two(n_1) & = a-c\\
  f_2^\two(n_2) - g_2^\two(n_2) & = a-b,\quad & f_2^\two(n_2) +h_2^\two(n_2) & = a-c\\
  f_3^\two(n_3) - g_3^\two(n_3) & = a-b,\quad & f_3^\two(n_3) -h_3^\two(n_3) & = a-c
 \end{aligned}
\ela
and $a>b>c$ are constants.
It is important to emphasise that, originally, these discrete confocal coordinate systems have not been obtained by merely replacing in \eqref{E3}, \eqref{E4} squares of functions by discrete squares. On the contrary, the relations \eqref{FF4}, \eqref{FF5} encode the geometric construction of discrete confocal coordinate systems presented in \cite{BobenkoSchiefSurisTechter18}. In fact, we now summarise some of their important geometric properties.

If we focus on the subset of vertices of $\left(\frac{1}{2}\Z\right)^3$ given by the union of the vertices of $\Z^3$ and its dual $\left(\Z^3\right)^* \cong \left(\Z + \frac{1}{2}\right)^3$ (that is, the vertices of the body-centred cubic (bcc) lattice) then one may directly verify that each edge of the lattice $\br(\Z^3\cap\Omega)$ is orthogonal to the four dual edges of the lattice $\br(\left(\Z^3\right)^*\cap\Omega)$ and vice versa (see Figure~\ref{fig:dual-lattices}).

This property is the discrete analogue of the orthogonality of confocal coordinate systems and implies, in turn, that
the quadrilaterals of the two-dimensional sub-lattices $n_i=\mbox{const}$ of $\br(\Z^3\cap\Omega)$ and  $\br(\left(\Z^3\right)^*\cap\Omega)$ are planar. The latter is the defining property of an extensively used discrete analogue of conjugate coordinate systems on two-dimensional surfaces (see \cite{BobenkoSuris08} and references therein). In this connection, it is recalled that lines of curvature on two-dimensional surfaces are characterised by the property that they are both orthogonal and conjugate \cite{Eisenhart60}. Since we are concerned with continuous one-parameter families of discrete ellipsoids, we now perform the continuum limit $n_3\rightarrow s_3$, leading to semi-discrete confocal coordinate systems
\bela{F4}
\begin{aligned}
  \br &= (x,y,z) : \left(\tfrac{1}{2}\Z\right)^2\times\R \supset\Omega \rightarrow\R^3\\[0.5em]
  x(n_1,n_2,s_3) &= \frac{f_1(n_1)f_2(n_2)f_3(s_3)}{\sqrt{(a-b)(a-c)}}\\
  y(n_1,n_2,s_3) &= \frac{g_1(n_1)g_2(n_2)g_3(s_3)}{\sqrt{(a-b)(b-c)}}\\
  z(n_1,n_2,s_3) &= \frac{h_1(n_1)h_2(n_2)h_3(s_3)}{\sqrt{(a-c)(b-c)}},
\end{aligned}
\ela
where the functions $f_i$, $g_i$ and $h_i$ are related by the functional equations
\bela{F5}
 \begin{aligned}
  f_1^\two(n_1) + g_1^\two(n_1) & = a-b,\quad & f_1^\two(n_1) +h_1^\two(n_1) & = a-c\\
  f_2^\two(n_2) - g_2^\two(n_2) & = a-b,\quad & f_2^\two(n_2) +h_2^\two(n_2) & = a-c\\
  f_3^2(s_3) - g_3^2(s_3) & = a-b,\quad & f_3^2(s_3) -h_3^2(s_3) & = a-c.
 \end{aligned}
\ela
If we choose $U \subset \Z^2$ and $U^* \subset \left(\Z^2\right)^*$ where $\Omega \supset (U\cup U^*)\times I$, $I\subset\R$,
then for any fixed value of $s_3$, the map $\br$ restricted to $U \cup U^*$
represents a pair of polyhedral analogues of ellipsoids which are composed of planar quadrilaterals such that dual edges are orthogonal.
This is reflected in the discretization
\begin{equation}
  \label{F7}
  \begin{aligned}
    \frac{x(n_1,n_2,s_3)x(n_1^{\sigma_1},n_2^{\sigma_2},s_3)}{f_3^2(s_3)} &+
    \frac{y(n_1,n_2,s_3)y(n_1^{\sigma_1},n_2^{\sigma_2},s_3)}{g_3^2(s_3)} \\
    &+ \frac{z(n_1,n_2,s_3)z(n_1^{\sigma_1},n_2^{\sigma_2},s_3)}{h_3^2(s_3)} = 1,
    \qquad \sigma_1, \sigma_2 \in \{ +, - \}
  \end{aligned}
\end{equation}
of the one-parameter family of ellipsoids \eqref{E1}$_3$, which is satisfied by the parametrization \eqref{F4}, \eqref{F5}.
It is now convenient to adopt the following definition \cite{principal-binets}.
\begin{figure}
  \centering
  \begin{overpic}[width=0.9\textwidth]{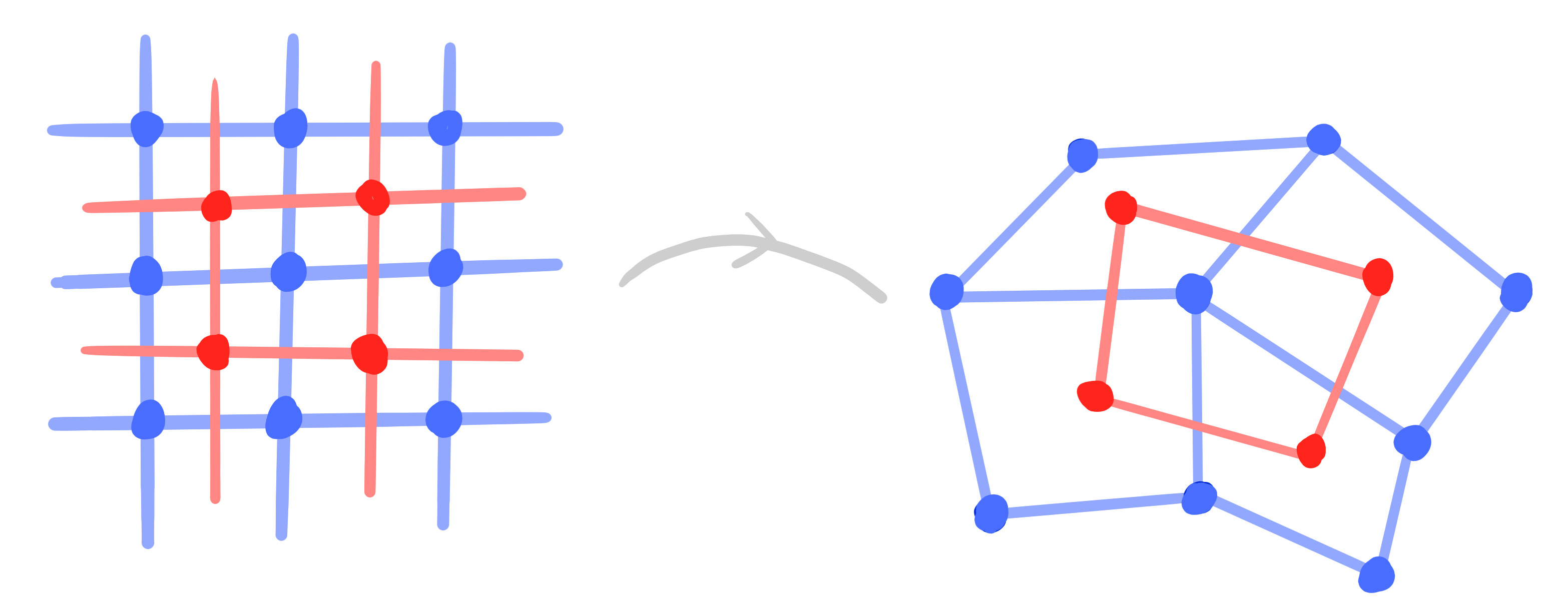}
    \put(3,0){$\textcolor{blue}{\Z^2} \cup \textcolor{red}{\left(\Z^2\right)^*}$}
    \put(95,32){$\R^3$}
    \put(48,25){$\br$}
    \put(5,13.5){\scriptsize\textcolor{red}{\contour{white}{$(n_1^-,n_2^-)$}}}
    \put(24.5,13.5){\scriptsize\textcolor{red}{\contour{white}{$(n_1^+,n_2^-)$}}}
    \put(5,26.8){\scriptsize\textcolor{red}{\contour{white}{$(n_1^-,n_2^+)$}}}
    \put(24.5,26.8){\scriptsize\textcolor{red}{\contour{white}{$(n_1^+,n_2^+)$}}}
    \put(19,18.5){\scriptsize\textcolor{blue}{\contour{white}{$(n_1,n_2)$}}}
    \put(59.5,11){\scriptsize\textcolor{red}{\contour{white}{$\br(n_1^-,n_2^-)$}}}
    \put(84.5,7.5){\scriptsize\textcolor{red}{\contour{white}{$\br(n_1^+,n_2^-)$}}}
    \put(61,26.5){\scriptsize\textcolor{red}{\contour{white}{$\br(n_1^-,n_2^+)$}}}
    \put(89,22){\scriptsize\textcolor{red}{\contour{white}{$\br(n_1^+,n_2^+)$}}}
    \put(77,17.5){\scriptsize\textcolor{blue}{\contour{white}{$\br(n_1,n_2)$}}}
  \end{overpic}
  \caption{A binet is a map $\Z^2 \cup \left(\Z^2\right)^* \rightarrow \R^3$.}
  \label{fig:binet}
\end{figure}
\begin{definition}
  A map
  \begin{equation}
    \label{eq:1}
    \br:\Z^2 \supset U \rightarrow \R^3
    \qquad\text{or}\qquad
    \br:\left(\Z^2\right)^* \supset U^* \rightarrow \R^3
  \end{equation}
  is called a \emph{discret net}.
  A pair of discrete nets
  \bela{F6}
  \br:\Z^2 \cup \left(\Z^2\right)^* \supset U \cup U^* \rightarrow \R^3
  \ela
  is called a \emph{binet} (see Figure~\ref{fig:binet}).
  A \emph{principal binet} is a binet for which the following two conditions are satisfied:
  \begin{enumerate}
  \item
    \label{def:principal-binet1}
    All elementary quadrilaterals $(\br(n_1^-,n_2^-), \br(n_1^+,n_2^-), \br(n_1^+,n_2^+), \br(n_1^-,n_2^+))$ are planar.
    We denote the corresponding planes by $\plane(n_1,n_2)$.
  \item
    All pairs of dual edges $(\br(n_1^-,n_2),\br(n_1^+,n_2))$ and $(\br(n_1,n_2^-),\br(n_1,n_2^+))$ are orthogonal.
  \end{enumerate}
\end{definition}
\begin{remark}
  For any binet satisfying \ref{def:principal-binet1}, we can define a normal line at $(n_1,n_2) \in U \cup U^*$
  as the line through the point $\br(n_1,n_2)$ orthogonal to the plane $\plane(n_1,n_2)$.
  With this definition, adjacent normal lines of a principal binet intersect.
  This is a discrete analogue of the classical fact that two ``infinitesimally'' adjacent normal lines along a line of curvature intersect.
\end{remark}
For any $\alpha > \beta > \gamma > 0$, we obtain a principal binet representing a single discrete ellipsoid
by replacing $a, b, c$ by $\alpha, \beta, \gamma$ in \eqref{F4}
and setting $f_3^2(s_3) = \alpha$, $g_3^2(s_3) = \beta$, $h_3^2(s_3) = \gamma$.
\begin{figure}[b]
  \centering
  \includegraphics[width=0.45\textwidth]{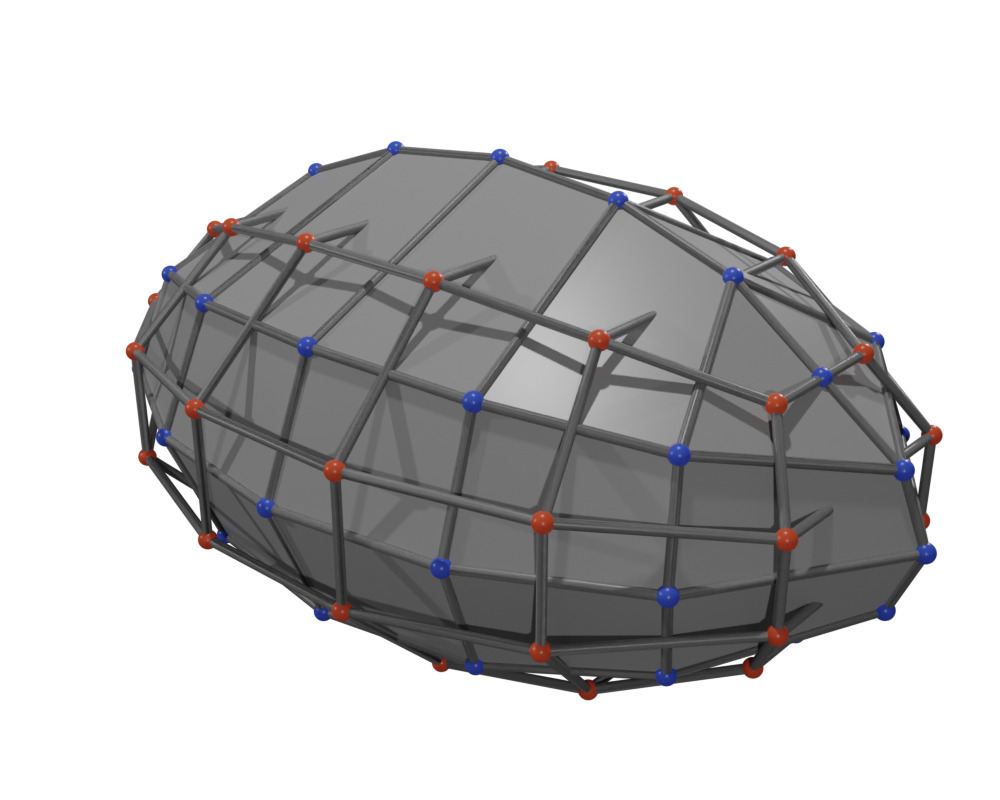}
  \hspace{0.5cm}
  \includegraphics[width=0.45\textwidth]{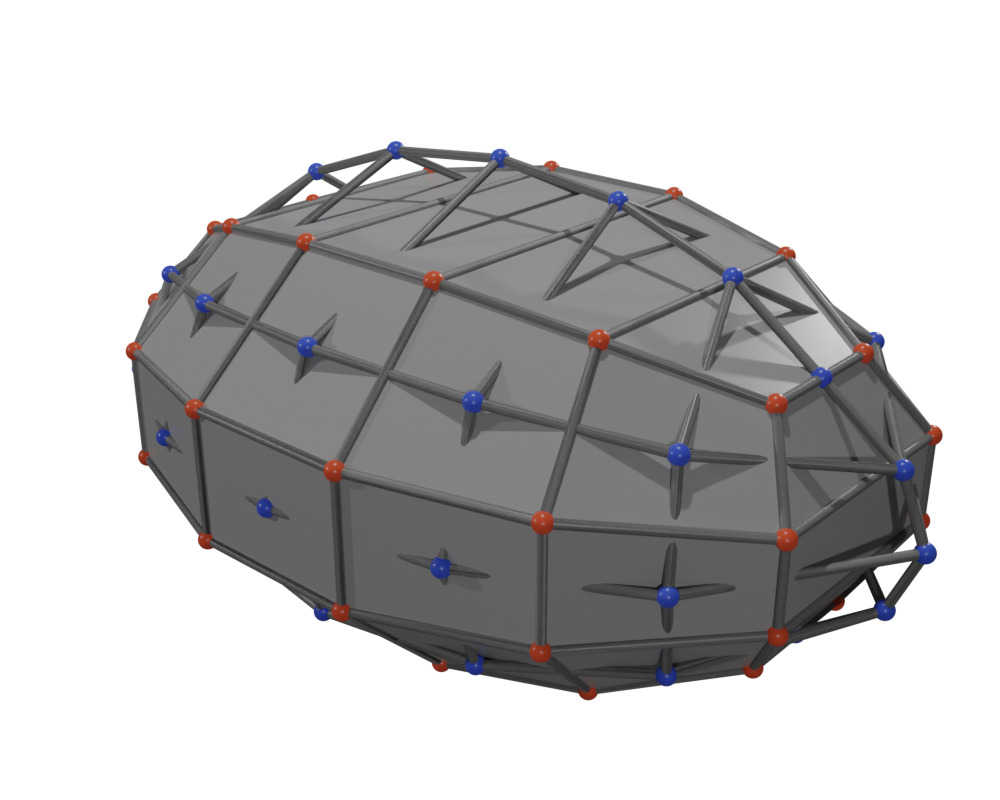}
  \caption{A discrete ellipsoid as a principal binet. Planar faces for each half of the binet are shown, respectively.}
  \label{fig:discrete-ellipsoid-curvature-lines}
\end{figure}
\begin{theorem}
  \label{thm:discrete-ellipsoid-curvature-line-parametrization}
  Let $\alpha > \beta > \gamma > 0$.
  Then, the binet
  \begin{equation}
    \label{eq:discrete-ellipsoid}
    \br:\Z^2 \cup \left(\Z^2\right)^* \supset U \cup U^* \rightarrow \R^3
  \end{equation}
  given by
  \begin{equation}
    \label{eq:discrete-ellipsoid-curvature-line-parametrization}
    \begin{split}
      x(n_1,n_2) &= \sqrt{\frac{\alpha}{(\alpha - \beta)(\alpha - \gamma)}} \, f_1(n_1)f_2(n_2) \\
      y(n_1,n_2) &= \sqrt{\frac{\beta}{(\alpha - \beta)(\beta - \gamma)}} \, g_1(n_1)g_2(n_2) \\
      z(n_1,n_2) &= \sqrt{\frac{\gamma}{(\alpha - \gamma)(\beta - \gamma)}} \, h_1(n_1)h_2(n_2)
    \end{split}
  \end{equation}
  with
  \begin{equation}
    \label{eq:discrete-f-single}
    \begin{aligned}
      f_1^\two(n_1) + g_1^\two(n_1) & = \alpha - \beta,\quad & f_1^\two(n_1) +h_1^\two(n_1) & = \alpha - \gamma\\
      f_2^\two(n_2) - g_2^\two(n_2) & = \alpha - \beta,\quad & f_2^\two(n_2) +h_2^\two(n_2) & = \alpha - \gamma
    \end{aligned}
  \end{equation}
  has the following properties:
  \begin{enumerate}
  \item
    $\br$ is a principal binet (see Figure~\ref{fig:discrete-ellipsoid-curvature-lines}).
  \item
    \label{thm:discrete-ellipsoid-curvature-line-parametrization2}
    For each vertex $\br(n_1,n_2)$, the plane $\plane(n_1,n_2)$ is its polar plane
    with respect to the classical ellipsoid \eqref{eq:ellipsoid}
    \[
      \frac{x^2}{\alpha} + \frac{y^2}{\beta} + \frac{z^2}{\gamma} = 1
    \]
    (see Figure~\ref{fig:discrete-ellipsoid-polarity}).
  \end{enumerate}
\end{theorem}
\begin{proof}\
  \nobreakpar
  \begin{enumerate}
  \item
    This is a direct consequence of the limiting procedure applied to discrete confocal quadrics as described above.
  \item
    The polarity follows from \eqref{F7}.
  \end{enumerate}
\end{proof}
\begin{figure}
  \centering
  \includegraphics[width=0.45\textwidth]{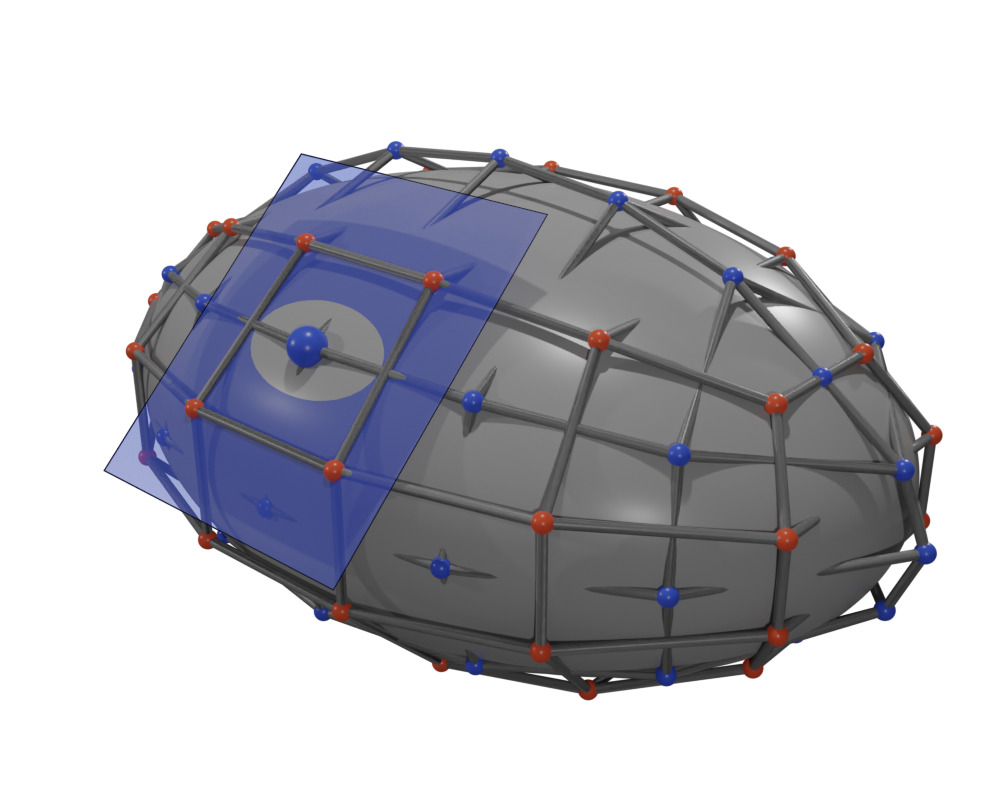}
  \caption{
    Polarity for the principal binet representing a discrete ellipsoid.
    Each vertex of the binet is the pole of the corresponding dual face plane
    with respect to an underlying classical ellipsoid.
  }
  \label{fig:discrete-ellipsoid-polarity}
\end{figure}

As in the continuous case, for any discrete binet
\[
  \br : \Z^2 \cup (\Z^*)^2 \supset U \cup U^* \rightarrow \R^3,\quad
\]
it is natural to introduce new discrete coordinates $(n_+,n_-)\in\Z^2$ defined by
\bela{G1}
n_\pm = \frac{n_1 \pm n_2}{2}
\ela
so that the domains $U$ and $U^*$ are characterised by
\bela{G2}
n_+ \pm n_- \in
\left\{\begin{aligned} & \Z \mbox{ on } U\\ & \Z^* \mbox{on }U^*.\end{aligned}\right.
\ela
Accordingly, a discrete curve $n_+=\mbox{const}$ is a pair of polygons consisting of a polygon
for which $n_++n_- \in \Z$ and a polygon for which $n_++n_- \in \Z^*$.
If $n_++n_-$ is in $\Z$ or in $\Z^*$ then the edges of the corresponding polygon $n_+=\mbox{const}$ are the diagonals of quadrilaterals of the discrete net $\br(U)$ or $\br(U^*)$ respectively.
This leads to the following natural definition.

\begin{definition}\
  \nobreakpar
  \label{def:discrete-diagonal}
  \begin{enumerate}
  \item
    Let
    \[
      \br : \Z^2 \supset U \rightarrow \R^3
    \]
    be a discrete net.
    Then, the binet
    \[
      \Z^2 \cup (\Z^*)^2 \supset \tilde{U} \cup \tilde{U}^* \rightarrow \R^3,\quad
      (n_+, n_-) \mapsto \br(n_+ + n_-, n_+ - n_-)
    \]
    where $\tilde{U} \cup \tilde{U}^*$ is defined such that $(n_+, n_-) \in \tilde{U} \cup \tilde{U}^*$ if and only if $(n_++n_-,n_+-n_i) \in U$,
    is called \emph{diagonal to} $\br$ (cf.\ Figure~\ref{fig:discrete-ellipsoid-diagonal}, top).
  \item
    Let
    \[
      \br : \Z^2 \cup \left(\Z^2\right)^* \supset U \cup U^* \rightarrow \R^3
    \]
    be a binet.
    Then, the pair of binets, which together may be defined on $(\frac{1}{2}\Z)^2$,
    \[
      (\tfrac{1}{2}\Z)^2 \supset \tilde{\Omega} \rightarrow \R^3,\quad
      (n_+, n_-) \mapsto \br(n_+ + n_-, n_+ - n_-),
    \]
    where $\tilde\Omega$ is defined such that $(n_+, n_-) \in \tilde{\Omega}$ if and only if $(n_++n_-,n_+-n_i) \in U \cup U^*$,
    is called \emph{diagonal to} $\br$ (cf.\ Figure~\ref{fig:discrete-ellipsoid-diagonal}, bottom).
  \end{enumerate}
\end{definition}
\begin{remark}
  Note that if $\br$ has planar faces, then the diagonal binets have intersecting dual edges.
\end{remark}

\section{A discretization of circular cross sections of ellipsoids and their isometric deformation}

In this section, it is demonstrated that the theory developed in \cite{BobenkoSchiefSurisTechter16,BobenkoSchiefSurisTechter18}
and extended in the previous section can be applied to the classical results concerning the circular cross sections of ellipsoids. In particular, this leads to the explicit construction of polyhedral analogues of ellipsoids composed of planar quadrilaterals with their diagonals forming planar closed polygons which constitute discretizations of the classical circular cross sections of ellipsoids.
This provides a further indication that the novel definition of discrete lines of curvature proposed in \cite{BobenkoSchiefSurisTechter18},
which are now termed principal binets \cite{principal-binets}, is both applicable and natural.

We start with the definition of a \emph{discrete circle},
which is based on polarity with respect to a classical circle,
similar to the polarity of a discrete ellipsoid with respect to a classical ellipsoid as observed in Theorem~\ref{thm:discrete-ellipsoid-curvature-line-parametrization}.
\begin{figure}[b]
  \centering
  \begin{overpic}[width=0.3\textwidth]{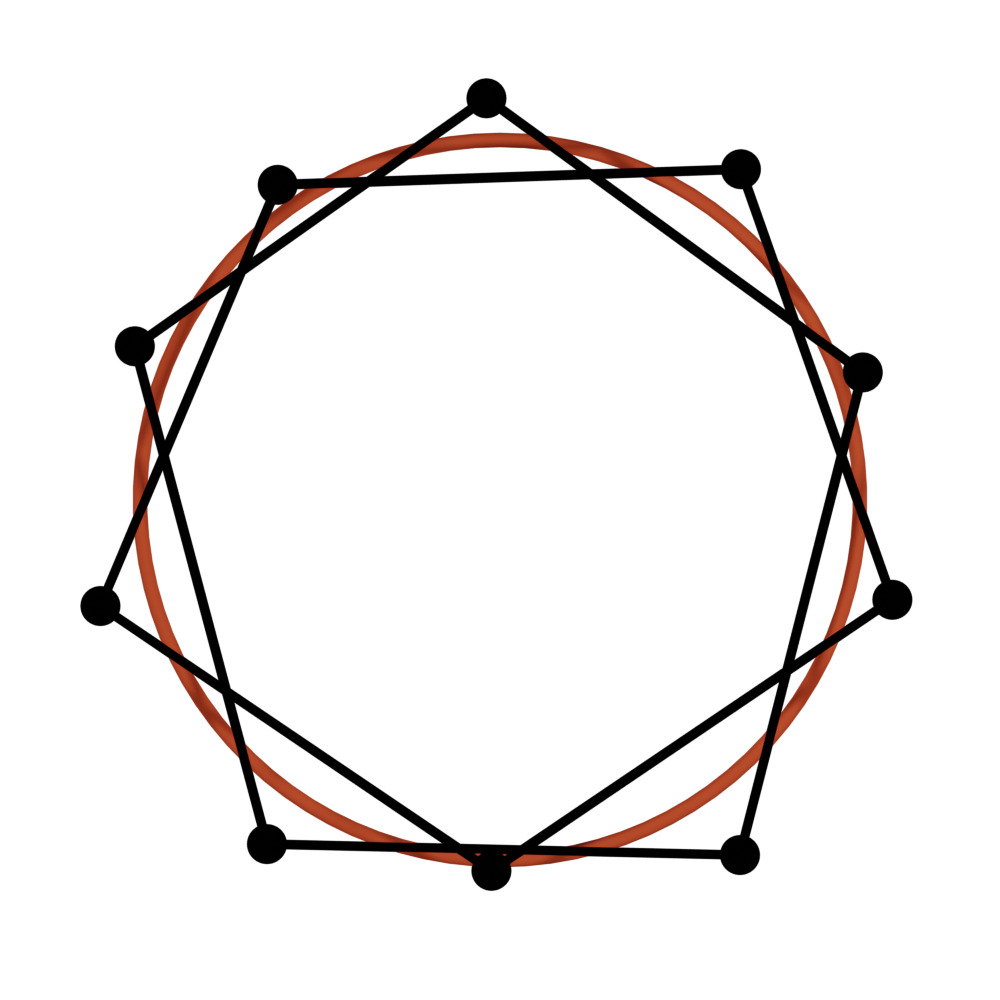}
    \put(90,62){$\br(n)$}
    \put(77.5,82){$\br(n^+)$}
    \put(93,39){$\br(n^-)$}
  \end{overpic}
  \caption{
    A discrete circle.
    Each vertex $\br(n)$ is the pole of the line passing through $\br(n^-)$ and $\br(n^+)$
    with respect to an underlying classical circle.
  }
  \label{fig:discrete-circle}
\end{figure}
\begin{definition}
  \label{def:discrete-circle}
  A planar discrete curve
  \[
    \br = (x, y) : \Z \cup \Z^* \supset I \cup I^* \rightarrow \R^2
  \]
  is called a \emph{discrete circle} if there exists $(x_0, y_0) \in \R^2$ and $R > 0$
  such that each vertex $\br(n)$ is the pole of the line passing through $\br(n^-)$ and $\br(n^+)$
  with respect to the circle of radius $R$ centred at $(x_0,y_0)$ (see Figure~\ref{fig:discrete-circle}).
  Algebraically, this is encoded in the relation
  \[
    (x(n) - x_0)^\two + (y(n) - y_0)^\two = R^2.
  \]
\end{definition}
\begin{remark}
  If we view a discrete circle as a pair of polygons given by the restrictions $\br(I)$ and $\br(I^*)$,
  the vertices of $\br(I)$ are the poles of the dual edges of $\br(I^*)$.
\end{remark}

Based on the coordinate transformation \eqref{G1}, we are now in a position to formulate and prove the following discrete analogue of Proposition~\ref{prop:ellipsoid-diagonal-relation}.
\begin{figure}
  \centering
  \includegraphics[width=0.42\textwidth]{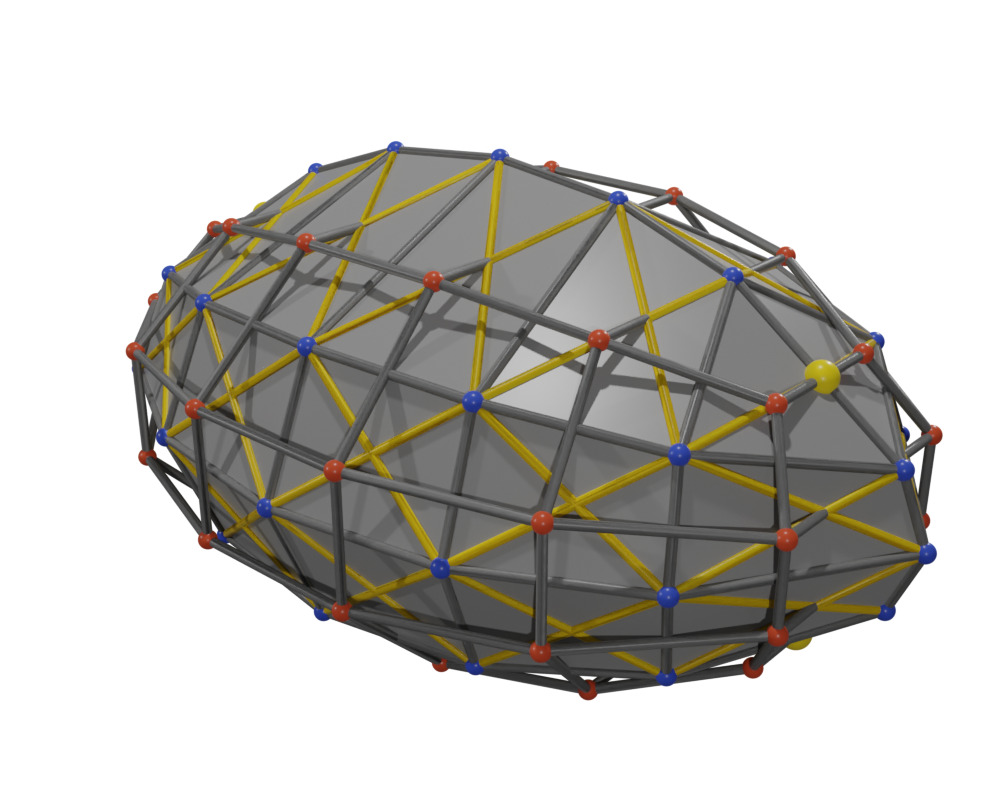}
  \hspace{0.5cm}
  \includegraphics[width=0.42\textwidth]{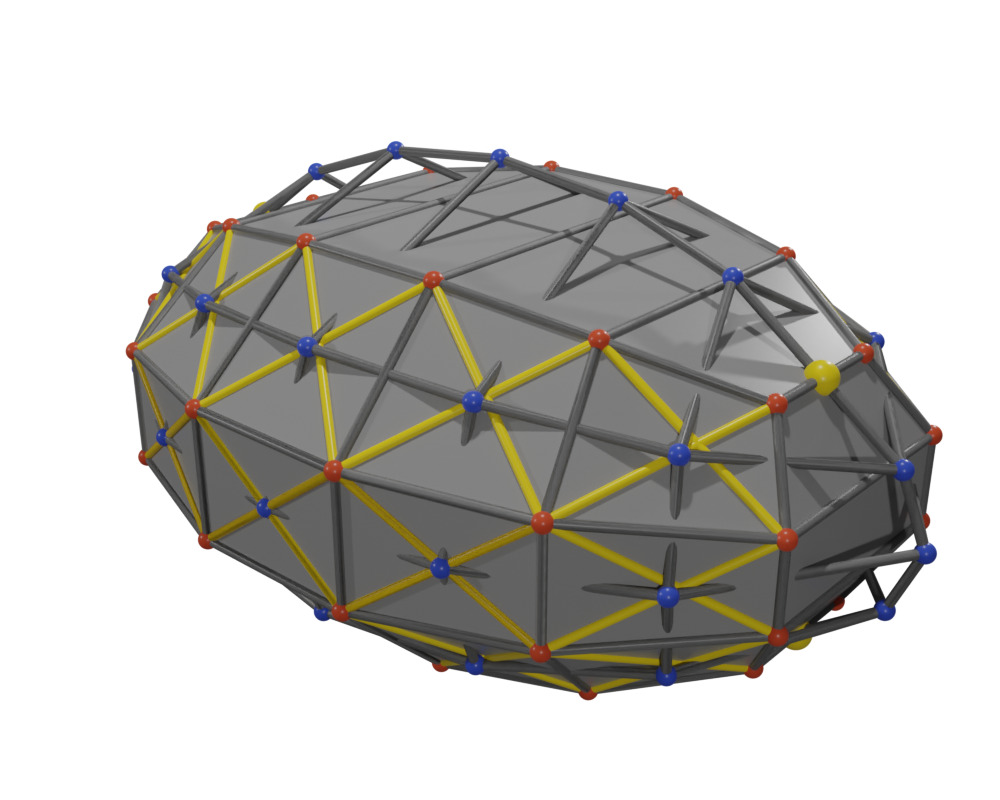}
  \includegraphics[width=0.42\textwidth]{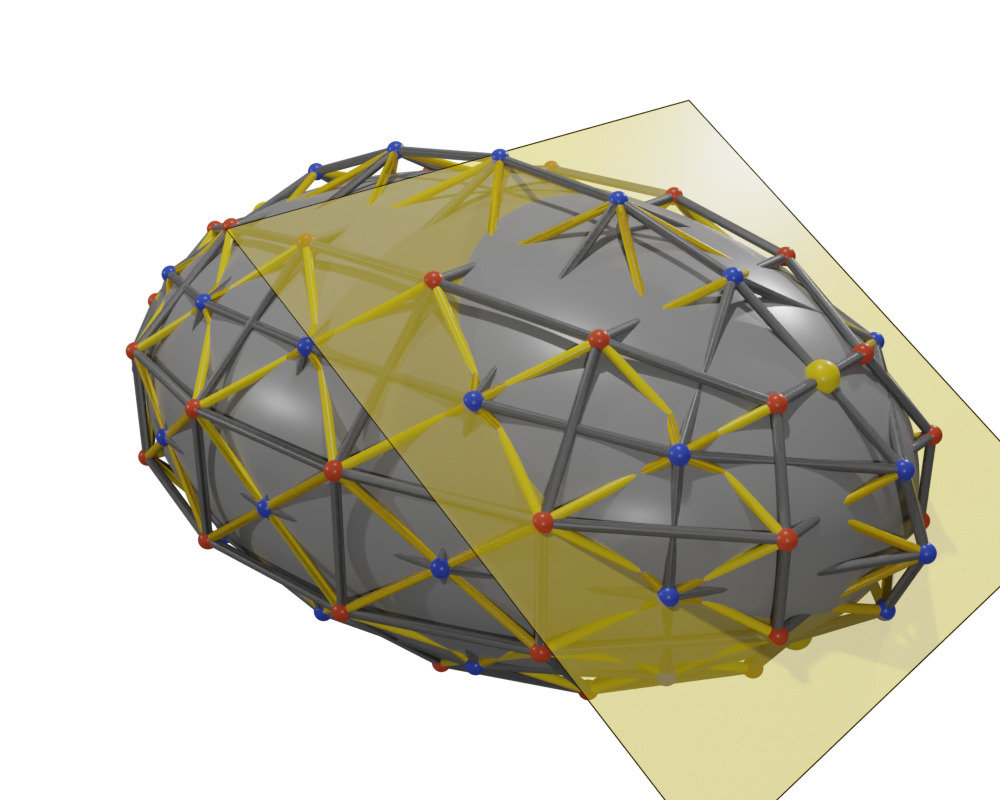}
  \caption{
    Discrete circular cross sections of a discrete ellipsoid.
    \emph{Top-left:}
    A pair of diagonal nets (yellow) to one half of a principal binet (shown with faces) representing a discrete ellipsoid.
    The two diagonal nets form a diagonal binet.
    \emph{Top-right:}
    A pair of diagonal nets (yellow) to the other half of a principal binet (shown with faces) representing a discrete ellipsoid.
    The two diagonal nets form another diagonal binet.
    \emph{Bottom:}
    Together, the two diagonal binets form the discrete circular cross sections of the discrete ellipsoid.
    The underlying classical ellipsoid is shown together with one of the planes of a circular cross section
    which contains one of the discrete circles.
  }
  \label{fig:discrete-ellipsoid-diagonal}
\end{figure}
\begin{theorem}
  \label{thm:discrete-diagonal-curvature-circular}
  Let $\alpha > \beta > \gamma > 0$ and $0 < \delta < \pi$.
  Then, the binet
  \begin{equation}
    \label{eq:discrete-ellipsoid-circles}
    \begin{split}
      \br &= (x,y,z) : \Z^2 \cup \left(\Z^2\right)^* \supset U \cup U^* \rightarrow \R^3\\[0.5em]
      x(n_1,n_2) &= \frac{1}{\cos(\delta/2)} \sqrt{\frac{\alpha(\alpha-\gamma)}{(\alpha - \beta)}} \, \sin(\delta n_1)\cos(\delta n_2) \\
      y(n_1,n_2) &= \sqrt{\frac{\beta}{(\alpha-\beta)(\beta-\gamma)}} g_1(n_1)g_2(n_2)\\
      z(n_1,n_2) &= \frac{1}{\cos(\delta/2)} \sqrt{\frac{\gamma(\alpha-\gamma)}{(\beta-\gamma)}} \, \cos(\delta n_1)\sin(\delta n_2),
    \end{split}
  \end{equation}
  where $g_1$ and $g_2$ are determined by initial conditions and the recurrence relations
  \bela{F13}
  g_1(n_1^+) = \frac{\alpha-\beta - f_1(n_1)^\two}{g_1(n_1)},\quad
  g_2(n_2^+) = \frac{f_2(n_2)^\two - \alpha+\beta}{g_2(n_2)},
  \ela
  has the following properties:
  \begin{enumerate}
  \item
    It constitutes a particular case of the class of discrete ellipsoids described in Theorem~\ref{thm:discrete-ellipsoid-curvature-line-parametrization}.
  \item
    The discrete curves
    \begin{equation}
      \label{eq:discrete-circular-cross-sections}
      n_\pm = \frac{n_1 \pm n_2}{2} = \mbox{const}
    \end{equation}
    are planar, and they are discrete circles in the sense of Definition~\ref{def:discrete-circle}.

    The collection of discrete circles form two binets
    which are diagonal to $\br$ in the sense of Definition~\ref{def:discrete-diagonal} (see Figure~\ref{fig:discrete-ellipsoid-diagonal}).
  \end{enumerate}
\end{theorem}
\begin{proof}
  A natural discrete analogue of the functions \eqref{eq:f-solution} is given by
  \bela{F11}
  \begin{aligned}
    f_1(n_1) & = \epsilon \sqrt{\alpha - \gamma}\sin (\delta n_1),\quad &h_1(n_1) &= \epsilon \sqrt{\alpha - \gamma}\cos (\delta n_1)\\
    f_2(n_2) & = \epsilon \sqrt{\alpha - \gamma}\cos (\delta n_2),\quad &h_2(n_2) &= \epsilon\sqrt{\alpha - \gamma}\sin (\delta n_2).
  \end{aligned}
  \ela
  Indeed, the functional equations \eqref{eq:discrete-f-single} are satisfied provided that
  \bela{F12}
  \epsilon = \frac{1}{\sqrt{\cos(\delta/2)}}
  \ela
  for some constant $0 < \delta < \pi$.
  It is evident that \eqref{eq:f-solution} is retrieved in the limit $s_i = \delta n_i$, $\delta\rightarrow 0$.
  The functions $g_1, g_2$ in \eqref{eq:discrete-f-single}
  are obtained from the recurrence relations \eqref{F13}
  which determine them up to their initial values.
  Thus, the binets \eqref{eq:discrete-ellipsoid-circles} constitute a particular case of \eqref{eq:discrete-ellipsoid}
  with the corresponding classical ellipsoid \eqref{eq:ellipsoid}
  \[
    \frac{x^2}{\alpha} + \frac{y^2}{\beta} + \frac{z^2}{\gamma} = 1.
  \]

  Substituting the binet \eqref{eq:discrete-ellipsoid} into the equations of the circular cross sections
  of the classical ellipsoid \eqref{eq:ellipsoid} given by \eqref{eq:ellipsoid-circular-sections}, we obtain
  \[
    f_1(n_1)f_2(n_2) \pm h_1(n_1)h_2(n_2) = \sqrt{\beta(\alpha-\gamma)} \mu_\pm.
  \]
  The functions \eqref{F11}, on the other hand, satisfy,
  \[
    f_1(n_1)f_2(n_2) \pm h_1(n_1)h_2(n_2) = \tilde\mu_\pm(\tfrac{n_1\pm n_2}{2})
  \]
  with
  \[
    \tilde\mu_\pm(n_{\pm}) = \cos(\delta/2)(\alpha - \gamma) \sin(2n_\pm).
  \]
  Thus, the parameter polygons of the diagonal binets lie in the planes of the circular cross sections
  of the classical ellipsoid \eqref{eq:ellipsoid}.
  By Theorem~\ref{thm:discrete-ellipsoid-curvature-line-parametrization}~\ref{thm:discrete-ellipsoid-curvature-line-parametrization2}, the point $\br(n_1,n_2)$ and the plane $\plane(n_1,n_2)$ are polar.
  In the restriction to a parameter polygon, say $n_- = \tfrac{n_1-n_2}{2} = \mathrm{const}$,
  this means that the point $\br(n_1,n_2)$ and the line spanned by $\br(n_1^-,n_2^-)$ and $\br(n_1^+,n_2^+)$ are polar
  with respect to the circular cross section of the ellipsoid.
  Thus, the parameter polygons of the diagonal binets form discrete circles.
\end{proof}
\begin{remark}
  Considering the binet \eqref{eq:discrete-ellipsoid-circles} on a suitable domain $U \cup U^*$
  will only generate the discrete analogue of one half of an ellipsoid.
  Analogous to the smooth case, to generate both halves of an ellipsoid,
  one should consider two binets with different (positive and negative) initial conditions for $g_1$.
  This will be discussed in more detail in Section~\ref{sec:boundary-conditions}.
\end{remark}

As in the continuous case,
to obtain an isometric deformation of the discrete circular cross sections,
we scale the discrete ellipsoids of the family \eqref{F4} by $\frac{1}{g(s_3)}$ such that
\bela{F8}
\begin{aligned}
  x(n_1,n_2,s_3) &= \frac{f_1(n_1)f_2(n_2)\hat f_3(s_3)}{\sqrt{(a-b)(a-c)}}\\
  y(n_1,n_2,s_3) &= \frac{g_1(n_1)g_2(n_2)}{\sqrt{(a-b)(b-c)}}\\
  z(n_1,n_2,s_3) &= \frac{h_1(n_1)h_2(n_2)\hat h_3(s_3)}{\sqrt{(a-c)(b-c)}},
\end{aligned}
\ela
subject to the functional equations
\bela{F9}
 \begin{gathered}
  f_1^\two(n_1) + g_1^\two(n_1)  = a-b,\quad  f_1^\two(n_1) +h_1^\two(n_1)  = a-c\\
  f_2^\two(n_2) - g_2^\two(n_2)  = a-b,\quad  f_2^\two(n_2) +h_2^\two(n_2)  = a-c\\
  (b-c)\hat{f}_3^2(s_3) + (a-b)\hat{h}_3^2(s_3) = a-c,
 \end{gathered}
\ela
where, as before,
\bela{F10}
  \hat{f}_3(s_3) = \frac{f_3(s_3)}{g_3(s_3)}, \quad \hat{h}_3(s_3) = \frac{h_3(s_3)}{g_3(s_3)}.
\ela
The one-parameter family (in $s_3$) of discrete ellipsoids encoded in \eqref{F8}, \eqref{F9} form the basis of the following discussion.

The following theorem is the discrete analogue of Theorem~\ref{thm:isometric-circles}.
\begin{figure}
  \centering
  \includegraphics[width=0.24\textwidth]{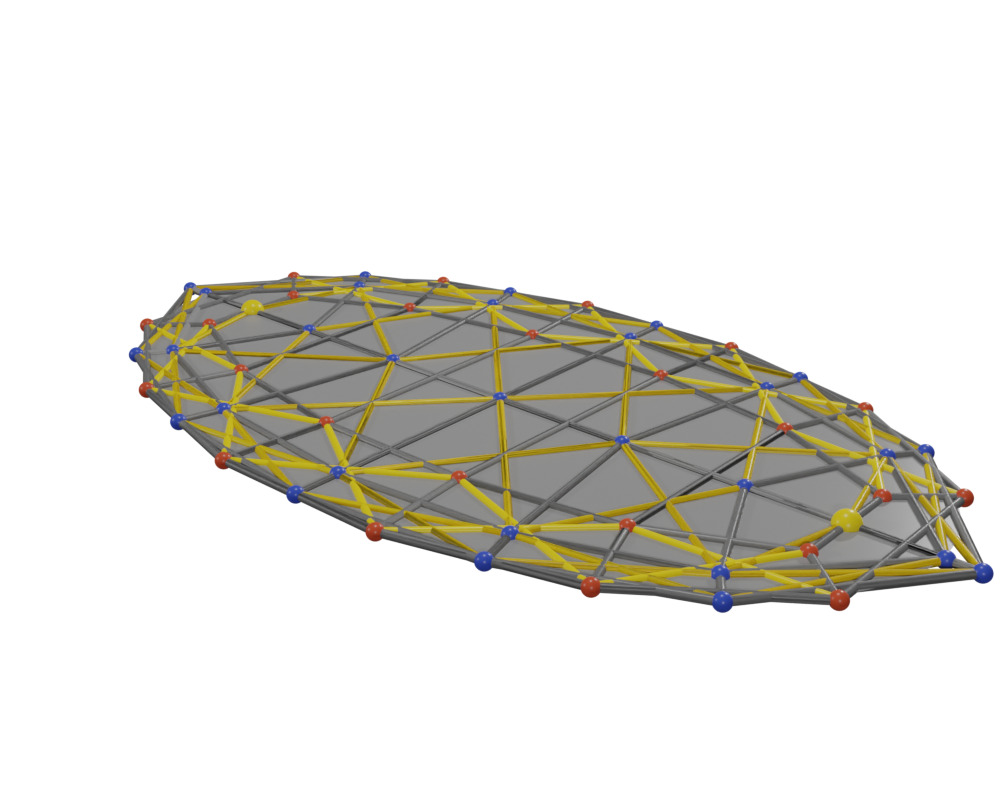}
  \includegraphics[width=0.24\textwidth]{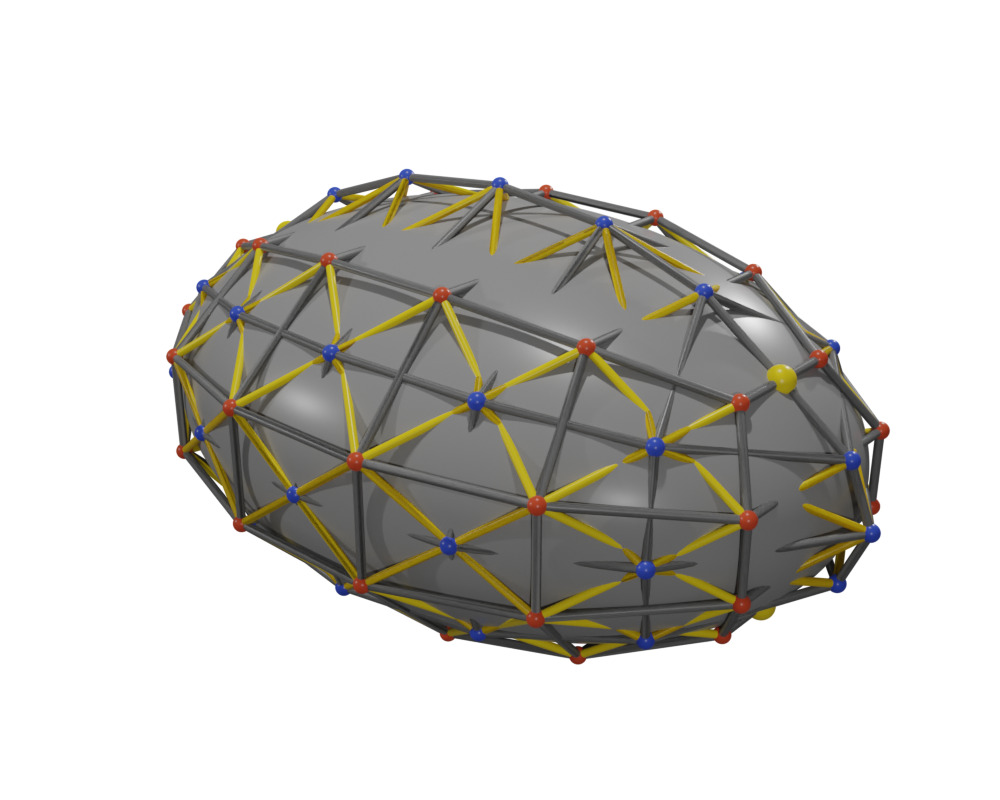}
  \includegraphics[width=0.24\textwidth]{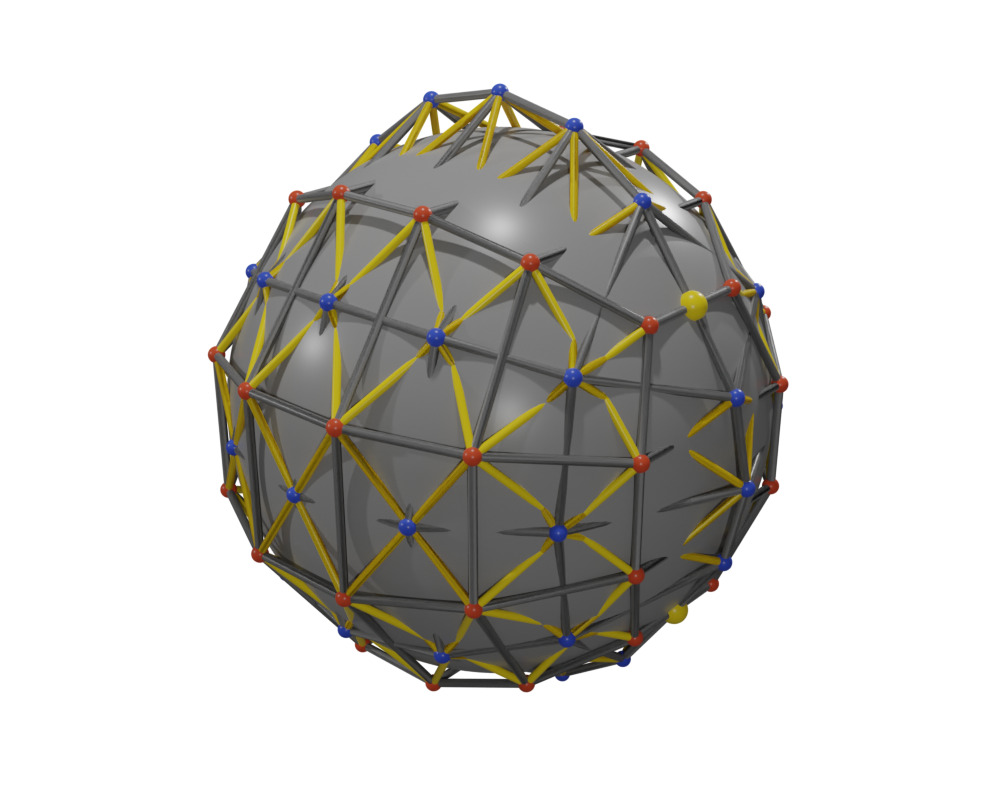}
  \includegraphics[width=0.24\textwidth]{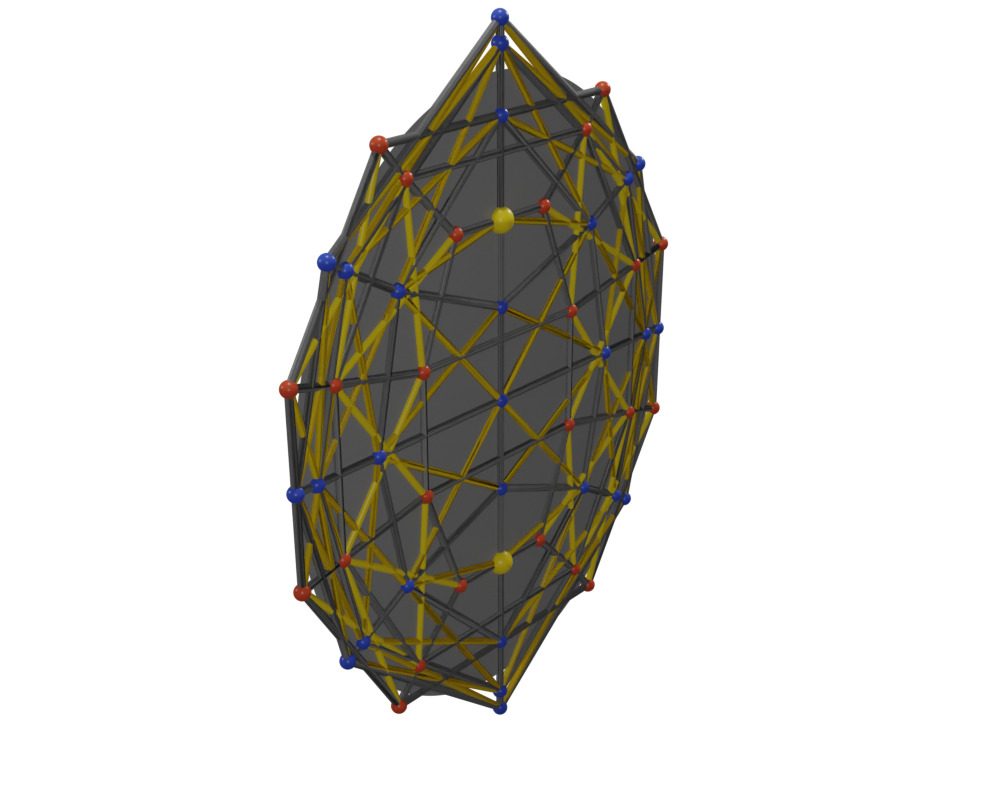}
  \caption{
    A one-parameter family of deformations of a discrete ellipsoid which is isometric along the discrete circular cross sections.
  }
  \label{fig:discrete-isometric-family}
\end{figure}
\begin{theorem}
  \label{thm:discrete-diagonal-curvature-circular2}
  Let $a > b > c > 0$ and $0 < \delta < \pi$.
  Then, the one-parameter family ($s_3 \in [0, \pi)$) of binets 
  \begin{equation}
    \label{eq:discrete-ellipsoid-circles-isometric}
    \begin{split}
      \br &= (x,y,z) : \Z^2 \cup \left(\Z^2\right)^* \supset U \cup U^* \rightarrow \R^3\\[0.5em]
      x(n_1,n_2) &= \frac{1}{\cos(\delta/2)} \frac{a-c}{\sqrt{(a-b)(b-c)}} \, \sin(\delta n_1)\cos(\delta n_2) \cos s_3 \\
      y(n_1,n_2) &= \frac{1}{\sqrt{(a-b)(b-c)}} g_1(n_1)g_2(n_2)\\
      z(n_1,n_2) &= \frac{1}{\cos(\delta/2)} \frac{a-c}{\sqrt{(a-b)(b-c)}} \, \cos(\delta n_1)\sin(\delta n_2) \sin s_3,
    \end{split}
  \end{equation}
  where $g_1$ and $g_2$ are determined by initial conditions and the recurrence relations
  \bela{F133}
  g_1(n_1^+) = \frac{a-b - f_1(n_1)^\two}{g_1(n_1)},\quad
  g_2(n_2^+) = \frac{f_2(n_2)^\two - a+b}{g_2(n_2)},
  \ela
  has the following properties (see Figure~\ref{fig:discrete-isometric-family}):
  \begin{enumerate}
  \item
    For each $s_3$, it constitutes a particular case of the class of discrete ellipsoids described in Theorem~\ref{thm:discrete-ellipsoid-curvature-line-parametrization} with
    \begin{equation}
      \alpha = \frac{b - c}{a - c}\cos^2 s_3,\quad
      \beta = 1,\quad
      \gamma = \frac{a - b}{a - c}\sin^2 s_3.
    \end{equation}
  \item
    For each $s_3$, the discrete curves
    \begin{equation}
      \label{eq:discrete-circular-cross-sections2}
      n_\pm = \frac{n_1 \pm n_2}{2} = \mbox{const}
    \end{equation}
    are planar, and they are discrete circles in the sense of Definition~\ref{def:discrete-circle}.
    
    For each $s_3$, the collection of discrete circles form two binets
    which are diagonal to $\br$ in the sense of Definition~\ref{def:discrete-diagonal}.
    
    For any two values of $s_3$,
    the two discrete circles $n_+ = \mbox{const}$ are congruent, and so are the two discrete circles $n_- = \mbox{const}$.
  \end{enumerate}
\end{theorem}

\begin{proof}
  As we have seen in the proof of Theorem~\ref{thm:discrete-diagonal-curvature-circular},
  the functions
  \begin{equation}
    \label{f-discrete-solutions-family}
    \begin{aligned}
      f_1(n_1) & = \sqrt{\frac{a-c}{\cos(\delta/2)}}\sin (\delta n_1),
      \quad &h_1(n_1) &= \sqrt{\frac{a-c}{\cos(\delta/2)}}\cos (\delta n_1)\\
      f_2(n_2) & = \sqrt{\frac{a-c}{\cos(\delta/2)}}\cos (\delta n_2),
      \quad &h_2(n_2) &= \sqrt{\frac{a-c}{\cos(\delta/2)}}\sin (\delta n_2)
    \end{aligned}
  \end{equation}
  satisfy \eqref{F9}$_{1,2}$.
  As in Theorem~\ref{thm:isometric-circles}, we choose
  \[
    \hat f_3(s_3) = \sqrt{\frac{a-c}{b-c}} \cos s_3,\qquad
    \hat h_3(s_3) = \sqrt{\frac{a-c}{a-b}} \sin s_3
  \]
  as solutions of \eqref{F9}$_{3}$.
  The functions $g_1, g_2$ in \eqref{F9}
  are obtained from the recurrence relations \eqref{F133}
  which determine them up to their signs and initial values.
  Thus, the one-parameter family of binets \eqref{eq:discrete-ellipsoid-circles-isometric}
  constitutes a particular case of \eqref{eq:discrete-ellipsoid}
  with the corresponding classical ellipsoids \eqref{eq:ellipsoid}, where
  \[
    \alpha = \hat f_3(s_3)^2, \quad
    \beta = 1, \quad
    \gamma = \hat h_3(s_3)^2,
  \]
  i.e., the one-parameter family of ellipsoids
  \begin{equation}
    \label{eq:deformation-family-again}
    \frac{b-c}{a-c}\frac{x^2}{\cos^2 s_3} + y^2 + \frac{a-b}{a-c}\frac{z^2}{\sin^2 s_3} = 1.
  \end{equation}
  Substituting the binet \eqref{F6}
  into the equations of the circular cross sections given by \eqref{eq:ellipsoid-circular-sections},
  we obtain
  \[
    \sqrt{\frac{\hat f_3(s_3)^2 - 1}{(a - b)\hat f_3(s_3)^2}} f_1(n_1)f_2(n_2)\hat f_3(s_3) \pm \sqrt{\frac{1 - \hat h_3(s_3)^2}{(b - c)\hat h_3(s_3)^2}} h_1(n_1)h_2(n_2)\hat h_3(s_3) = \sqrt{a - c}\,\mu_\pm,
  \]
  provided that $\hat f_3(s_3)^2 > 1$.
  The case  $\hat f_3(s_3)^2 < 1$ may be dealt with in a similar manner.
  Applying \eqref{E11-variation},
  this may be formulated as
  \begin{equation}
    \label{eq:ellipsoid-parametrization-into-planes3}
    f_1(s_1)f_2(s_2) \pm h_1(s_1)h_2(s_2) = \tilde\mu_\pm.
  \end{equation}
  The functions \eqref{f-discrete-solutions-family}, on the other hand, satisfy
  \[
    f_1(n_1)f_2(n_2) \pm h_1(n_1)h_2(n_2) = \cos(\delta/2)(a - c) \sin(n_1 \pm n_2)
  \]
  so that \eqref{eq:ellipsoid-parametrization-into-planes3} holds with
  \[
    \tilde\mu_\pm(n_{\pm}) = \cos(\delta/2)(a - c) \sin(2n_\pm).
  \]
  Thus, the parameter polygons of the diagonal binets lie in the planes of the circular cross sections
  of the family \eqref{eq:deformation-family-again}.
  In the same way as in the proof of Theorem~\ref{thm:discrete-diagonal-curvature-circular},
  it follows that these parameter polygons form discrete circles.
  By Lemma~\ref{lem:ellispoid-isometric-circles}, the affine transformations
  that relate the ellipsoids in the family \eqref{eq:deformation-family-again}
  are isometric on the planes of the circular cross sections.
  Hence, corresponding discrete circles are congruent.
\end{proof}

In summary, we are now in a position to state the main theorem,
which constitutes a discrete analogue of Theorem~\ref{thm:circular-cross-sections}.
\begin{theorem}
  \label{discretemaintheorem}
  The discrete ellipsoids \eqref{eq:discrete-ellipsoid-circles} exhibit the following properties.
  \begin{enumerate}
  \item
    \label{discretemaintheorem1}
    The parameter polygons $\frac{n_1 \pm n_2}{2} = \mathrm{const}$ are discrete circles,
    which lie in parallel planes (see Figure~\ref{fig:discrete-ellipsoid-diagonal}).
  \item
    \label{discretemaintheorem2}
    The planes $\plane(n_1,n_2)$ along the circular cross sections $\frac{n_1 \pm n_2}{2} = \mathrm{const}$
    meet at a point (see Figure~\ref{fig:discrete-tangent-cone}).
  \item
    The collection of discrete circles form two binets which are diagonal to the principal binet (see Figure~\ref{fig:discrete-ellipsoid-diagonal}).
  \item
    The two families of discrete circular cross sections
    may be continuously deformed in such a manner that the deformed discrete circles are congruent to the original discrete circles
    and remain the discrete circular cross sections of discrete ellipsoids (see Figure~\ref{fig:discrete-isometric-family}).

    Furthermore, any two discrete ellipsoids from this family of deformations are discrete confocal quadrics
    in the sense of \cite{BobenkoSchiefSurisTechter18} up to a scaling.
  \end{enumerate}
\end{theorem}
\begin{figure}
  \centering
  \includegraphics[width=0.42\textwidth]{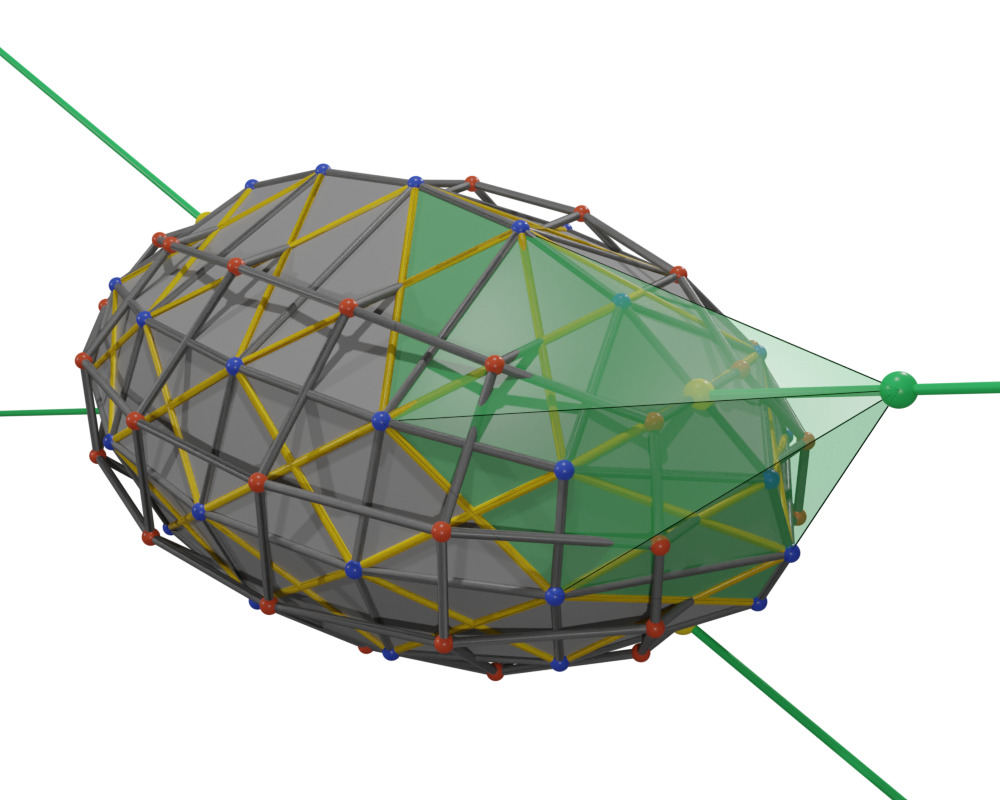}
  \hspace*{1cm}
  \includegraphics[width=0.42\textwidth]{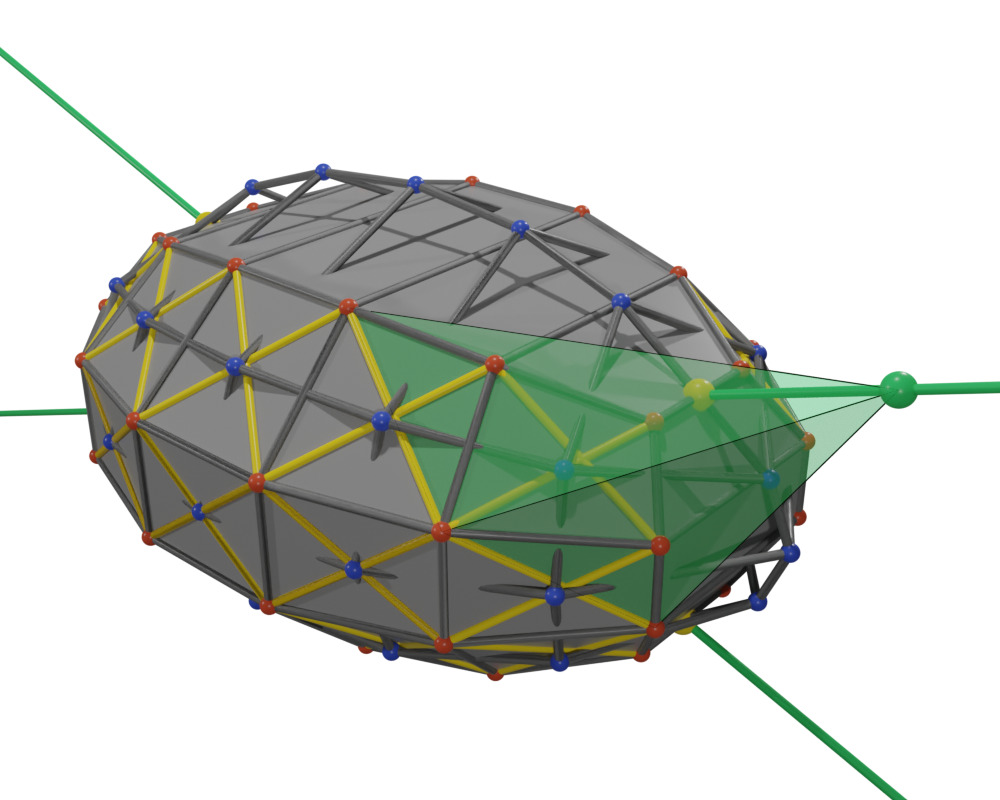}
  \caption{
    The planes along a discrete circular cross section of the binet representing a discrete ellipsoid meet at a point,
    which is the pole of the plane of the discrete circular cross section.
    Together, they form a ``discrete touching cone'', the apex of which lies on one of two lines passing through opposite ``discrete umbilic vertices''.
  }
  \label{fig:discrete-tangent-cone}
\end{figure}
\begin{proof}
  We yet have to prove \ref{discretemaintheorem2}.
  All other statements follow directly from Theorems~\ref{thm:discrete-diagonal-curvature-circular} and~\ref{thm:discrete-diagonal-curvature-circular2}.
  By \ref{discretemaintheorem1}, the parameter polygons $\frac{n_1 \pm n_2}{2} = \mathrm{const}$ are discrete circles.
  Thus, by Theorem~\ref{thm:discrete-ellipsoid-curvature-line-parametrization}~\ref{thm:discrete-ellipsoid-curvature-line-parametrization2},
  the planes $\plane(n_1,n_2)$ along a circular cross section $\frac{n_1 \pm n_2}{2} = \mathrm{const}$
  are the polar planes of the vertices of a discrete circle.
  Since the vertices of a discrete circle lie in a plane, the planes $\plane(n_1,n_2)$ intersect in the pole of that plane.
\end{proof}

\begin{remark}
  \label{rem:missing-umbilics}
  Theorem~\ref{thm:circular-cross-sections} contains some additional statements about the umbilic points of the ellipsoid.
  We discuss discrete analogues of umbilic points in Section~\ref{sec:umbilic} and formulate corresponding additional statements in Theorem~\ref{thm:discrete-umbilics}.
\end{remark}

\section{Boundary conditions and the combinatorics of discrete ellipsoids}
\label{sec:boundary-conditions}

\subsection{Boundary conditions}
In order to construct closed discrete ellipsoids, we specify explicitly suitable domains $U$ and $U^*$
and associated boundary conditions.

Let $N_1$ and $N_2$ be positive integers and
\bela{F26}
\begin{split}
  U &= \set{(n_1,n_2)\in\Z^2}{-N_i\leq n_i\leq N_i, ~i=1,2}\\
  U^* &= \set{(n_1,n_2)\in\left(\Z^2\right)^* }{ -N_i + \tfrac{1}{2} \leq n_i \leq N_i - \tfrac{1}{2}, ~i=1,2}.
\end{split}
\ela
From now on, we consider the entire family of isometric deformations of discrete ellipsoids
\bela{F27}
\begin{aligned}
  \br &= (x,y,z) : U\cup U^*\rightarrow \R^3\\
  x(n_1, n_2, s_3) &= \frac{f_1(n_1)f_2(n_2)\hat{f}_3(s_3)}{\sqrt{(a-b)(a-c)}}\\
  y(n_1, n_2, s_3) &= \frac{g_1(n_1)g_2(n_2)}{\sqrt{(a-b)(b-c)}}\\
  z(n_1, n_2, s_3) &= \frac{h_1(n_1)h_2(n_2)\hat{h}_3(s_3)}{\sqrt{(a-c)(b-c)}},
\end{aligned}
\ela
where
\bela{F29}
\begin{aligned}
  f_1(n_1) & = \sqrt{\frac{a-c}{\cos(\delta/2)}}\sin (\delta n_1),
  \quad &h_1(n_1) &= \sqrt{\frac{a-c}{\cos(\delta/2)}}\cos (\delta n_1)\\
  f_2(n_2) & = \sqrt{\frac{a-c}{\cos(\delta/2)}}\cos (\delta n_2),
  \quad &h_2(n_2) &= \sqrt{\frac{a-c}{\cos(\delta/2)}}\sin (\delta n_2).
\end{aligned}
\ela
with $a > b > c > 0$ and $0 < \delta < \pi$.
The functions $\hat{f}_3$ and $\hat{h}_3$ are constrained by
\bela{F32}
  (b-c)\hat{f}_3^2(s_3) + (a-b)\hat{h}_3^2(s_3) = a-c.
\ela
The functions
\[
  \begin{aligned}
    g_1 : &\set{n_1 \in \tfrac{1}{2}\Z}{ -N_1 \leq n_1 \leq N_1} \rightarrow \R\\
    g_2 : &\set{n_2 \in \tfrac{1}{2}\Z}{ -N_2 \leq n_2 \leq N_2} \rightarrow \R\\
  \end{aligned}
\]
are determined by initial conditions and the recurrence relations
\bela{F31}
\begin{aligned}
  g_1(n_1+\tfrac{1}{2}) &= \frac{a-b - f_1(n_1)f_1(n_1+\tfrac{1}{2})}{g_1(n_1)}\\
  g_2(n_2+\tfrac{1}{2}) &= \frac{f_2(n_2)f_2(n_2+\tfrac{1}{2}) - a+b}{g_2(n_2)}.
\end{aligned}
\ela
It is noted that, as in the continuous case, \eqref{F27}, \eqref{F31}
represents only one half of an ellipsoid.
In order to glue two halves together, we now examine suitable boundary conditions.
Indeed, a canonical choice is $y = 0$ on the boundary of $U$, corresponding to
$g_1(\pm N_1) = 0$, $g_2(\pm N_2) = 0$ (see Figure~\ref{fig:discrete-boundary-conditions}).
\begin{figure}
  \centering
  \includegraphics[width=0.42\textwidth]{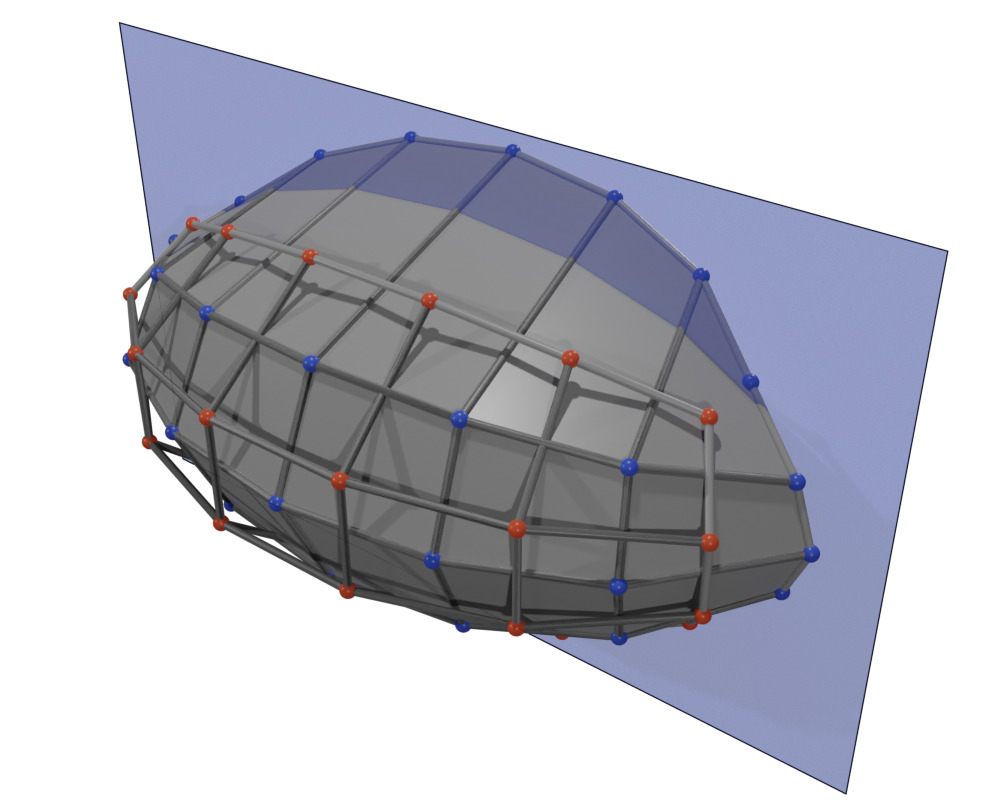}
  \hspace*{1cm}
  \includegraphics[width=0.42\textwidth]{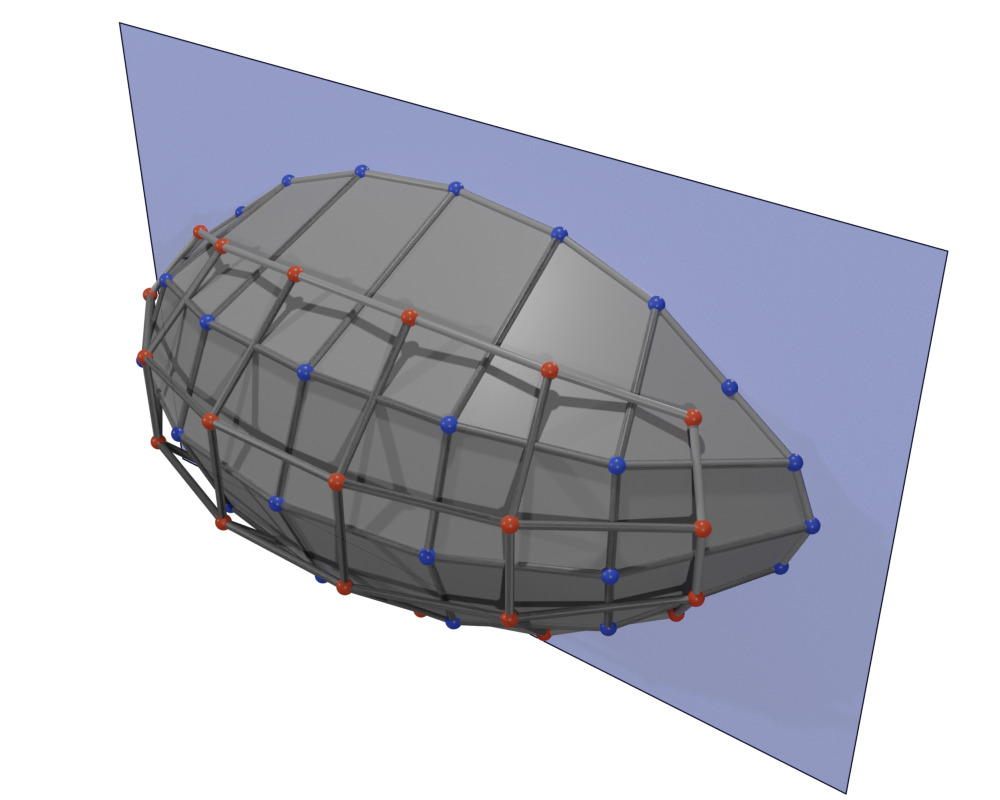}
  \caption{
    Discrete semi-ellipsoids together with the coordinate plane $y = 0$.
    \emph{Left:} A discrete semi-ellipsoid not obeying the boundary conditions.
    \emph{Right:} A discrete semi-ellipsoid satisfying the boundary conditions.
  }
  \label{fig:discrete-boundary-conditions}
\end{figure}
\begin{theorem}
  \label{thm:boundary-conditions}
  Let $g_1$ and $g_2$ be solutions of \eqref{F31}.
  Then, $g_1$ and $g_2$ satisfy the boundary conditions
  \bela{F16}
  g_1(\pm N_1) = 0,\qquad g_2(\pm N_2) = 0
  \ela
  and do not change their signs on their respective domains, if
  \bela{F33}
  \frac{a-2b+c}{a-c}
  = -\frac{\displaystyle\cos\frac{(4N_1-1)\pi}{4N_1+4N_2-2}}{\displaystyle\cos\frac{\pi}{4N_1+4N_2-2}}
  = \frac{\displaystyle\cos\frac{(4N_2-1)\pi}{4N_1+4N_2-2}}{\displaystyle\cos\frac{\pi}{4N_1+4N_2-2}}
  \ela
  and
  \bela{F24}
  \delta = \frac{\pi}{2N_1 + 2N_2 - 1}.
  \ela
\end{theorem}
\begin{proof}
  The condition that $g_1$ and $g_2$ do not change signs
  may be expressed as
  \bela{F17}
  \begin{aligned}
    a-b - f_1(n_1)f_1(n_1+\tfrac{1}{2}) > 0&\quad\mbox{ for}\quad -N_1+\tfrac{1}{2} \leq n_1 \leq N_1-1\\
    f_2(n_2)f_2(n_2+\tfrac{1}{2}) - a+b > 0&\quad\mbox{ for}\quad -N_2+\tfrac{1}{2} \leq n_2 \leq N_2-1.
  \end{aligned}
  \ela
  The condition $g_1(\pm N_1)=0$ is given by
  \bela{F18}
  a-b - \frac{a-c}{\cos(\delta/2)}\sin\delta(N_1-\tfrac{1}{2})\sin\delta N_1 = 0
  \ela
  On introduction of
  \[
    q \coloneqq \frac{a-2b+c}{a-c}
  \]
  and using the trigonometric identity
  $
    2\sin A \sin B = \cos(A-B) - \cos(A+B),
  $
  this may be formulated as
  \bela{F19}
  \cos\delta(2N_1-\tfrac{1}{2}) = -q\cos\frac{\delta}{2}.
  \ela
  Similarly, the requirement that $g_2(\pm N_2) = 0$ leads to
  \bela{F20}
  \cos\delta(2N_2-\tfrac{1}{2}) = q\cos\frac{\delta}{2},
  \ela
  and the inequalities \eqref{F17} may be written as
  \bela{F21}
  \begin{aligned}
    \cos\delta(2n_1+\tfrac{1}{2}) &> -q\cos\frac{\delta}{2} = \cos\delta(2N_1-\tfrac{1}{2})\\
    \cos\delta(2n_2+\tfrac{1}{2}) &> q\cos\frac{\delta}{2} = \cos\delta(2N_2-\tfrac{1}{2}).
  \end{aligned}
  \ela
  If we eliminate $q$ between \eqref{F19} and \eqref{F20} and use the trigonometric identity
  $
  2\cos A \cos B = \cos(A-B) + \cos(A+B),
  $
  we obtain
  \bela{F22}
  \cos\delta(N_1 + N_2 - \tfrac{1}{2})\cos\delta(N_1-N_2) = 0
  \ela
  so that $q$ is determined by either \eqref{F19} or \eqref{F20} once $\delta$ is known.
  If we make the choice
  \bela{F25}
  \delta = \frac{\pi}{2N_1 + 2N_2 - 1}
  \ela
  then \eqref{F22} holds and,
  due to the monotonicity of the cosine, it is evident that the inequalities \eqref{F21} are satisfied
  for the values of $n_1$ and $n_2$ stated in \eqref{F17}.
\end{proof}
\begin{remark}
  It turns out that all other solutions for $\delta > 0$ of
  \eqref{F22} lead to a violation of one of the inequalities
  \eqref{F21} if their extension to the continuous interval is
  postulated.  The latter constitutes a natural stronger assumption in
  view of the continuum limit.  Indeed, we first note that a priori,
  $N_1$ and $N_2$ do not seem to be on equal footing as can be seen
  from \eqref{F33}.  However, for the purpose of the argument below,
  which is only concerned with the inequalities \eqref{F21}, they are.
  Thus, without loss of generality, we assume that $N_1 \geq N_2$.
  The length of the interval
  $\mathcal{I} = [-N_1 + \tfrac{1}{2}, N_1-1]$ on which the cosine in
  the expression
  \[
    \cos\delta(2n_1 + \tfrac{1}{2})
  \]
  is evaluated is given by
  \[
    \Delta = \delta(2(N_1-1) + \tfrac{1}{2}) - \delta(2(-N_1+\tfrac{1}{2}) + \tfrac{1}{2})
    = \delta(4N_1 - 3).
  \]
  For the inequality
  \[
    \cos\delta(2n_1 + \tfrac{1}{2}) > \cos\delta(2N_1 - \tfrac{1}{2})
  \]
  to be satisfied on the continuous interval $\mathcal{I}$, it is necessary that $\Delta < 2\pi$.
  
  On the other hand, condition \eqref{F22} holds if and only if
  \[
    \delta = \frac{(2k+1)\pi}{2N_1 + 2N_2 - 1}
  \]
  for some $k \geq 0$ or
  \[
    \delta = \frac{(2k+1)\pi}{2N_1 - 2N_2}
  \]
  for some $k \geq 0$.
  In the first case, we find that $\Delta < 2\pi$ if and only if
  \[
    (2k+1)(4N_1-3) < 4N_1 + 4N_2 -2,
  \]
  which only holds for $k=0$.
  In the second case, it is required that
  \[
    (2k+1)(4N_1-3) < 4N_1 - 4N_2,
  \]
  which is not satisfied for any $k\geq 0$.
\end{remark}
\begin{remark}
  If we adopt the point of view that the positive integers $N_1$ and $N_2$ may be prescribed then the constraints \eqref{F19} and \eqref{F20} determine the quantities $\delta$ and $q$. This implies that $q$ can only attain a countable number of values in the open interval $(-1,1)$. The restrictions on possible values of $q$ will be explored in detail in Section~\ref{sec:discrete-vs-continuous}.
\end{remark}

\subsection{Initial conditions}

In the continuous case, the parametrization \eqref{eq:discrete-ellipsoids-parametrization-family}
of an ellipsoid $\mathcal{E}(s_3)$ may be regarded as a compound of the parametrization of half of an ellipsoid (semi-ellipsoid)
and its mirror image with respect to the plane $y=0$.
In the discrete setting, this symmetry may be maintained by choosing initial conditions
\[
  g_1(0) = \pm g_1^0,\qquad g_2(0) = g_2^0
\]
which define via \eqref{F31} a pair of discrete semi-ellipsoids denoted by
\[
  \begin{aligned}
    &\br^\pm : U \cup U^* \rightarrow \R^3\\
    &\br^+ \coloneqq \br|_{  g_1(0) = g_1^0, g_2(0) = g_2^0}\\
    &\br^- \coloneqq \br|_{  g_1(0) = -g_1^0, g_2(0) = g_2^0}.
  \end{aligned}
\]
A priori, $g_1^0$ and $g_2^0$ may be chosen arbitrarily, but a natural choice might be
\bela{F14}
g_1^0 = \sqrt{a-b},\quad g_2^0 = \sqrt{b-c}  
\ela
so that
\bela{F15}
\begin{pmatrix} x\\y\\z\end{pmatrix}(0,0,s_3) = \begin{pmatrix}0\\ \pm1\\ 0\end{pmatrix}
\ela
as in the classical continuous setting.
Even though the following considerations do not depend on $g_1^0$ and $g_2^0$,
the pictorial illustration of the results is based on the above choice.

\subsection{Combinatorics and topology}
\begin{figure}
  \begin{center}
    \includegraphics[width=0.8\textwidth]{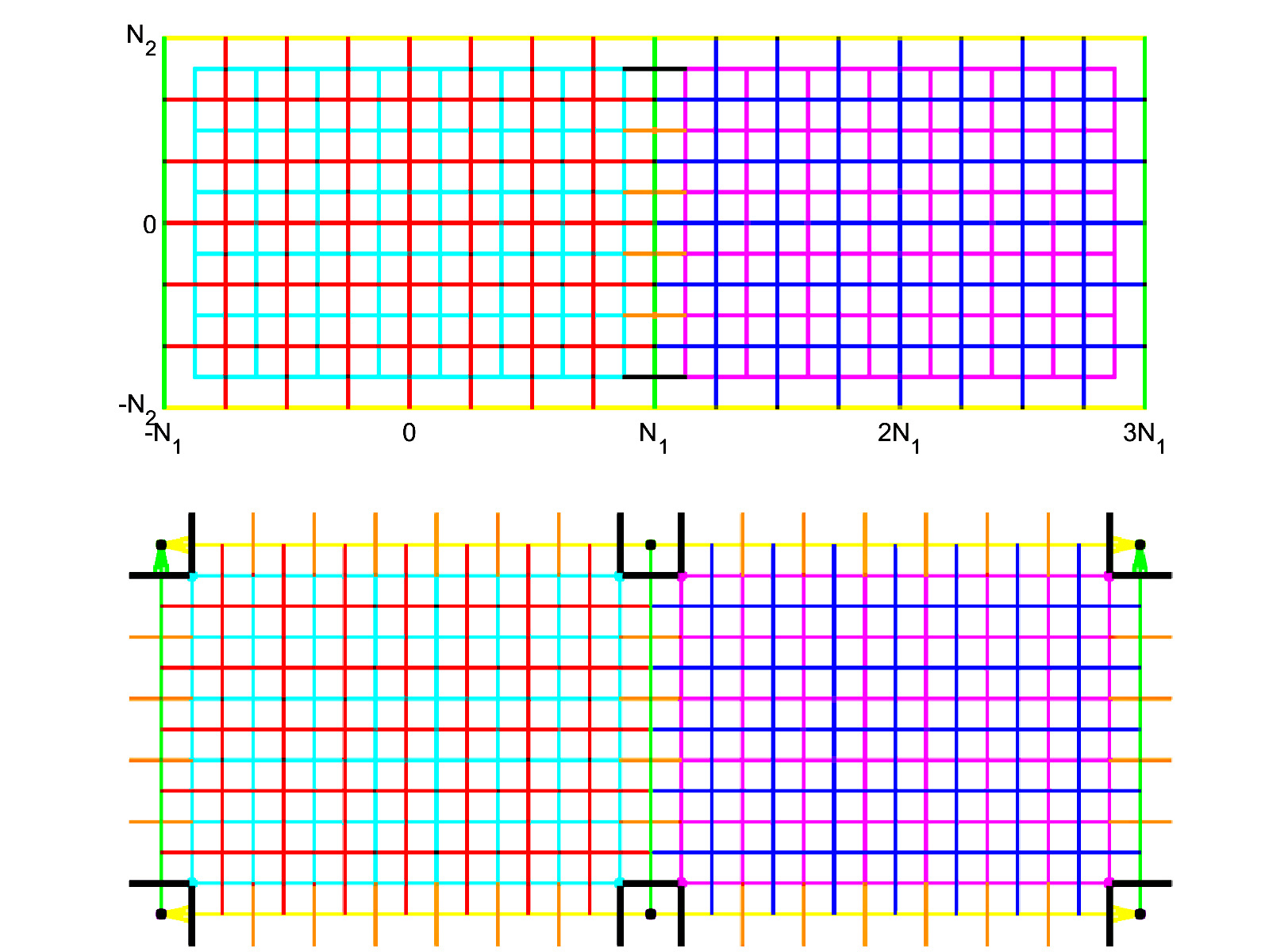}\\
    \vspace*{1em}
    \hspace*{1em}\includegraphics[width=0.7\textwidth]{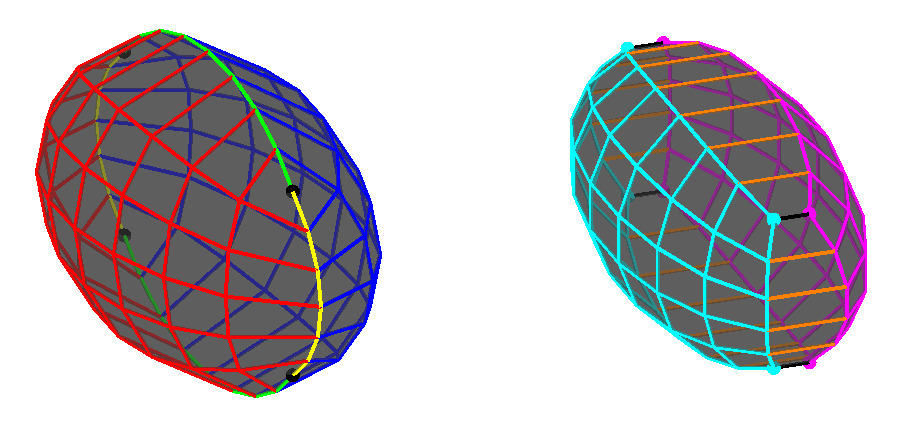}
  \end{center}
  \caption{
    \emph{Top:} The domain $V\cup V^*$ obtained by glueing the two copies of the domain $U$ along the edges on the green line $n_1=N_1$ and by connecting the two copies of the domain $U^*$ via the orange and black edges.
    \emph{Middle:} The domains $V$ and $V^*$ interpreted topologically as discrete spheres by folding $V\cup V^*$ along the line $n_1=N_1$, identifying corresponding yellow and green edges and adding orange and black edges.
    \emph{Bottom-left:} The discrete ellipsoid $\bR(V)$. The black vertices have valence 2 and constitute discrete analogues of umbilic points.
    \emph{Bottom-right:} The discrete dual ellipsoid $\bR(V^*)$. The cyan and magenta vertices have valence 3. The black edges constitute discrete analogues of umbilic points.
  }  
  \label{comb}
\end{figure}
Until now we referred to the binet \eqref{eq:discrete-ellipsoid} as a \emph{discrete ellipsoid}.
In the following it is more convenient to refer to the two polyhedral surfaces given by the restrictions to $U$ and $U^*$ as a \emph{discrete ellipsoid and its dual}.

In order to formalize the gluing process of the discrete semi-ellipsoids (and their duals) $\br^+(U)$, $\br^+(U^*)$, $\br^-(U)$, $\br^-(U^*)$,
we introduce the domains
\bela{F34}
 \begin{split}
  V &= \set{(n_1,n_2)\in\Z^2}{ -N_1\leq n_1\leq 3N_1,\, -N_2\leq n_2\leq N_2}\\
  V^* &= \set{(n_1,n_2)\in\left(\Z^2\right)^*}{ -N_1<n_1< 3N_1,\,-N_2< n_2<N_2}
 \end{split}
\ela
and combine the binets $\br^\pm$ according to
\bela{F35}
 \begin{split}
  \bR : V\cup V^*&\rightarrow\R^3\\
  \bR(n_1,n_2) &= \left\{\begin{aligned} &\br^+(n_1,n_2) & &\mbox{ if } n_1\leq N_1\\
                                                   &\br^-(2N_1-n_1,n_2) && \mbox{ if } n_1\geq N_1.\end{aligned}\right.
 \end{split}
\ela
The discrete semi-ellipsoids $\br^+(U)$ and $\br^-(U)$ are glued at the corresponding edges of their boundaries, while the dual discrete semi-ellipsoids $\br^+(U^*)$ and $\br^-(U^*)$ are glued by introducing additional edges connecting corresponding vertices of their boundaries. This is illustrated best by examining the combinatorics and topology of the pre-image of the pair of dual discrete ellipsoids. Firstly, the two copies of $U$ are linked at the edges of $V$ along the (green) line $n_1=N_1$ and the two copies of $U^*$ are linked by the (orange and black) edges of $V^*$ which connect the vertices $(N_1-\frac{1}{2},n_2)$ and $(N_1+\frac{1}{2},n_2)$ (see Figure \ref{comb}, top).
The second step is to fold the domain $V\cup V^*$ along the line $n_1=N_1$ and identify corresponding edges on the (yellow) line $n_2 = -N_2$, the (yellow) line $n_2=N_2$, and the (green) lines $n_1=-N_1$ and $n_1=3N_1$. In this manner, the toplogy of the domain $V$ may be regarded as that of a discrete sphere. Finally, the vertices $(n_1,-N_2+\frac{1}{2})$ and $(2N_1-n_1,-N_2+\frac{1}{2})$ are linked by (orange and black) edges and so are the vertices $(n_1,N_2-\frac{1}{2})$ and $(2N_1-n_1,N_2-\frac{1}{2})$ (see Figure \ref{comb}, middle).
Likewise, the vertices $(-N_1+\frac{1}{2},n_2)$ and $(3N_1-\frac{1}{2},n_2)$ are connected by (orange and black) edges so that, topologically, the domain $V^*$ also becomes a discrete sphere.
\begin{remark}
  It is emphasised that, for reasons of symmetry,
  the additional edges of the discrete ellipsoid $\bR(V^*)$ are also orthogonal to the corresponding dual edges of $\bR(V)$
  and the additional quadrilaterals generated by the completion of the set of edges of $\bR(V^*)$ are also planar.
  Hence, both discrete ellipsoids are entirely composed of planar quadrilaterals
  and all corresponding dual edges are mutually orthogonal. 
\end{remark}
\begin{remark}
  \label{rem:full-circles}
  By Theorem~\ref{discretemaintheorem}~\ref{discretemaintheorem1},
  the diagonal polygons of each of the discrete semi-ellipsoids constitute discrete semi-circles.
  While on each of the two parts of the domain $n_1 \leq N_1$ and $n_1 \geq N_1$,
  the circular cross sections are still given by the parameter polygons that satisfy $\tfrac{n_1 \pm n_2}{2} = \text{const}$,
  they do not fit together in this manner when crossing the line $n_1 = N_1$,
  or any other of the lines along which the semi-ellipsoids are glued.
  Instead, when crossing any such line, the diagonal direction needs to be switched to continue on the same discrete circle.
  On the dual discrete ellipsoid, one of the added edges (orange in Figure~\ref{comb}) may be added to the discrete circle before switching direction.
  In this way, two discrete semi-circles may be glued together to obtain one ``full'' (closed) discrete circle.
\end{remark}

\subsection{Umbilic vertices and edges}
\label{sec:umbilic}
As stated in Remark~\ref{rem:missing-umbilics}, Theorem~\ref{thm:circular-cross-sections} also highlights
the connection between the umbilic points on an ellipsoid and its circular cross sections.
In the discrete setting, we first observe that
even though the discrete ellipsoids constructed in the preceeding are entirely composed of planar quadrilaterals, there exist vertices which do not have valence 4. Specifically, any discrete ellipsoid $\bR(V)$ has four vertices of valence 2, namely the vertices $\bR(N_1,\pm N_2)$ and $\bR(-N_1,\pm N_2)= \bR(3N_1,\pm N_2)$ marked in black in Figure \ref{comb}, corresponding to the vertices where $g_1(n_1)$ and $g_2(n_2)$ vanish simultaneously. 
Since the vanishing of the functions $g_1$ and $g_2$ characterises the umbilic points in the continuous case, we may refer to the four vertices of valence 2 as ``umbilic vertices''. In the case of the discrete ellipsoids $\bR(V^*)$, there exist 8 vertices of valence 3. These are the images of the cyan and magenta vertices in Figure \ref{comb}, namely $\bR(N_1\pm\frac{1}{2},-N_2+\frac{1}{2})$, \mbox{$\bR(N_1\pm\frac{1}{2},N_2-\frac{1}{2})$}, $\bR(-N_1+\frac{1}{2},\pm N_2 \mp\frac{1}{2})$ and $\bR(3N_1-\frac{1}{2},\pm N_2 \mp\frac{1}{2})$. These vertices are pairwise linked by four (black) edges which are the images of four sets of identical edges displayed in Figure \ref{comb}. Those sets are the four black edges in the top-left and top-right corners of Figure~\ref{comb} which are identical due to the glueing process of the domain $V^*$, the four black edges in the bottom-left and bottom-right corners and the two triples of black edges in the vicinity of the line $n_1=N_1$. Accordingly, we may refer to those four edges as ``umbilic edges".
\begin{theorem}
  \label{thm:discrete-umbilics}
  The discrete ellipsoids \eqref{F35} exhibit the following properties (see Figure~\ref{fig:discrete-tangent-cone}).
  \begin{enumerate}
  \item
    The umbilic vertices and edges constitute degenerations of the discrete circular cross sections.
  \item
    The planes $\plane(n_1,n_2)$ along any discrete circular cross section
    intersect in a point which lies on one of the two lines passing through opposite umbilic vertices.
  \end{enumerate}
\end{theorem}

\begin{proof}\
  \begin{enumerate}
  \item
    According to Remark~\ref{rem:full-circles},
    the discrete circular cross sections are given by closed diagonal polygons in $V \cup V^*$.
    Thus, we may think of the set of discrete circular cross sections as the union of two finite sequences of diagonal polygons.
    Each of these sequences starts and ends in a discrete circle consisting of an umbilic vertex and an umbilic edge.
  \item
    By Theorem~\ref{discretemaintheorem}~\ref{discretemaintheorem2},
    the planes $\plane(n_1,n_2)$ along a discrete circular cross section intersect in a point.
    This point is the pole of the plane of the discrete circular cross section,
    which coincides with the plane of a circular cross section of the underlying classical ellipsoid \eqref{eq:ellipsoid} with
    \[
      \alpha = \hat f_3(s_3)^2, \quad
      \beta = 1, \quad
      \gamma = \hat h_3(s_3)^2.
    \]
    By Theorem~\ref{thm:circular-cross-sections}~\ref{thm:circular-cross-sections2}, this point lies on one of the two lines passing through opposite umbilic points of the classical ellipsoid.
    Thus, it remains to show that these two lines coincide with the two lines passing through opposite discrete umbilic vertices,
    or equivalently, that the polar planes of the discrete umbilic vertices contain circular cross sections of the classical ellipsoid.

    Now, the four discrete umbilic vertices are given by
    \[
      \begin{pmatrix}
        \pm\frac{1}{\cos(\delta/2)}\sqrt{\frac{a-c}{a-b}}\hat{f}_3(s_3)\sin \delta N_1 \cos \delta N_2\\
        0\\
        \pm\frac{1}{\cos(\delta/2)}\sqrt{\frac{a-c}{b-c}}\hat{h}_3(s_3)\cos \delta N_1 \sin \delta N_2
      \end{pmatrix}
    \]
    and thus their polar planes are
    \begin{equation}
      \label{eq:discrete-umbilic-polar-planes}
      \frac{\sin \delta N_1 \cos \delta N_2}{\sqrt{a-b}} \frac{x}{\hat{f}_3}
      \pm \frac{\cos \delta N_1 \sin \delta N_2}{\sqrt{b-c}} \frac{z}{\hat{h}_3}
      = \pm \frac{\cos(\delta/2)}{\sqrt{a-c}}.
    \end{equation}
    On the other hand, by Proposition~\ref{prop:ellipsoid-circular-sections}, the families of planes of the circular cross sections of the classical ellipsoid
    are given by
    \[
      \sqrt{\hat{f}_3^2 - 1} \frac{x}{\hat{f}_3} \pm \sqrt{1 - \hat{h}_3^2} \frac{z}{\hat{h}_3} = \mu_\pm
    \]
    or, by applying \eqref{E11-variation}, by
    \[
      \sqrt{a-b} \frac{x}{\hat{f}_3} \pm \sqrt{b-c} \frac{z}{\hat{h}_3} = \tilde\mu_\pm.
    \]
    Hence, the four planes \eqref{eq:discrete-umbilic-polar-planes} belong to these two families of planes if and only if
    \[
      \det
      \begin{pmatrix}
        \frac{\sin \delta N_1 \cos \delta N_2}{\sqrt{a-b}} & \sqrt{a-b}\\
        \frac{\cos \delta N_1 \sin \delta N_2}{\sqrt{b-c}} & \sqrt{b-c}
      \end{pmatrix}
      = 0,
    \]
    which is equivalent to
    \[
      \frac{a-b}{b-c} = \frac{\tan\delta N_1}{\tan\delta N_2}.
    \]
    To prove this identity,
    we reformulate the relation between the ratios of differences of the constants $a,b,c$ and the positive integers $N_1, N_2$
    given by \eqref{F19} and \eqref{F20} according to
    \bela{AA4}
    \frac{a-b}{b-c} = \frac{1+q}{1-q} = \frac{\cos\delta/2 + \cos\delta(2N_2-1)}{\cos\delta/2 - \cos\delta(2N_2-1)}
    =\frac{\cos\delta N_2\cos\delta(N_2-\frac{1}{2})}{\sin\delta N_2\sin\delta(N_2-\frac{1}{2})}
    \ela
    and apply the definition of $\delta$ formulated as
    \bela{AA5}
    \delta\left(N_2-\frac{1}{2}\right) = \frac{\pi}{2} - \delta N_1.
    \ela
  \end{enumerate}
\end{proof}

\begin{remark}
  A closer examination of the properties of the vertices which do not have valence~4 in the context of the original unscaled discrete confocal ellipsoids parametrized by \eqref{F4}, \eqref{F5} confirms the validity of the terminology introduced above.
  Specifically, we consider the dependence on $s_3$ of, for instance, the umbilic vertex $\bR(N_1,N_2)$ and the vertex $\bR(N_1^-,N_2^-)$ incident to the corresponding umbilic edge. Since $g_1(N_1)=g_2(N_2)=0$, the functional equations \eqref{F5}$_{1,2,3,4}$ imply that
  \bela{F36}
  \begin{aligned}
    f_1^\two(N_1^-) &= f_2^\two(N_2^-) = a-b\\
    h_1^\two(N_1^-) &= h_2^\two(N_2^-) = b-c.  
  \end{aligned}
  \ela
  Hence, on use of \eqref{F5}$_{5,6}$, ``squaring'' $x$ and $z$ as given by \eqref{F4} results in
  \bela{F37}
  \frac{x(N_1,N_2,s_3)x(N_1^-,N_2^-,s_3)}{a-b} -
  \frac{z(N_1,N_2,s_3)z(N_1^-,N_2^-,s_3)}{b-c} = 1,
  \quad
  y(N_1,N_2,s_3)y(N_1^-,N_2^-,s_3) = 0,
  \ela
  which can be regarded as an analogue of the classical focal hyperbola given by
  \[
    \frac{x^2}{a - b} + \frac{z^2}{b - c} = 1, \quad y = 0.
  \]
  Indeed, it is known that this hyperbola constitutes the locus of all umbilic vertices
  on the family of confocal ellipsoids \eqref{E1}$_3$.
\end{remark}

\section{Discrete vs classical ellipsoids}
\label{sec:discrete-vs-continuous}

\subsection{Closed discrete ellipsoids}

It is evident that the ratio of any two differences of the constants $a>b>c$ completely characterises the associated one-parameter family of continuous ellipsoids.
In the discrete setting, the closed discrete ellipsoids considered above also depend on the choice of the positive integers $N_1$ and $N_2$.
However, given $N_1$ and $N_2$, only one difference of the constants $a>b>c$ may be prescribed arbitrarily.
Any other difference is determined by the relation \eqref{F33}, that is,
\bela{F38}
   q := \frac{a-2b+c}{a-c} =  \frac{\displaystyle\cos\frac{(4N_2-1)\pi}{4N_1+4N_2-2}}{\displaystyle\cos\frac{\pi}{4N_1+4N_2-2}}.
\ela
Since
\bela{F39}
  q - 1 = - 2\frac{b-c}{a-c} <0,\quad q + 1 = 2\frac{a-b}{a-c} > 0,
\ela
it is concluded that $q \in (-1,1)$.
In fact, the following statement may be made.
It may be regarded as a first measure of the quality of the discretization scheme proposed here.
\begin{theorem}
  \label{thm:boundary-condition-quality}
  Any given number $q\in(-1,1)$ may be approximated arbitrarily well by
  the right-hand side of \eqref{F38} by
  choosing suitable positive integers $N_1$ and $N_2$.
\end{theorem}
\begin{proof}
  We introduce the sequence
  \bela{F40}
  S_N(n) = \frac{\displaystyle\cos\frac{(4n-1)\pi}{4N-2}}{\displaystyle\cos\frac{\pi}{4N-2}},\quad 0\leq n\leq N
  \ela
  for any integer $N\geq2$ so that the set of values of $q$ generated by the right-hand side of \eqref{F38} is given by
  \bela{F41}
  Q := \{S_{N}(n): 0 < n < N,\, N\geq 2\}.
  \ela
  We now observe that 
  \bela{F42}
  S_N(0) = 1,\quad S_N(N)=-1 
  \ela
  and that the sequence $S_N(n)$ is strictly decreasing. In fact, since
  \bela{F43}
  S_N(n) - S_N(n+1) = 4\sin\frac{(4n+1)\pi}{4N-2}\sin\frac{\pi}{4N-2},\quad 0\leq n<N,
  \ela
  we conclude that
  \bela{F44}
  0 < S_N(n) - S_N(n+1) < \frac{4\pi}{4N-2}.
  \ela
\end{proof}

\subsection{Closed samplings of webs on classical ellipsoids}

In the classical case, we can find a continuous web of diagonally related nets from lines of curvature and circular cross sections on any ellipsoid (arbitrary $a > b > c$).
However, if we want to construct a closed sampling of this web with a prescribed number of lines of curvatures in each direction, say $2M_1$ and $2M_2$, this creates an additional closing condition (cf.\ Figure~\ref{fig:smooth-vs-discrete}, left).
\begin{theorem}
  \label{thm:closing-condition-web}
  A closed sampling of a web of diagonally related nets of lines of curvature and circular cross sections on an ellipsoid
  which consists of $2M_1$ and $2M_2$ lines of curvature respectively exists if and only if
  \begin{equation}
    \label{eq:web-closing-condition}
    \sqrt{\frac{a - b}{a - c}} = \cos \frac{M_2 \pi}{2 (M_1 + M_2)}
  \end{equation}
\end{theorem}
\begin{proof}
  Consider the curvature line parametrization \eqref{eq:discrete-ellipsoids-parametrization-family} of one half of an ellipsoid,
  which exhibits the property that the curves $s_1 \pm s_2 = \text{const}$ are the circular cross sections.
  This parametrization has a rectangular domain of lengths $2s_1^0$ and $2s_2^0$.
  To create a closed sampling of the web of the two diagonally related nets,
  we have to use the same constant step size in both directions.
  For a sampling on the whole ellipsoid consisting of $2M_2$ lines of curvature in $s_1$-direction and $2M_1$ lines of curvature in $s_2$-direction, the following condition must be satisfied:
  \[
    \frac{2s_1^0}{M_1} = \frac{2s_2^0}{M_2},
  \]
  or, equivalently,
  \[
    2M_2\sin^{-1} \sqrt{\frac{a-b}{a-c}} =  2M_1\cos^{-1} \sqrt{\frac{a-b}{a-c}}.
  \]
  Using $\sin^{-1}x = \frac{\pi}{2} - \cos^{-1}x$, this yields \eqref{eq:web-closing-condition}.
\end{proof}
Condition \eqref{F33} may be viewed as a discrete analogue of \eqref{eq:web-closing-condition}.
In terms of $q$, the condition \eqref{eq:web-closing-condition} may be written as
\begin{equation}
  \label{eq:web-closing-condition-var}
  q = 2 \frac{a - b}{a - c} - 1 = \cos \frac{M_2 \pi}{M_1 + M_2}.
\end{equation}
To obtain the same number of lines in the sampling of the web as the number of discrete parameter lines on the discrete ellipsoid,
we set $M_1 = 2N_1$ and $M_2 = 2N_2$ (cf.\ Figure~\ref{fig:smooth-vs-discrete}).
With that we see that \eqref{eq:web-closing-condition-var} behaves asymptotically ($N_1, N_2 \rightarrow \infty$) as \eqref{F33}.

\appendix

\section{Proof of Proposition~\ref{prop:ellipsoid-circular-sections}}
\label{sec:appendix-proof}
We embed $\R^3$ into the real projective space $\RP^3$ and further into the complex projective space $\CP^3$.
We introduce homogeneous coordinates $x_1, x_2, x_3, x_4$ by $x = \frac{x_1}{x_4}, y = \frac{x_2}{x_4}, z = \frac{x_3}{x_4}$ and identify the plane $x_4 = 0$ with the plane at infinity.
Recall that a conic in a plane is a circle if and only if it intersects the absolute (imaginary) conic
\[
  \mathcal{Z} : \quad x_1^2 + x_2^2 + x_3^2 = 0, \quad x_4 = 0
\]
in two points.
The four points of intersection of $\mathcal{Z}$ and the ellipsoid
\[
  \mathcal{E} : \quad \frac{x^2}{\alpha} + \frac{y^2}{\beta} + \frac{z^2}{\gamma} = 1
\]
are given by
\[
  P_{\sigma, \tau} =
  \left[
    \begin{matrix}
      \sigma \sqrt{\tfrac{1}{\gamma} - \tfrac{1}{\beta}}\\
      \tau i \sqrt{\tfrac{1}{\gamma} - \tfrac{1}{\alpha}}\\
      \sqrt{\tfrac{1}{\beta} - \tfrac{1}{\alpha}}\\
      0
    \end{matrix}
  \right],\qquad
  \sigma, \tau \in \{+,-\}
\]
and come in two complex conjugate pairs
$
P_{+,+} = \widebar{P}_{+,-}, \quad P_{-,+} = \widebar{P}_{-,-}.
$
Thus, they span two real lines at infinity
\[
  \ell_\sigma
  = \rm{span}
  \left\{
    \left(
      \begin{matrix}
        \sigma \sqrt{\tfrac{1}{\gamma} - \tfrac{1}{\beta}}\\
        0\\
        \sqrt{\tfrac{1}{\beta} - \tfrac{1}{\alpha}}\\
        0
      \end{matrix}
    \right),
    \left(
      \begin{matrix}
        0\\
        1\\
        0\\
        0
      \end{matrix}
    \right)
  \right\},\qquad
  \sigma, \in \{+,-\}.
\]
The two one-parameter families of planes that contain $\ell_\sigma$ respectively are given by
\[
  \sqrt{\tfrac{1}{\beta} - \tfrac{1}{\alpha}} \, x_1 - \sigma \sqrt{\tfrac{1}{\gamma} - \tfrac{1}{\beta}} \, x_3 - \mu_\sigma x_4 = 0
\]
for $\mu_\sigma \in \R \cup \{\infty\}$.
Since the two lines $\ell_\sigma$ are at infinity, the two families consist of parallel planes.
One easily checks that the intersection of these planes with $\mathcal{E}$ is non-empty
for the values of $\mu_\sigma$ given in \eqref{eq:ellipsoid-circular-sections}.
In fact, by construction, these intersections constitute all circles contained in $\mathcal{E}$.
Furthermore, one easily verifies that, for the boundary values of $\mu_\pm$, the poles of the (tangent) planes
coincide with the umbilic points of $\mathcal{E}$, which are given by
\[
  \begin{pmatrix}
    \pm\sqrt{\frac{\alpha(\alpha - \beta)}{\alpha - \gamma}}\\
    0\\
    \pm\sqrt{\frac{\gamma(\beta - \gamma)}{\alpha - \gamma}}
  \end{pmatrix}.
  \eqno\qed
\]

\printbibliography

\end{document}